\documentclass[10pt]{amsart}
\usepackage{version,amsmath, amssymb}
\usepackage{color}
\usepackage{amsthm, amsopn, amsfonts, xypic}
\usepackage[colorlinks=true]{hyperref}
\usepackage{mathrsfs}
\hypersetup{urlcolor=blue, citecolor=blue}

\usepackage[color=yellow]{todonotes}

\let\arXiv\arxiv

\def\doi#1{ {\href{http://dx.doi.org/#1}
   {{\mdseries\ttfamily DOI}}}}

\renewcommand{\tt}{ \tilde t}

\newcommand{\at}{{\tilde t}}

\newcommand{\tq}{{\tilde q}}
\newcommand{\tf}{{\tilde f}}
\newcommand{\td}{{\tilde \delta}}

\newcommand{\ax}{{\tilde x}}

\newcommand{\lL}{{\underline L}}

\newcommand{\newsection}[1]
{\section{#1}\setcounter{theorem}{0} \setcounter{equation}{0}
\par\noindent}

\newtheorem{theorem}{Theorem}

\newtheorem{lemma}[theorem]{Lemma}

\newcommand{\R}{{\mathbb R}}

\newcommand{\ang}{{\not\negmedspace\partial }}

\newcommand{\la}{\langle}
\newcommand{\ra}{\rangle}

\newcommand{\rs}{{r^*}}

\renewcommand{\S}{{\mathbb S}}
\newcommand{\M}{\mathcal M}

\newcommand{\tJ}{{\tilde{J}}}

\newcommand{\tY}{{\tilde{Y}}}
\newcommand{\tN}{{\tilde{N}}}

\renewcommand{\th}{{\tilde{h}}}
\newcommand{\uL}{\underline{L}}

\newcommand{\weight}{{\Bigl(1-\frac{2M}{r}\Bigr)}}

\begin{document}

\title
{
Global existence for quasilinear wave equations close to Schwarzschild
}

\author{Hans Lindblad}
\address{Department of Mathematics, Johns Hopkins University, Baltimore,
  MD  21218}
\email{lindblad@math.jhu.edu}
\author{Mihai Tohaneanu}
\address{Department of Mathematics, University of Kentucky, Lexington,
  KY  40506}
\email{mihai.tohaneanu@uky.edu}

\begin{abstract}
  In this article we study the quasilinear wave equation $\Box_{g(u, t, x)} u = 0$ where the metric $g(u, t, x)$ is close to the Schwarzschild metric. Under suitable assumptions of the metric coefficients, and assuming that the initial data for $u$ is small enough, we prove global existence of the solution. The main technical result of the paper is a local energy estimate for the linear wave equation on metrics with slow decay to the Schwarzschild metric.
 \end{abstract}

\maketitle

\newsection{Introduction}
The goal of this paper is to understand the global behavior of solutions to the quasilinear wave equation
\begin{equation}\label{maineeqn}
\Box_{g(u, t, x)} u = 0, \qquad u |_{\tt=0} = u_0, \qquad \tilde T u |_{\tt=0} = u_1
\end{equation}
Here $\Box_g$ denotes the d'Alembertian with respect to a Lorentzian metric $g$, which is equal the Schwarzschild metric when $u\equiv 0$, and
$\tilde T$ is a smooth, everywhere timelike vector field that equals
$\partial_t$ away from the black hole.  The coordinate $\tt$
is chosen so that the slice $\tt=0$ is space-like and so $\tt=t$
away from the black hole.

 Our motivation comes from the black hole stability problem, which roughly speaking asserts that solutions to Einstein's Equations with initial data that start close to a Kerr solution have domains of outer communication that will converge towards a (possibly different) Kerr solution. The related problem for Minkowski spacetimes has been settled in two different ways by Christodoulou -Klainerman \cite{CK} and Lindblad - Rodnianski \cite{LR1, LR2}. The second approach uses wave coordinates in which the metric satisfies
 \begin{equation}\label{wavecoord}
 \partial_\alpha \big( | g |^{ 1 / 2 } g^{ \alpha \beta } \big) = 0,
 ,\qquad
|g|=|\det{g}|
 \end{equation}
 and in which Einstein's Equations become  a system of quasilinear wave equations
\begin{equation} \label{waveeinstein}
\widetilde{\Box}_g g_{\alpha\beta} = F_{\alpha\beta}(g)[\partial g, \partial g].
\end{equation}
 A model of the system above was considered in \cite{Lin2,A2,Lin1}, where it was shown that the equation
\begin{equation}\label{model}
 \widetilde {\Box}_{g(u)} u = 0, \qquad \widetilde {\Box}_{g} = g^{\alpha\beta}\partial_{\alpha}\partial_{\beta}
\end{equation}
has global solutions for small initial data, assuming that $g(0)$ is the Minkowski metric. However the characteristics for \eqref{waveeinstein} are much less divergent than those for \eqref{model}, as the wave coordinate condition \eqref{wavecoord} forces the components of the metric that determine the main behavior to be close to those of Schwarzschild \cite{LR2,Lin3}.

 As a toy model of the system near Schwarzschild black holes, one can look at the quasilinear wave equation \eqref{maineeqn}, where the metric $g$ is given by
\begin{equation}\label{quasmetric}
g^{\alpha\beta} = g_S^{\alpha\beta} + H^{\alpha\beta}(t,x)u + O(u^2)
\end{equation}
with $g_S$ denoting the Schwarzschild metric and $H^{\alpha\beta}$ being smooth functions. Ideally we would like to assume only that $H^{\alpha\beta}$, as well as certain vector fields applied to $H^{\alpha\beta}$, are bounded functions. However, already in the close to Minkowski case \eqref{model} with constant $H$ the solution is badly behaved and to avoid this we will impose conditions on $H^{\alpha\beta}$ that resemble the wave coordinate condition \eqref{wavecoord}. These conditions will in particular
imply that components of $H^{\alpha\beta}$ corresponding to coefficient in
the wave operator of derivatives transversal to the outgoing light cones
must decay. We refer the reader to Section 3 for the exact conditions on $H^{\alpha\beta}$.
The main theorem of our paper is the following:

\begin{theorem}\label{mainthmvague}
Assume that the metric $g$ is like in \eqref{quasmetric}, and satisfies a couple of extra conditions near the photonsphere (see Section 3 and the discussion of why such conditions are needed). Then there exists a global classical solution to \eqref{maineeqn}, provided that the initial data is smooth, compactly supported and small enough.
\end{theorem}

 We will provide a more precise version in Section 3, Theorem \ref{mainthm}, after more notation is introduced.
The main new technical ingredient of the paper is a local energy estimate for the linear wave equation $\Box_g u = f$ for a metric $g$ close to the Schwarzschild metric. Due to the presence of the trapped null geodesics at the photonsphere $r=3M$, local energy estimates are difficult to establish even for small perturbations of the metric. However, for metrics converging to the Schwarzschild metric at a rate of $t^{-1-}$ (i.e. $t^{-1-\epsilon}$ for any $\epsilon>0$),  near the photonsphere, one can prove local energy estimates perturbatively (see, for example, \cite{MTT}). We are able to prove such an estimate by assuming a slower rate of convergence of only $t^{-1/2}$. We refer the reader to Section 4 for the precise statement of the result. We hope to prove a similar linear local energy estimate for perturbations of the Kerr metric using a combination of the methods here and modifications of methods in \cite{TT}, leading to the same nonlinear result.
 It is interesting that the required decay is the best one can hope to prove from just energy and local energy bounds. Using energies with growing weights one may be able to prove some additional decay under stronger assumptions. However for equations with nonlinear terms that only satisfy the weak null condition such as Einstein's equations in wave coordinates, the best one can hope for is $t^{-1}$ decay, see \cite{Lin3}, so our improvement of \cite{MTT} is needed even if one proves more decay.

 The linear wave equation $\Box_g u = f$ on Schwarzschild and Kerr manifolds has been studied extensively. Local energy estimates, Strichartz estimates and sharp pointwise decay rates of solutions are by now well-understood (see \cite{BS1},
\cite{BS2}, \cite{BSter}, \cite{DR1}, \cite{DR2}, \cite{MMTT}, \cite{Luk1}, \cite{DSS} for
Schwarzschild, \cite{TT}, \cite{DaRoNotes}, \cite{AB09}, \cite{To}, \cite{Luk2}, \cite{Tat}, \cite{MTT} for Kerr with
small angular momentum,  \cite{DaRoNew}, \cite{DaRoNew2}, \cite{DaRoNew3} for
Kerr with $|a| \!< \!M$, and \cite{Aret} for $|a| \!=\! M$).

Global existence for semilinear equations of the form $\Box_g u = u^p$ with small initial data was shown in \cite{DR3} ($p>4$ with radial data), \cite{BSter} ($p>3$) in the Schwarzschild case, and
\cite{LMSTW} for $p>1+\sqrt{2}$ for Kerr with $|a|\ll M$. For semilinear wave equations with a null condition on Kerr with small angular momentum, global existence was shown in \cite{Luk3,IK}. For quasilinear wave equations with small initial data, global existence was shown for time dependent metrics close to Minkowski in \cite{Yang1}, \cite{Yang2}. whereas in the asymptotically Kerr- de Sitter case a similar result to ours was proved in \cite{HV}.

 The paper is structured as follows. In Section 2 we introduce local energy estimates in Minkowski and Schwarzschild spacetimes. In Section 3 we give the statement of our main theorem, and present the bootstrap argument that finishes the proof. Section 4 contains our main linear estimate for perturbations of the Schwarzschild metric. Section 5 deals with commuting the equation with vector fields. Section 6 gives pointwise decay from local energy estimates. Section 7 is a refinement of Section 5 for lower order energies.

\newsection{Local energy estimates on Minkowski and Schwarzschild backgrounds}
We use $(t=x_0, x)$ for the coordinates in $\R^{1+3}$. We use Latin
indices $i,j=1,2,3$ for spatial summation and Greek indices $\alpha,\beta=0,1,2,3$ for space-time summation. In $\R^3$ we also use
polar coordinates $x = r \omega$ with $\omega \in \S^2$.  By $\la r\ra$ we denote a smooth radial function which agrees with $r$ for
large $r$ and satisfies $\la r \ra \geq 1$.  We consider a partition
of $ \R^{3}$ into the dyadic sets $A_R= \{\la r \ra \approx R\}$ for
$R \geq 1$, with the obvious change for $R=1$. We will use the notation $A\lesssim B$ to mean that there is a constant $C$ independent of $u$ and $\epsilon$ so that $A\leq CB$; the value of $C$ might change from line to line.

We introduce the local energy norm $LE$
\begin{equation}
 \| u\|_{LE} = {\sup}_R \, \| \la r\ra^{-\frac12} u\|_{L^2 (\R \times A_R)}  ,\qquad
 \| u\|_{LE[t_0, t_1]} ={\sup}_R \, \| \la r\ra^{-\frac12} u\|_{L^2 ([t_0, t_1] \times A_R)},
\label{ledef}\end{equation}
its $H^1$ counterpart
\begin{equation}
  \| u\|_{LE^1} = \| \nabla u\|_{LE} + \| \la r\ra^{-1} u\|_{LE},\qquad
 \| u\|_{LE^1[t_0, t_1]} = \| \nabla u\|_{LE[t_0, t_1]} + \| \la r\ra^{-1} u\|_{LE[t_0, t_1]},
\end{equation}
as well as the dual norm
\begin{equation}
 \| f\|_{LE^*} = {\sum}_R  \| \la r\ra^{\frac12} f\|_{L^2 (\R \times A_R)},\qquad
 \| f\|_{LE^*[t_0, t_1]} = {\sum}_R  \| \la r\ra^{\frac12} f\|_{L^2 ([t_0, t_1] \times A_R)}.
\label{lesdef}\end{equation}

Then we have the following scale invariant local energy estimate on Minkowski backgrounds:
\begin{equation}\label{localenergyflat}
\|\nabla u\|_{L^{\infty}_t L^2_x} + \| u\|_{LE^1}
 \lesssim \|\nabla u(0)\|_{L^2} + \|\Box u\|_{LE^*+L^1_t L^2_x}
\end{equation}
and a similar estimate involving the $LE^1[t_0, t_1]$ and $LE^*[t_0, t_1]$ norms.
This is proved using a small variation of Morawetz's method, with
multipliers of the form $a(r) \partial_r + b(r)$ where $a$ is
positive, bounded and increasing.
There are many similar results obtained in the case of
perturbations of the Minkowski space-time; see, for example, \cite{M},
\cite{KSS}, \cite{KPV}, \cite{SmSo},\cite{St}, \cite{Strauss},
\cite{Al}, \cite{MS} and \cite{MT}.

 The metric for the Schwarzschild space-time can be written in the exterior region $r>2M$ as
\begin{equation}
ds^2=-\weight d\at^2+{\weight}^{-1}dr^2+r^2d\omega^2
\label{swm}\end{equation}
where $d\omega^2$ is the measure on the sphere $\S^2$.
The surface $r=2M$ is called the event horizon.  While the singularity
at $r=0$ is a true metric singularity, we note that the apparent
singularity at $r=2M$ is merely a coordinate singularity. Indeed,
let the Regge-Wheeler coordinate be given by
\[
\rs=r+2M\log(r-2M)-3M-2M\log M,
\]
and set $v = \at+\rs$. Then in the $(r,v,\omega)$ coordinates the metric
is expressed in the form
\[
ds^2=-\weight dv^2+2 dvdr+r^2d\omega^2,
\]
which extends analytically into the black hole region $r<2M$.

Following \cite{MMTT}, we introduce the function $t$ defined by
\[
t = v - \mu(r)=\at+\rs-\mu(r)
\]
where $\mu$ is a smooth function of $r$.
In the $(t,r,\omega)$ coordinates the metric has the form
\begin{equation*}
ds^2=-\weight dt^2 +2\left(1-\weight\mu'(r)\right) dt dr +
    \Bigl(2 \mu'(r) - \weight (\mu'(r))^2\Bigr)  dr^2  +r^2d\omega^2.
\end{equation*}

On the function $\mu$ we impose the following two conditions:

(i) $\mu (r) \geq  \rs$ for $r > 2M$, with equality for $r >
{5M}/2$.

(ii)  The surfaces $t = const$ are space-like, i.e.
\[
\mu'(r) > 0, \qquad 2 - \weight \mu'(r) > 0.
\]
\noindent (i) insures that the $(t,\ax)$,
coordinates, where $\ax\!=\!r\omega$, coincide with the $(\at,\ax)$ coordinates in $r
\!>\!{5M}\!/2$.

In the $r^*$ coordinates the metric takes the form
 \begin{equation}
ds^2=\weight (-dt^2+d{r^*}^2)+r^2d\omega^2
\label{swm2}\end{equation}

 Next we introduce rectangular coordinates
 \begin{equation}
 x=r\omega
 \end{equation}
 and express the Schwarzschild metric in the $(t,x)$ coordinates.
For $r>5M/2$ the expression for Schwarzschild metric in rectangular coordinates is
\begin{equation}
ds^2=-\weight dt^2+\Big(\weight^{-1} \omega_i\,\omega_j +\delta_{ij}-\omega_i\omega_j  \Big)dx^i dx^j
\end{equation}
Note that in these coordinates $\det{g_S}=-1$,  since $r^2 drd\omega =dx$.
The inverse metric in these coordinates is
\begin{equation}
g_S^{00}=-\weight^{-1},\qquad g_S^{0i}=0,\qquad
g_S^{ij}=\weight\omega^i\omega^j + \delta^{ij}-\omega^i\omega^j
\end{equation}

Alternatively the metric can be expressed in the rectangular Regge-Wheeler coordinates
 \begin{equation}
 x^*=r^*\omega
 \end{equation}
 and express the Schwarzschild metric in the $(t,x^*)$ coordinates.
For $r>5M/2$ the expression for Schwarzschild metric in rectangular Regge-Wheeler coordinates is
\begin{equation}
ds^2=-\weight dt^2+\Big(\weight \omega_i\,\omega_j +\delta_{ij}-\omega_i\omega_j  \Big)d{x^*}^i d{x^*}^j
\end{equation}
The inverse metric in these coordinates is
\begin{equation}
{g_S^*}^{00}=-\weight^{-1},\qquad {g_S^*}^{0i}=0,\qquad
{g_S^*}^{ij}=\weight^{-1}\omega^i\omega^j + \delta^{ij}-\omega^i\omega^j
\end{equation}

 Given $0 < r_e <2M$ we consider the wave equation
\begin{equation}
 \Box_{g_S} u = f
 \label{boxsinhom}\end{equation}
in the cylindrical region
\begin{equation}
 \M_{[t_0, t_1]} =  \{ t_0 \leq t \leq t_1, \ r \geq r_e \}
\label{mr}\end{equation}

The lateral boundary of $\M_{[t_0, t_1]}$,
\begin{equation}
 \Sigma_R^+ =   \M_{[t_0, t_1]} \cap \{ r = r_e\}
\label{mr+}\end{equation}
is space-like, and can be thought of as the exit surface
for all waves which cross the event horizon.

We define the  outgoing energy on $\Sigma_R^+$ as
\begin{equation}\label{energy2}
 E[u](\Sigma_R^+) = \int_{\Sigma_R^+}
 \left(  |\partial_r u|^2 +   |\partial_t u|^2    +
|\ang u|^2 \right) r_e^2   dt d\omega
\end{equation}
and the energy on an arbitrary $t$ slice as
\begin{equation}\label{energy3}
 E[u](\tau) = \int_{ \M_{[t_0, t_1]} \cap \{t = \tau\}}
 \left(
|\partial_r u|^2 +   |\partial_t u|^2    + |\ang u|^2
\right) r^2  dr  d\omega
\end{equation}

We now introduce the local energy norm $LE_S^1$ associated to the Schwarzschild space-time:
\begin{equation} \label{leS}
 \| u\|_{LE_S^1[t_0, t_1]}  = \|\partial_r u\|_{LE} + \left\| \left(1-\tfrac{3M}r \right) \partial_{t}
 u\right\|_{LE} + \left\| \left(1-\tfrac{3M}r \right) \ang u\right\|_{LE} + \| r^{-1} u\|_{LE}
\end{equation}
For the inhomogeneous term we use the norm
\begin{equation} \label{leSstara}
\|f\|_{LE^*_{S}[t_0, t_1]} = \big\|  \left(1-\tfrac{3M}r \right)^{-1} u\big\|_{LE}
\end{equation}
We implicitly assume that all norms on the right hand side of the formulas above
are restricted to the set $\M_{[t_0, t_1]}$ where we study the wave equation
\eqref{boxsinhom}.
The following result was proved in \cite{MMTT}:
\begin{theorem}\label{Schwarz}
 Let $u$ be so that $\Box_{g_S} u = f$.
 Then we have
\begin{equation}\label{main.est.Schw}
E[u](\Sigma_R^+) + \sup_{t_0\leq t\leq t_1} E[u](t) + \|u\|_{LE_S^1[t_0, t_1]}^2
\lesssim E[u](t_0) + \|f\|_{L^1L^2+LE^*_S}^2.
\end{equation}
\end{theorem}

We remark that the constant in \eqref{main.est.Schw} does not depend on $t_0$ and $t_1$. In particular we can obtain a global in time estimate, the counterpart for \eqref{localenergyflat} on Schwarzschild backgrounds, if we take $t_0=0$ and $t_1=\infty$.

For black holes, local energy estimates were first proved in \cite{LS} for radially
symmetric Schr\"odinger equations on Schwarzschild backgrounds.  In
\cite{BS1, BS2, BSerrata}, those estimates are extended to allow for
general data.  The same authors, in
\cite{BS3,BS4}, have provided studies that give certain improved
estimates near the photon sphere $r=3M$.  Moreover, we note that
variants of these bounds have played an important role in the works
\cite{BSter} and \cite{DR1}, \cite{DR2} which prove analogues of the
Morawetz conformal estimates on Schwarzschild backgrounds.

 As we will generalize Theorem \ref{Schwarz} to perturbations of Schwarzschild, we recall the key steps in its proof as done in \cite{MMTT}.  We begin with the
energy-momentum tensor
\[
Q_{\alpha\beta}[g]=\partial_\alpha u \partial_\beta u -
\frac{1}{2}g_{\alpha\beta}\partial^\gamma u \partial_\gamma u
\]

Its contraction with respect to a vector field $X$ is denoted by
\[
P_\alpha[g,X]=Q_{\alpha\beta}[g]X^\beta
\]
and its divergence is
\[
\nabla^\alpha P_\alpha[g,X] = \Box_g u \cdot Xu + \frac{1}{2}Q[g, X]
\qquad\text{where}\quad
Q[g, X] = Q_{\alpha\beta}[g]\pi_X^{\alpha \beta}
\]
and $\pi_X^{\alpha \beta} $ is the deformation tensor of $X$, which is given in terms of the Lie derivative by
\[
\pi_{\alpha\beta}^X=\nabla_\alpha X_\beta + \nabla_\beta X_\alpha=({\mathcal L}_X g)_{\alpha\beta}.
\]
A special role is played by the Killing vector field whose deformation tensor is zero
\[
K=\partial_{t} = \partial_{\tt}
\]

In general for a vector field $X$ independent of the metric we compute
\begin{equation}\label{deformation}
\pi^{\alpha\beta}_X=-{\mathcal L}_X g^{\alpha\beta}=-X(g^{\alpha\beta})
+g^{\alpha\gamma} \partial_\gamma X^\beta+g^{\beta\gamma}\partial_\gamma X^\alpha,
\end{equation}
in which case
\begin{equation}\label{defform}
Q[g, X] =-\frac{1}{\sqrt{|g|}}(X(\sqrt{|g|}g^{\alpha\beta}))
\partial_\alpha u\,\partial_\beta u \\
+(g^{\alpha\gamma}\partial_\gamma X^\beta +g^{\beta\gamma}\partial_\gamma X^\alpha)\partial_\alpha u\partial_\beta u- \partial_\gamma X^\gamma \, g^{\alpha\beta}\partial_\alpha u\,\partial_\beta u
\end{equation}

In particular if, in the rectangular coordinates, $X =b(r)\omega^i \partial_i$ then we get
\begin{multline}\label{QpiX2g}
Q[g, X] =-\frac{1}{\sqrt{|g|}}(X(\sqrt{|g|}g^{\alpha\beta}))
\partial_\alpha u\,\partial_\beta u \\
+2\frac{b(r)}{r} \big(g^{\alpha j}-g^{\alpha i}\omega_i \omega^j\big) \partial_\alpha u\partial_j u
+2 b^\prime(r) g^{\alpha r}\partial_r u\,\partial_\alpha u- \frac{1}{r^2} \partial_r(r^2 b(r)) \, g^{\alpha\beta}\partial_\alpha u\,\partial_\beta u,
\end{multline}
where latin indices are summed over $i,j=1,2,3$ only, $\omega^j=\omega_j$ and $g^{\alpha r}=g^{\alpha i}\omega_i$. Written in polar coordinates, the same expression takes the simpler form
\begin{equation}\label{QpiX2gpolar}
Q[g, X] =-\frac{1}{\sqrt{|g|}}(X(\sqrt{|g|}g^{\alpha\beta}))
\partial_\alpha u\,\partial_\beta u +2 b^\prime(r) g^{\alpha r}\partial_r u\,\partial_\alpha u- b^\prime(r) \, g^{\alpha\beta}\partial_\alpha u\,\partial_\beta u
\end{equation}
 If $g=g_S$ is the Schwarzschild metric then $\det{g_S}=-1$ and $g_S^{00}=-\big(1-\frac{2M}{r}\big)^{-1}$ so with $b(r)=a(r)\big(1-\frac{2M}{r}\big)$
\begin{multline}\label{QpiX2}
Q[g_S, X] =-b(r)\big(1-\frac{2M}{r}\big)^{-1}\partial_r \Big( \big(1-\frac{2M}{r}\big)g_S^{ij}\Big)
\partial_i u\,\partial_j u \\
+2\frac{b(r)}{r} \big(g_S^{i j}-g_S^{r r} \omega^i\omega^j\big) \partial_i u\partial_j u
+2 b^\prime(r) g_S^{r r}\partial_r u\,\partial_r u\\
- \frac{1}{r^2} \big(1-\frac{2M}{r}\big)\partial_r
\Big(r^2 \big(1-\frac{2M}{r}\big)^{-1} b(r)\Big) \, g_S^{\alpha\beta}\partial_\alpha u\,\partial_\beta u\\
=\frac{2b(r)}{r}\Big(1-\frac{3M}{r}\Big) \big(1-\frac{2M}{r}\big)^{-1}|\ang u|^2
+\Big(2b^\prime(r)\big(1-\frac{2M}{r}\big)-b(r)\frac{4M}{r^2}\Big) (\partial_r u)^2
\\
- \frac{1}{r^2} \big(1-\frac{2M}{r}\big)\partial_r
\Big(r^2 \big(1-\frac{2M}{r}\big)^{-1} b(r)\Big) \, g_S^{\alpha\beta}\partial_\alpha u\,\partial_\beta u\\
=2 \frac{a(r)}{r} \Big(1-\frac{3M}{r}\Big)|\ang u|^2+2 a^\prime(r)\big(1-\frac{2M}{r}\big)^2 (\partial_r u)^2
- \frac{1}{r^2} \big(1-\frac{2M}{r}\big)\partial_r
\Big(r^2 a(r)\Big) \, g_S^{\alpha\beta}\partial_\alpha u\,\partial_\beta u
\end{multline}

Integrating the above divergence relation for a suitable choice of $X$
does not suffice in order to prove the local energy estimates, as in
general the deformation tensor can only be made positive modulo a
Lagrangian term $q \partial^\alpha u \partial_\alpha u $.
Hence some lower order corrections are required.  For a vector field
$X$, a scalar function $q$ and a 1-form $m$ we define
\[
P_\alpha[g,X,q,m] = P_\alpha[g,X] + q u \partial_\alpha u - \frac12
(\partial_\alpha q) u^2 + \frac{1}{2}m_{\alpha}u^2.
\]
The divergence formula gives
\begin{equation}
\nabla^\alpha P_\alpha[g,X,q,m] =  \Box_g u \Bigl(Xu +
 q u\Bigr)+ Q[g,X,q,m],
\label{div}\end{equation}
where
\[
 Q[g,X,q,m] =
\frac{1}{2}Q[g, X] + q
\partial^\alpha u\, \partial_\alpha u + m_\alpha u\,
\partial^\alpha u + (\nabla^\alpha m_\alpha -\frac{1}{2}
\nabla^\alpha \partial_\alpha q) \, u^2.
\]

To prove the local energy decay in Schwarzschild space-time,
$X$, $q$ and $m$ are chosen as in the following lemma from \cite{MMTT}:
\begin{lemma}
There exist a smooth vector field
\[
 \tilde X=a(r)\big(1-\frac{2M}{r}\big)\partial_{r} + c(r)K
\]
a smooth function $\tilde q(r)$ and a smooth $1$-form $m$ supported near the
event horizon $r=2M$ so that

(i) The quadratic form $Q[g_S,\tilde X,\tilde q,m]$
is positive definite,
\begin{equation}
Q[g_S,\tilde X,\tilde q,m] \gtrsim r^{-2} |\partial_r u|^2 + \big(1-\frac{3M}r
\big)^2 (r^{-2} |\partial_t u|^2 + r^{-1}|\ang u|^2) + r^{-4}
u^2.
\label{posS}\end{equation}

(ii) $\tilde X(2M)$ points toward the black hole, $\tilde X(dr)(2M) < 0$, and
$\langle m,dr\rangle(2M) > 0$.

\label{ibp}
\end{lemma}

The local energy estimate is obtained by integrating the divergence
relation \eqref{div} with $\tilde X+CK$ instead of $\tilde X$, where $C$ is a large
constant,  on the domain $\M_{[t_0, t_1]}$, with respect to the measure induced by the metric, $dV_S = r^2 dr
d\omega dt$.  This yields
\begin{equation}
\int_{\M_{[t_0, t_1]}} Q[g_S,\tilde X,\tilde q,m]  dV_S =
- \int_{\M_{[t_0, t_1]}} \Box_{g_S} u \Bigl((\tilde X+CK)u + \tilde q u\Bigr) dV_S
+ BDR^S[u]
\label{intdiv}\end{equation}
where $BDR^S[u]$ denotes the boundary terms
\[
BDR^S[u] = \int \langle dt, P[g_S,\tilde X + C K,\tilde q,m]\rangle
  r^2 dr  d\omega \Big|_{t = t_0}^{t = t_1}
 \!\!  - \!  \int \langle dr,  P[g_S,\tilde X + C K,\tilde q,m]\rangle
  r_e^2 dt  d\omega
\]
Using the condition (ii) in the Lemma and Hardy type inequalities,
it is shown in \cite{MMTT} that for large  $C$ and $r_e$ close
to $2M$ the boundary terms have the correct sign,
\begin{equation}
  BDR^S[u]  \leq c_1 E[u](t_0) -
c_2 (E[u](t_1) + E[u](\Sigma_R^+)), \qquad c_1, c_2 > 0
\label{bdrpos}\end{equation}

We will use the notation $\chi_{R}(r)$ to denote a smooth nondecreasing cutoff function supported in $\{r>R\}$ so that $\chi \equiv 1$ in $\{r>2R\}$. For technical reasons we define, for any $m>0$ and a large enough constant $R_1$, the dual weighted norms $LE_{S,m}$ and $LE_{S,m}^*$ by
\[
\| u\|_{LE_{S,m}} = \|(1- \chi_{R_1}) u\|_{LE_S^1} + \| \chi_{R_1} r^{-1/2-m} \partial u\|_{L^2L^2} + \| \chi_{R_1} r^{-3/2-m} u\|_{L^2L^2}
\]
\[
\| F\|_{LE_{S,m}^*} = \|(1- \chi_{R_1}) F\|_{LE_S^*} + \| \chi_{R_1} r^{1/2+m} F\|_{L^2L^2}
\]

By applying the Cauchy-Schwartz inequality
for the first term on the right of \eqref{intdiv} we obtain
a slightly weaker form of the local energy estimate
\eqref{main.est.Schw}, namely
\begin{equation}\label{main.est.Schwa}
E[u](\Sigma_R^+) + {\sup}_{t_0\leq t\leq t_1} E[u](t) + \|u\|_{LE_{S,1}}^2
\lesssim E[u](\Sigma_R^-) + \|f\|_{LE^*_{S,1}}^2.
\end{equation}

These norms are equivalent with the stronger norms $LE_S^1$,
respectively $LE^*_S$ for $r$ in a bounded set. On the other hand for
large $r$ the Schwarzschild space can be viewed as a small perturbation
of the Minkovski space. Thus the transition from
\eqref{main.est.Schwa} to \eqref{main.est.Schw} is achieved in
\cite{MMTT} by cutting away a bounded region and then using a
perturbation of a Minkowski space estimate.

For our purposes, it will be useful to slightly modify $\tilde X$ near infinity in order to improve the $LE_{S,1}$ norm in \eqref{main.est.Schwa} to a $LE_{S,\delta}$ norm for some $\delta>0$ small enough. Let $R_1$ be large, and define
\[
 X_4 = \chi_{R_1} f(r) \partial_r  ,\qquad q_2(r) =  \chi_{R_1} \frac{f(r)}r
\]
 We pick the function $f$ to satisfy the conditions (for $r\geq R_1$):
\[
f'(r) \approx r^{-1-\delta}, \qquad \frac{f(r)}r - \frac12 f'(r) \gtrsim r^{-1-\delta}, \qquad -\Delta\Bigl(\frac{f(r)}r\Bigr) \gtrsim r^{-3-\delta}
\]
 One could take, for example,
 \begin{equation}\label{imprvinfty}
 f(r) = {\sum}_{j\geq 0} 2^{-\delta j} \frac{r}{r+2^j}
 \end{equation}
 By Proposition 8 in \cite{MT}, one has that
\[
Q[g_S, f(r) \partial_r, \frac{f(r)}r, 0] \gtrsim r^{-1-\delta} |\partial u|^2 + r^{-3-\delta} u^2, \qquad r\geq R_1
\]
 Since the derivative of $\chi_{R_1}$ is supported in the region $r\approx R_1$ and bounded by $r^{-1}$, we obtain
\begin{equation}\label{X4err}
  \int_{\M_{[t_0,t_1]}}  Q[g_S, X_4, q_2, 0] dV_S \gtrsim
\| \chi_{R_1} u\|_{LE_{S,\delta}}^2 - \int_{\M_{[t_0,t_1]}}\int_{r\approx R_1} r^{-1-\delta} |\partial u|^2 dV_S
\end{equation}

The boundary terms satisfy
\[
\Bigl |P_{\alpha}\Bigr | \lesssim |\partial u|^2 + \frac{1}{r^2} u^2
\]
on the time slices $t= t_i$, $i=0,1$. Note that, after integrating in space, the second term on the right can be controlled by the first by Hardy's inequality.

 Let
\begin{equation}\label{Xdef}
X = \tilde X + CK + \delta_2 X_4, \qquad q(r) = \tilde q(r) + \delta_2 q_2(r)
\end{equation}
 for some $\delta_2 \ll 1$ very small. The last term in \eqref{X4err} can now be absorbed in $Q[g_S, \tilde X, \tilde q, m]$. We now get from the inequalities above
 \begin{equation}\label{intdivs2}
    \int_{\M_{[t_0,t_1]}}  Q[g_S, X, q, m] dV_S=
    - \int_{\M_{[t_0,t_1]}} \Box_{\mathbf S} u \cdot  Xu \  dV_{\mathbf S}
    - \left. BDR^{\mathbf S}[u]\right|_{t = t_0}^{t = t_1} - \left. BDR^{\mathbf S}[u] \right|_{r=r_e}
  \end{equation}
  with
 \begin{equation}\label{imprvinftynorm}
  \int_{\M_{[t_0,t_1]}}  Q[g_S, X, q, m] dt dx\gtrsim
\| u\|_{LE_{S,\delta}}^2
  \end{equation}
and the boundary terms satisfying \eqref{bdrpos}. We thus obtain, after applying Cauchy-Schwarz:
\begin{equation}\label{Schwimpr}
E[u](\Sigma_R^+) + \sup_{t_0\leq t\leq t_1} E[u](t) + \|u\|_{LE_{S,\delta}}^2
\lesssim E[u](\Sigma_R^-) + \|f\|_{LE^*_{S,\delta}}^2.
\end{equation}


\newsection{Main Theorem}
 In this section we will give a precise version of our main theorem and outline the boot strap argument.

Let us start with some notation. We will use the coordinates $(\tt, x^{i})$, where $x^i =  r\omega$,
However, since the event horizon plays little role in our analysis, we will slightly abuse notation and denote by $\tt$ by $t$. For large $r$  the coordinates $x^{*i} =  \rs\omega$, may have better adapted to study perturbations of the Schwarzschild metric but it makes little difference since there we will use a null frame, see below.

Let $\tilde r$ denote a smooth strictly increasing function (of $r$) that equals $r$ for $r\leq R_1$ and $\rs$ for $r\geq 2R_1$, where $R_1 >>  6M$.  We will use throughout the paper the notation
\[
\M_{[t_0,t_1]}^{[r_1, r_2)} := \M_{[t_0,t_1]}\cap \{r_1\leq\tilde r< r_2\}
\]

Our favorite sets of vector fields will be
\[
\partial = \{ \partial_t, \partial_i\}, \qquad \Omega = \{x^i \partial_j -
x^j \partial_i\}, \qquad S = t \partial_t + \tilde r \partial_{\tilde r},
\]
namely the generators of translations, rotations and scaling. We set
$Z = \{ \partial,\Omega,S\}$.

For a triplet $\Lambda = (i,j,k)$ of multi-indices $i$, $j$ and $k$ we
denote $|\Lambda| = |i|+3|j|+3k$ and
\[
u_{\Lambda} = \partial^i \Omega^j S^k u, \qquad u_{\leq m} =
(u_{\Lambda})_{|\Lambda|\le m}.
\]

Since derivatives will play a special role, we will also use the notation
\[
\partial_{\leq m} u = (\partial_{\alpha} u)_{|\alpha|\leq m}.
\]

For two triplets $\Lambda_1$, $\Lambda_2$ we say that $\Lambda_1 \leq \Lambda_2$ if
\[
i_1 \leq i_2, \qquad j_1\leq j_2, \qquad k_1\leq k_2,
\]
and $\Lambda_1 < \Lambda_2$ if at least one of the inequalities above is strict.

For a positive function $f(t,r)$, we define the classes $S^Z(f)$ of
functions in $\R^+ \times \R^3$ by
\[
a \in S^Z(f) \Longleftrightarrow |Z^j a(t, r\omega)| \leq c_{j} f(t, r), \quad j \geq 0.
\]

 Let $h^{\alpha\beta}:= g^{\alpha\beta}-g_S^{\alpha\beta}$ be the difference in the metric coefficients. We will allow the metric $g$ to depend on the solution $u$, so that the difference in the metric coefficients $h^{\alpha\beta}(t, x, u)$ are smooth functions satisfying
\begin{equation}\label{metricform}
h^{\alpha\beta}(t, x, u) = H^{\alpha\beta}(t,x)u + O(u^2).
\end{equation}
Since we want $h\approx u$ we want the functions $H^{\alpha\beta}$ to satisfy
\begin{equation}\label{metriccoeff}
H^{\alpha\beta}\in S^Z(1).
\end{equation}

 For the derivatives of $H$ we need to impose extra conditions to make the metric satisfy the conditions in Theorem~\ref{LE}, namely
\begin{equation}\label{derivcoeff}
\partial H^{\alpha\beta}\in S^Z\Bigl(\frac{t^{1+\delta}}{r^{1+\delta}\la t-\rs\ra^{1/2+\delta}}\Bigr), \qquad  \rs\geq R_1^*,
\end{equation}
where $R_1^*=r^*(R_1)$.
 We remark that a natural condition to impose on $\partial H^{\alpha\beta}$ is
\[
 \partial H^{\alpha\beta}\in S^Z\Bigl(\frac{t}{r\la t-\rs\ra}\Bigr)
\]
as this yields that $\partial h\approx \partial u$. However, we chose to instead work with the weakest possible assumption under which we can prove our result, which is \eqref{derivcoeff}.

 Let $N$ be a large enough number. Let $N_1\!=\!\frac{N}2\! +\!2$. We assume that the function $u$ satisfies the decay rates
\begin{equation}\label{u-decay}
|u_{\leq N_1}|\lesssim \frac{\epsilon\la t-\rs\ra^{1/2}}{\la t\ra}, \qquad |\partial u_{\leq N_1}| \lesssim \frac{\epsilon}{\la r\ra\la t-\rs\ra^{1/2}}.
\end{equation}
 Combining \eqref{u-decay} and \eqref{metriccoeff}-\eqref{derivcoeff} we obtain decay rates for $h^{\alpha\beta}$ compatible with the  assumptions in Theorem~\ref{LE}:
\begin{equation}\label{metricdecay}
\begin{split}
|h^{\alpha\beta}_{\leq N_1}|\lesssim \frac{\epsilon\la t-\rs\ra^{1/2}}{\la t\ra}, \\
|\partial h^{\alpha\beta}_{\leq N_1}| \lesssim \frac{\epsilon}{\la t\ra^{1/2}}, \qquad r_e\leq r\leq R_1 ,\\
|\partial h^{\alpha\beta}_{\leq N_1}| \lesssim \frac{\epsilon}{r^{1+\delta}}, \qquad R_1^* \leq \rs\leq \frac{t}2 ,\\
|\partial h^{\alpha\beta}_{\leq N_1}| \lesssim \frac{\epsilon}{\la t\ra\la t-\rs\ra^{\delta}}, \qquad \frac{t}2\leq \rs.
\end{split}\end{equation}

 Moreover, due to \eqref{metricform}, \eqref{metriccoeff}, \eqref{derivcoeff}, and \eqref{u-decay} we obtain
\begin{equation}\label{metricvsu}
\begin{split}
& |h^{\alpha\beta}_{\Lambda}|\lesssim |u_{\leq \Lambda}| \\
& |\partial h^{\alpha\beta}_{\Lambda}|\lesssim |\partial u_{\leq \Lambda}| + \frac{t^{1+\delta}}{r^{1+\delta}\la t-\rs\ra^{1/2+\delta}} |u_{\leq \Lambda}|, \quad R_1^* \leq \rs
\end{split}\end{equation}
 We also note that, since $\partial_t = \frac1t S - \frac{\tilde{r}}t\partial_{\tilde r}$, we have better estimates for the time derivative in the region $\rs\leq \frac{t}2$:
\begin{equation}\label{t-derivest}
\begin{split}
|\partial_t u_{\leq N_1-2}| \lesssim \frac{\epsilon}{\la t\ra^{3/2}} ,\\
|\partial_t h^{\alpha\beta}_{\leq N_1-2}| \lesssim \frac{\epsilon}{\la t\ra^{3/2}}, \qquad r_e\leq r\leq R_1 ,\\
|\partial_t h^{\alpha\beta}_{\leq N_1-2}| \lesssim \frac{\epsilon}{\la r\ra^{\delta}\la t\ra}, \qquad R_1^* \leq \rs\leq \frac{t}2 ,
\end{split}\end{equation}

In addition we will need two technical conditions near the trapped set and the light cone. Let us define
\begin{equation}\label{W-def}
W^{\beta} =  \frac{1}{\sqrt{|g_S|}}\partial_{\alpha} (g_S^{\alpha\beta}\sqrt{|g_S|}) - \frac{1}{\sqrt{|g|}}\partial_{\alpha} (g^{\alpha\beta}\sqrt{g}).
\end{equation}

A priori, near the trapped set $W$ satisfies the bound (see Lemma~\ref{wavecoordscommut})
\begin{equation}\label{W-estfar}
|W_{\Lambda}| \lesssim |\partial h_{\Lambda}| + |h_{\leq \Lambda}| \lesssim |H||\partial u_{\Lambda}| + |u_{\leq \Lambda}|, \qquad |\Lambda|\leq N
\end{equation}
which for the highest order term will not suffice to close the estimates under the assumption \eqref{metriccoeff}. We will thus impose the extra assumption that
\begin{equation}\label{W-est}
|W_{\Lambda}| \lesssim \la t\ra^{-1/2} |\partial u_{\Lambda}| + |u_{\leq \Lambda}|, \qquad |\Lambda|\leq N,\qquad \text{when}\quad \tfrac{5M}2 \leq r \leq \tfrac{7M}2.
\end{equation}

One way to make sure this condition is satisfied is to assume, for example, that
\[
|H_{\leq N}| \lesssim r^{1/2}\la t\ra^{-1/2}
\]
near the trapped set. Indeed, the condition above clearly implies \eqref{W-est}. We remark that in the context of Einstein's Equations written down in generalized wave  coordinates $W = 0$, so good behavior of $W$ can be expected.

On the other hand, in the region close to the cone $\rs\sim t$, $\rs>\frac{t}2$ we need to assume
additional decay for the components of $h$ multiplying the worst decaying derivatives.
To formulate this we need to express $h$ in a nullframe:
\begin{equation}
\uL=\partial_t - \partial_{\rs}, \quad L= \partial_t + \partial_{\rs},
\quad A=A^i(\omega)\partial_i,\quad B=B^i(\omega)\partial_i,\quad
\partial_{\rs} =\big(1-\frac{2M}{r}\big)\omega^i\partial_i,
\label{nullframe}
\end{equation}
in the Schwarzschild metric $g_S$. Here $\uL$ and $L$ are null vectors
\[
g_S(L,L)=g_S(\uL,\uL)=0,\qquad g_S(L,\uL)=-2\big(1-\frac{2M}{r}\big),
\]
and $A$ and $B$
 are  two orthonormal vectors
 \[
g_S(A,A)=g_S(B,B)=1,\qquad g_S(A,B)=0,
 \]
 tangential to the spheres where $r$ is constant:
 \[
 g_S(L,B)=g_S(L,A)=g_S(\uL,B)=g_S(\uL,A)=0.
 \]

  Expanding $h$ in the nullframe
\begin{multline}
h^{\alpha\beta}=h^{\uL\uL}{\uL}^\alpha{\uL}^\beta
+ \sum_{T\in \mathcal{T}} h^{\uL T} ({\uL}^\alpha T^\beta +T^\alpha {\uL}^\beta)
+\sum_{U,T\in \mathcal{T}} h^{UT} U^\alpha T^\beta\\
=h^{\uL\uL}{\uL}^\alpha{\uL}^\beta
+\sum_{T\in \mathcal{T}} h^{\uL T} T^\alpha {\uL}^\beta
+\sum_{T\in \mathcal{T}} h^{\alpha T} T^\beta,
\qquad\text{where}\quad \mathcal{T}=\{L,A,B\},\label{nullframeexpansion}
\end{multline}
we see that we need to assume additional decay on $h^{\uL\uL}$.
We note that the coefficients in the nullframe expansion can be determined from
$h$ applied to the dual vectors with respect to the Schwarzschild metric
\begin{equation}
U_\alpha=g_{S\alpha\beta} U^\beta,\qquad U^\alpha V_\alpha=g_S(U,V).\label{loowering}
\end{equation}
A calculation shows that $A_i=A^i$, $B_i=B^i$, and
 \begin{equation}
 L_0=-\Big(1-\frac{2M}{r}\Big),\quad  L_i=\omega_i,\qquad
 \uL_0=-\Big(1-\frac{2M}{r}\Big),\quad  \uL_i=-\omega_i.\label{lowernull}.
 \end{equation}
With the notation
\begin{equation}
h_{\alpha\beta} = g_{S\alpha\gamma}g_{S\alpha\delta} h^{\gamma\delta},
\qquad h(U,V)=h_{\alpha\beta}U^\alpha V^\beta=h^{\alpha\beta}U_\alpha V_\beta,
\end{equation}
we have in particular
\begin{equation}
h^{\uL\uL}=h(L,L)/g_S(L,\uL)^2.\label{huLuLcomponent}
\end{equation}

We need to assume that $h^{\lL\lL}$ decays at a faster rate like in \eqref{badpart}
because it is the coefficient multiplying the  second derivative transversal to the
light cones that has the least decay.
More explicitly, we assume that it satisfies the decay estimates
\begin{equation}\label{badpart2}
|h^{\lL\lL}_{\Lambda}|\lesssim \la t-\rs\ra^{\delta} \la t \ra^{-\delta}|u_{\leq \Lambda}|, \qquad |\partial h^{\lL\lL}_{\Lambda} | \lesssim \la t-\rs\ra^{\delta}\la t \ra^{-\delta}\Bigl(|\partial u_{\leq \Lambda}| + \la t-\rs\ra^{-1/2} |u_{\leq \Lambda}| \Bigr)
\end{equation}
for some small $\delta>0$ and all $|\Lambda|\leq N$.

 Again, in the context of Einstein's Equations in wave  coordinates, we expect \eqref{badpart2} to hold (see \cite{LR2,Lin3}). Here it follows from the following assumption on
 $H$:
 \begin{equation}\label{nullcond}
|H^{\lL\lL}_{\leq N}|\lesssim \la t-\rs\ra^{\delta}\la t \ra^{-\delta},\qquad |\partial  H^{\lL\lL}_{\leq N} |\lesssim \la t-\rs\ra^{\delta}\la t\ra^{-\delta} \la t-\rs\ra^{-1/2}
 \end{equation}

The metric coefficient $h^{\lL \lL}$ is in front of the derivative with the least decay
$\partial_{\underline{L}}^2 u$. In \cite{LR2,Lin3}) it was proven that for Einstein's equations in wave coordinates
\begin{equation}
|Z^I h^{\lL \lL}|\leq \frac{C\varepsilon}{1+t+r} \Big(\frac{1+|t-r^*|}{1+t}\Big)^\gamma,\label{einsteinwave}
\end{equation}
if initial data are asymptotically flat, i.e. $h\big|_{t=0}=M/r+O\big(r^{-1-\gamma}\big)$,
$0\!<\!\gamma\!<\!1$. Our method here works for the case corresponding to any small $\gamma>0$; if one assumes more decay of the coefficient
$|H^{\lL\lL}|\leq C\varepsilon \langle t-r^*\rangle^\gamma \langle t\rangle^{-\gamma}$ corresponding to larger $\gamma$ it may be possible to prove some additional decay for the solution in the interior.

 We are now ready to state the our main result. We pick a large enough integer $N\geq 36$, and define
\[
\mathcal{E}_N(t) = \|\partial  u_{\leq N}\|_{L^{\infty}[0, t] L^2}^2 + \|u_{\leq N}\|_{LE_S^1[0, t]}^2
\]

\begin{theorem}\label{mainthm}

Assume that the metric $g$ is like in \eqref{quasmetric}, and satisfies the extra conditions \eqref{metriccoeff}, \eqref{derivcoeff}, \eqref{W-est} and \eqref{nullcond}. Then there exists a global classical solution to \eqref{maineeqn}, provided that the initial data is smooth, compactly supported and satisfies, for a certain $\epsilon_0 \ll 1$,
\[
\mathcal{E}_N(0) \leq \epsilon_0^2
\]
Moreover, the solution satisfies the estimates \eqref{enapbds} and \eqref{ptwseapbds} below.

\end{theorem}

We will now outline the bootstrap argument. The rest of the paper is devoted to proving the required higher order local energy estimates and pointwise decay bounds necessary for the bootstrap.

We assume that the initial data is small enough,
\begin{equation}\label{indata}
\mathcal{E}_N(0) \leq \mu_N \epsilon^2
\end{equation}
where $N\geq 36$ and $\mu_N>0$ is a fixed, small $N$-dependent constant to be determined below.

 Let $N_1=\frac{N}2 +2$ and $\tN = N-3$.We will assume that the following a-priori bounds hold for some large constant $\tilde C$ independent of $\epsilon$ and $t$, and a fixed small $\delta>0$
\begin{equation}\label{enapbds}
\mathcal{E}_N(t) \leq \tilde C \mu_N\epsilon^2 \la t\ra^{\delta}, \qquad \mathcal{E}_{\tN}(t) \leq \tilde C \mu_N\epsilon^2,
\end{equation}
\begin{equation}\label{ptwseapbds}
|u_{\leq N_1}| \leq \frac{\epsilon\la t-\rs\ra^{1/2}}{\la t\ra}, \qquad |\partial u_{\leq N_1}| \leq \frac{\epsilon}{\la r\ra\la t-\rs\ra^{1/2}}
\end{equation}

Clearly \eqref{enapbds} and \eqref{ptwseapbds} are true for small enough times by standard local theory existence combined with \eqref{indata} and Sobolev embeddings. We will now assume that \eqref{enapbds} and \eqref{ptwseapbds} hold for all $0\leq t\leq T$, and we will improve the constants on the right hand side by a factor of $1/2$. By a standard continuity argument this will imply the desired result.

Due to the assumptions \eqref{ptwseapbds} we can apply Theorem~\ref{higherLE}. We obtain
\[
\mathcal{E}_N(t) \leq C_N \la t\ra^{C_N \epsilon} \mathcal{E}(0)
\]
If we take $\tilde C = 2C_N$ and $\epsilon < \frac{\delta}{C_N}$ we improve the a-priori bound for $\mathcal{E}_N(t)$ to
\[
\mathcal{E}_N(t) \leq \frac12 \tilde C  \mu_N \epsilon^2 \la t\ra^{\delta}
\]

Similarly due to the assumptions \eqref{ptwseapbds} we can apply Theorem~\ref{lowerLE} to obtain
\[
\mathcal{E}_{\tN}(t) \leq \frac12 \tilde C  \mu_N \epsilon^2
\]

 Finally, since $N_1\leq \tN - 13$, we can apply Theorem~\ref{ptwsedcy} and obtain
\[
|u_{\leq N_1}| \leq C'_{N_1}\frac{\la t-\rs\ra^{1/2}}{\la t\ra} \sqrt{\mathcal{E}_{\tN}(T)} \leq \frac12 \frac{\epsilon\la t-\rs\ra^{1/2}}{\la t\ra}
\]
\[
|\partial u_{\leq N_1}| \leq C'_{N_1}\frac{1}{\la r\ra\la t-\rs\ra^{1/2}} \sqrt{\mathcal{E}_{\tN}(T)} \leq \frac12 \frac{\epsilon}{\la r\ra\la t-\rs\ra^{1/2}}
\]
since we can choose $\mu_0$ so that $C'_{N_1}\tilde{C}^{1/2}\mu_N^{\frac12} \leq 1/2$.

\newsection{Local energy estimates for perturbations of Schwarzschild}
Let $g_S$ be the Schwarzschild metric, $R_1$ be a large constant, and $\delta>0$ be an arbitrarily small number. Let $g$ be a metric that is a small perturbation of $g_S$ in the sense that the difference $h^{\alpha\beta}:= g^{\alpha\beta}-g_S^{\alpha\beta}$ satisfies
\begin{equation}\label{diff}
 |h^{\alpha\beta}|, |\partial h^{\alpha\beta}|\lesssim \epsilon
\end{equation}
 everywhere in the coordinates $(t, x)$, where $x=r\omega$. Moreover, near the trapped set and the light cone we need additional decay estimates as follows:

i) When $\frac{11M}4 < r < \frac{13M}4$ (which is a region close to the trapped set) we have
\begin{align}
 | h^{\alpha\beta} |&\lesssim \epsilon \la t\ra^{-1/2} \label{cpt1}\\
 |\partial_{t} h^{\alpha\beta}| &\lesssim \epsilon \la t\ra^{-1-\delta},\label{cpt2}\\
 |\partial_r h^{\alpha\beta}|  &\lesssim \epsilon \la t\ra^{-1/2},\label{cpt3}
\end{align}

ii) In the intermediate region $R_1^* \leq \rs\leq \frac{t}2$ we will assume that
\begin{equation}\label{intrm}
| h^{\alpha\beta}| \lesssim \epsilon r^{-\delta}, \qquad |\partial h^{\alpha\beta}| \lesssim \epsilon r^{-1-\delta}
\end{equation}

iii) In the region close to the cone $\rs>\frac{t}2$ we need to assume different decay rates for different components.

The component of the metric that multiply the derivatives with worst decay $\partial_{\lL}^2 u$ will be required to satisfy the better decay estimates
\begin{equation}\label{badpart}
|\partial h^{\uL\uL} | \lesssim \epsilon\la t \ra^{-1-\delta}, \qquad |h^{\uL\uL}|\lesssim \epsilon\la t-\rs\ra \la t\ra^{-1-\delta}.
\end{equation}
This is needed for the estimates and is
consistent with what holds for Einstein's equations \eqref{einsteinwave}.

 The other components of $h$ only need to satisfy the weaker estimates:
\begin{equation}\label{goodpart}
 |\partial h | \lesssim \epsilon\la t \ra^{-\frac12-\delta}\la t-\rs\ra^{-\frac12-\delta}, \qquad |h|\lesssim \epsilon\la t-\rs\ra^{\frac12-\delta}\la t\ra^{-\frac12-\delta}.
\end{equation}

  Note in particular that since $|\partial g_S^{\alpha\beta}| \lesssim r^{-2}$ we have in the region $r\geq R_1$:
\begin{equation}\label{deriv}
 |\partial g^{\alpha\beta}| \lesssim r^{-2} + \epsilon (r^{-1-\delta} + \la t \ra^{-\frac12-\delta}\la t-\rs\ra^{-\frac12-\delta})
\end{equation}

 We will also denote by
 \[
 \M_{[t_0,t_1]}^{ps} = \M_{[t_0,t_1]}\cap\{5M/2 \leq r\leq 7M/2\}
 \]
 a neighborhood of the photonsphere.

 The main goal of this section is to prove the following local energy estimate:

\begin{theorem}\label{LE} Let $u$ solve the inhomogeneous linear wave equation
\begin{equation}\label{inhomwave}
 \Box_g u = F
\end{equation}
where $g$ is a Lorentzian metric satisfying the conditions above. Then for any $t_0 < t_1$
\begin{equation}\label{LES}\begin{split}
\|\partial  u\|_{L^{\infty}[t_0, t_1] L^2}^2 + \|u\|_{LE_S^1[t_0, t_1]}^2 \lesssim & \int_{\M_{[t_0,t_1]}^{ps}} \frac{\epsilon}{t}|\partial u|^2 dV_g + \|\partial u(t_0)\|_{L^2}^2 + \\
 & \inf_{F=F_1+F_2} \int_{\M_{[t_0,t_1]}} |F_1| (|\partial u| + r^{-1} |u|) + |\partial_{\leq 1} F_2|^2 dV_g
\end{split}\end{equation}
where $F_2$ is supported near $r=3M$.
\end{theorem}

 We presented the theorem in the form that is most convenient to us for applications in subsequent sections. The $F_2$ term will only be useful to treat commutations with vector fields near the trapped set, see for example \eqref{vfcpt3}, and will otherwise equal $0$.

 After applying Cauchy Schwarz one can easily obtain a result similar Theorem~\ref{Schwarz}, namely
\[
\|\partial  u\|_{L^{\infty}[t_0, t_1] L^2}^2 + \|u\|_{LE_S^1[t_0, t_1]}^2 \lesssim \int_{\M_{[t_0,t_1]}^{ps}} \frac{\epsilon}{t}|\partial u|^2 dV_g + \|F\|_{L^1L^2+LE^*_S}^2
\]
 which combined with Gronwall's inequality implies in particular that
\begin{equation}\label{LESgrow}
\|\partial  u\|_{L^{\infty}[t_0, t_1] L^2}^2 + \|u\|_{LE_S^1[t_0, t_1]}^2 \lesssim t_1^{C_0\epsilon}\Bigl(\|F\|_{L^1L^2+LE^*_S}^2 + \|\partial u(t_0)\|_{L^2}^2\Bigr)
\end{equation}
for some constant $C_0$ independent of $t_0$, $t_1$.

 We also remark that the growth in $t$ can be removed if we allow a loss of one derivatives on the initial data and the inhomogeneous term $F$. We refer the reader to Theorem~\ref{LESbded} for such a statement.

 We will now prove Theorem~\ref{LE}. We start with a few technical lemmas.

\begin{lemma}\label{wavecoords}
Let
\[
\th^{\alpha\beta}=\sqrt{|g|} g^{\alpha\beta}-\sqrt{|g_S|}g_S^{\alpha\beta}.
\]

 Then $\th$ satisfies the estimates:
\begin{equation}\label{metdiff}
 | \th^{\alpha\beta}| \lesssim |h| \sqrt{|g|}, \qquad | \th^{\uL\uL}| \lesssim |h^{\uL\uL}| \sqrt{|g|}
\end{equation}
\begin{equation}\label{derivmetdiff1}
 \Bigl|\partial_{\mu} \th^{\alpha\beta} \Bigr| \lesssim (r^{-2}|h| + |\partial_{\mu} h|) \sqrt{|g|}, \qquad \Bigl|\partial_{t} \th^{\alpha\beta} \Bigr| \lesssim |\partial_{t} h| \sqrt{|g|}
\end{equation}
\begin{equation}\label{derivmetdiff2}
|\partial_{\mu} \th^{\uL\uL}| \lesssim (|\partial_{\mu} h^{\uL\uL}| + |h^{\uL\uL}|(r^{-2} + |\partial_\mu h|)) \sqrt{|g|}
\end{equation}

\end{lemma}

In particular, if $h$ satisfies the estimates \eqref{diff}, \eqref{cpt1}, \eqref{cpt2}, \eqref{cpt3}, \eqref{intrm}, \eqref{badpart} and \eqref{goodpart}, then so does $\th$.

\begin{proof}
We start by noticing that
\[
|\sqrt{|g|} g^{\alpha\beta}-\sqrt{|g_S|}g_S^{\alpha\beta}| \leq |h^{\alpha\beta}| \sqrt{|g|} + |g_S^{\alpha\beta}(\sqrt{|g|}-\sqrt{|g_S|})|
\]
which yields the first half of \eqref{metdiff} since $g_S^{\alpha\beta}$ is uniformly bounded. The second part of \eqref{metdiff} follows immediately since $g_S^{\uL\uL}=0$.

 To prove \eqref{derivmetdiff1} we write
\[\begin{split}
\Bigl|\partial_{\mu} (g^{\alpha\beta} \sqrt{|g|} - g_S^{\alpha\beta} \sqrt{|g_S|}) \Bigr| \lesssim |\partial_{\mu} (\sqrt{|g|} h^{\alpha\beta})| + \Bigl|\partial_{\mu} \Bigl(g_S^{\alpha\beta} (\sqrt{|g|}-\sqrt{|g_S|})\Bigr)\Bigr| \\
\lesssim |(\partial_{\mu}\sqrt{|g|}) h^{\alpha\beta}| +  |\sqrt{|g|}\partial_{\mu} h^{\alpha\beta}| + |(\partial_{\mu} g_S^{\alpha\beta})(\sqrt{|g|}-\sqrt{|g_S|})| + |g_S^{\alpha\beta}\partial_{\mu} (\sqrt{|g|}-\sqrt{|g_S|})|
\end{split}\]

 Let us estimate the first term. We use the formula (see for example \cite{Wald})
\begin{equation}\label{detderiv}
 |g|^{-1}\partial_{\mu} |g|= g^{\alpha\beta} \partial_{\mu} g_{\alpha\beta} = - g_{\alpha\beta} \partial_{\mu} g^{\alpha\beta}
\end{equation}

 This implies
\[
|(\partial_{\mu}\sqrt{|g|}) h^{\alpha\beta}| \lesssim |\sqrt{|g|} g_{\lambda\nu}(\partial_{\mu} g^{\lambda\nu}) h^{\alpha\beta}| \lesssim (r^{-2} + |\partial_\mu h|) |h^{\alpha\beta}|\sqrt{|g|} \lesssim (r^{-2}|h| + |\partial_{\mu} h|) \sqrt{|g|}
\]

The second term clearly satisfies the desired estimate.

 For the third term, we start by noting that since $|\partial g_S^{\alpha\beta}| \lesssim r^{-2}$ we have
\[
 |(\partial_{\mu} g_S^{\alpha\beta})(\sqrt{|g|}-\sqrt{|g_S|})| \lesssim r^{-2}|h|\sqrt{|g|}
\]

 For the fourth term, we note that
\[
|\partial_{\mu} (\sqrt{|g|}-\sqrt{|g_S|})| \lesssim |(\partial_{\mu} h^{\lambda\nu}) g_{\lambda\nu}\sqrt{|g|}| + |\partial_{\mu} g_S^{\lambda\nu} (h_{\lambda\nu}\sqrt{|g|} + (g_{S})_{\lambda\nu} (\sqrt{|g|}-\sqrt{|g_S|}))|
\lesssim  (r^{-2}|h| + |\partial_\mu h|) \sqrt{|g|}
 \]

 The proof of the first half of \eqref{derivmetdiff1} is now complete. For the second part, we use the same argument combined with the fact that $|\partial_t g_S| = 0$.

 To prove \eqref{derivmetdiff2}, we use that $g_S^{\underline{L}\underline{L}} = 0$:
\[
|\partial_{\mu} \th^{\uL\uL}| \lesssim |\partial_{\mu} h^{\uL\uL}|\sqrt{|g|} + |h^{\uL\uL}\partial_{\mu}\sqrt{|g|}| \lesssim (|\partial_{\mu} h^{\uL\uL}| + |h^{\uL\uL}|(r^{-2} + |\partial_\mu h|)) \sqrt{|g|}
\]
\end{proof}

The second lemma is the following refined version of a weighted local energy estimate in \cite{LR1}:

\begin{lemma}\label{tangweight}
Assume that the metric $g$ and the functions $u$ and $F$ are the ones in Theorem~\ref{LE}.
Then
\begin{equation}\label{tangweighteq}
 \int_{t_0}^{t_1} \int_{r\geq R_1} \frac{|\overline{\partial} u|^2}{\la t-\rs\ra^{1+\delta}}dV_g \lesssim E[t_0] + \int_{t_0}^{t_1} \int_{r\geq R_1/2} r^{-1-\delta}|\partial u|^2 + |F\partial_t u| dV_g
\end{equation}

Here and for the rest of the paper, we denote by $\overline{\partial}$ the directions tangent to the light cones in Schwarzschild:
\[
\overline{\partial} = \{L, \frac1r e_A, \frac1r e_B\}
\]
where $\{e_A, e_B\}$ is an orthonormal frame associated to the unit sphere.
\end{lemma}

\begin{proof}
Let $\rho=t-\rs$ and $Y=f(\rho) \partial_t$, where $f(x)= x^{-\delta}$. By using \eqref{defform} we obtain
\begin{multline}
Q[g, Y] +\frac{1}{\sqrt{|g|}}(Y(\sqrt{|g|}g^{\alpha\beta}))
\partial_\alpha u\,\partial_\beta u =2g^{\alpha\gamma}\partial_\gamma f(\rho)\,\partial_\alpha u\,\partial_t u- \partial_t f(\rho) \, g^{\alpha\beta}\partial_\alpha u\,\partial_\beta u\\
=\partial_t f(\rho)\, \big( g^{00} (\partial_t u)^2-g^{ij}\partial_i u\,\partial_j u\big)+2\partial_i f(\rho)\, g^{\alpha i} \partial_t u\, \partial_\alpha u\\
=f^\prime(\rho)\Big( g^{00} (\partial_t u)^2-g^{ij}\partial_i u\,\partial_j u-2\omega_i \frac{dr^*}{dr} g^{\alpha i} \partial_t u\, \partial_\alpha u\Big)
\label{QgY}
\end{multline}

In particular if, $g=g_S$ is the Schwarzschild metric  then
\begin{multline*}
Q[g_S, Y] =-f^\prime(\rho)\,\Big( \big(1-\frac{2M}{r}\big)^{-1}(\partial_t u)^2+
 \big(1-\frac{2M}{r}\big) (\partial_r u)^2 +|\ang u|^2 +2
 \big(1-\frac{2M}{r}\big)^{-1} \omega_i\, g_S^{\alpha i} \partial_t u\, \partial_\alpha u
\Big)\\
=-f^\prime(\rho)\,\Big( \big(1-\frac{2M}{r}\big)^{-1} \big((\partial_t+\partial_{r^*})u\big)^2  +|\ang u|^2\Big)
\end{multline*}
and so
\begin{equation}\label{gS}
Q[g_S, Y] \approx \frac{|\overline{\partial} u|^2}{\la t-\rs\ra^{1+\delta}}, \qquad r\geq R_1
\end{equation}

For the difference we have using the first line of \eqref{QgY}
\begin{equation*}
\sqrt{|g|} \,Q[g, Y]-\sqrt{|g_S|} \, Q[g_S, Y]
=f(\rho)\partial_t \th^{\alpha\beta}\, \partial_\alpha u\, \partial_\beta u
-f^\prime(\rho)\Big( 2\th^{\alpha\gamma}\partial_\gamma\rho\,\partial_\alpha u\,\partial_t u + \, \th^{\alpha\beta}\partial_\alpha u\,\partial_\beta u\Big).
\end{equation*}
(One can alternatively write $\partial_\gamma\,\rho =-g_{S\,\gamma\delta} {L}^\delta dr^*\!/dr$ and $L=\partial_t+\partial_{r^*}$.)
Expanding $h$ in the null frame using \eqref{nullframeexpansion} and using that
 $\partial_T \rho=0$, for $T\in\mathcal{T}$, we see that
\[
|\th^{\alpha\gamma}\partial_\gamma \rho\,\partial_\alpha u\,\partial_t u| \lesssim
|\th^{\uL\uL}|(\partial u)^2+|\th| |\overline{\partial} u| |\partial u|,\qquad\quad
|\th^{\alpha\beta}\partial_\alpha u\,\partial_\beta u| \lesssim |\th^{\uL\uL}|(\uL u)^2 + |\th| |\overline{\partial} u| |\partial u|.
\]
 By using Cauchy Schwarz and the extra decay for $h^{\uL\uL}$, namely \eqref{derivmetdiff1} and \eqref{badpart}, in conjunction with Lemma~\ref{wavecoords}, we obtain that
\begin{equation}\label{g-gS}
|\sqrt{|g|} \,Q[g, Y]-\sqrt{|g_S|} \, Q[g_S, Y]| \lesssim \epsilon \Bigl(r^{-1-\delta} |\partial u|^2 + \rho^{-1-\delta}|\overline{\partial} u|^2\Bigr)
\end{equation}

We now define
\[
\tY = \chi_{R_1/2} (r) Y
\]
Clearly we have that
\begin{equation}\label{Y-tY}
Q[g, \tY] = 0, \quad r\leq R_1/2, \qquad |Q[g, \tY] - Q[g, Y]| \lesssim \chi\prime_{R_1/2} |\partial u|^2, \quad r\geq R_1/2
\end{equation}
In order to prove the lemma, we multiply \eqref{inhomwave} by $\tY u$ and apply the divergence theorem. We obtain
\[
\int_{\M_{[t_0,t_1]}} Q[g, \tY] dV_g = -\int_{\M_{[t_0,t_1]}} F \cdot  \tY u \  dV_{g} - \left. BDR^{g}[u]\right|_{t = t_0}^{t = t_1}
\]

 Due to \eqref{gS}, \eqref{g-gS} and \eqref{Y-tY} we obtain
\[
\int_{\M_{[t_0,t_1]}} Q[g, \tY] dV_g \gtrsim \int_{t_0}^{t_1} \int_{r\geq R_1} \frac{|\overline{\partial} u|^2}{\la t-\rs\ra^{1+\delta}}dV_g - \int_{t_0}^{t_1} \int_{r\geq R_1/2} r^{-1-\delta} |\partial u|^2
\]

 On the other hand, the two boundary terms trivially satisfy
\[
\left. BDR^{\mathbf S}[u]\right|_{t = t_i} \approx  \  \|\chi_{R_1/2} \partial  u(t_i)\|_{L^2}^2,
  \qquad i=1,2
\]

 Lemma~\ref{tangweight} now follows from the divergence theorem.
\end{proof}

 The third lemma is a Hardy type inequality near the cone.

 \begin{lemma}\label{Hardylemma}

The lower order terms satisfy
\begin{equation}\label{lot}
\int_{t_0}^{t_1} \int_{r\geq R_1} r^{-2} \la t\ra^{-\frac12-\delta} \la t-\rs\ra^{-\frac12-\delta} |u|^2 dV_g \lesssim  \int_{t_0}^{t_1} \int_{r\geq R_1} r^{-1-\delta}|\partial u|^2 + r^{-3-\delta}|u|^2 dV_g
\end{equation}

\end{lemma}

\begin{proof}
Let $\chi(x)$ be a smooth cutoff so that $\chi\equiv 1$ when $x\geq 1$ and $\chi\equiv 0$ when $x\leq 1/2$. Since $\la t-\rs\ra \gtrsim r$ in the region where $\chi(r/t) \ne 1$, we obtain far from the cone:
\[
\int_{t_0}^{t_1} \int_{r\geq R_1} r^{-2} \la t\ra^{-\frac12-\delta} \la t-\rs\ra^{-\frac12-\delta} (1-\chi^2(r/t))|u|^2 dV_g \lesssim  \int_{t_0}^{t_1} \int_{r\geq R} r^{-3-\delta}|u|^2 dV_g
\]

 Close to the cone, we apply the following Hardy-type inequality from \cite{LR1}, which holds for all $C^1$ compactly supported functions $f$:
\begin{equation}\label{LRHardy}
 \int_{\R^3} \frac{f^2}{\la t-\rs\ra^{\gamma}} dx \lesssim \int_{\R^3} \frac{(\partial_r f)^2}{\la t-\rs\ra^{\gamma-2}} dx, \qquad \gamma > 1, \qquad \gamma\ne 3
\end{equation}

 We will apply \eqref{LRHardy} to the function $\chi(r/t)u$, and obtain (taking also into account the support properties of $\chi$ and $\sqrt{|g|}\approx \sqrt{|g_S|}$):
\[\begin{split}
& \int_{t_0}^{t_1} \int_{r\geq R_1} r^{-2} \la t\ra^{-\frac12-\delta} \la t-\rs\ra^{-\frac12-\delta} \chi^2(r/t)|u|^2 dV_g
 \lesssim \int_{t_0}^{t_1} \int_{r\geq R_1} t^{-2-\delta} \la t-\rs\ra^{-1-\delta} \chi^2(r/t)|u|^2 dV_S \\
& \lesssim \int_{t_0}^{t_1} t^{-2-\delta} \int_{r\geq R_1} \la t-\rs\ra^{1-\delta} \partial_r (\chi^2(r/t)u)^2 dV_S
 \lesssim \int_{t_0}^{t_1} \int_{r\geq R_1} r^{-2-\delta}(\partial_r u)^2 + r^{-3-\delta} u^2 dV_g
\end{split}
\]
 The proof of \eqref{lot} is now complete.
\end{proof}

 We are now ready for the proof of Theorem~\ref{LE}.

 \begin{proof}
  Let $X$ and $q$ be the vector field and scalar function defined in \eqref{Xdef}. We proved in the previous section (see \eqref{intdivs2}) that
 \[
    \int_{\M_{[t_0,t_1]}}  Q[g_S,X,q,m]  dt dx=
    - \int_{\M_{[t_0,t_1]}} \Box_{\mathbf S} u \cdot  Xu \  dV_{\mathbf S}
    - \left. BDR^{\mathbf S}[u]\right|_{t = t_0}^{t = t_1} - \left. BDR^{\mathbf S}[u] \right|_{r=r_e}
 \]
  with (see \eqref{imprvinftynorm}
  \[
  \int_{\M_{[t_0,t_1]}}  Q[g_S,X,q,m]  dt dx\gtrsim
\| u\|_{LE_{S,\delta}[t_0, t_1]}^2
  \]
and the boundary terms satisfying
\[
  \left. BDR^{\mathbf S}[u]\right|_{t = t_i} \approx  \  \|\partial  u(t_i)\|_{L^2}^2,
  \quad i=1,2,\qquad\quad
  \left. BDR^{\mathbf S}[u]\right|_{r=r_e} \approx \ \| u\|_{H^1(\Sigma^+_{[t_0,t_1]})}^2.
\]

  We will first prove a weaker version of \eqref{LES} for the metric $g$, namely
 \begin{equation}\label{wkdecay}
\begin{split}
\|\partial  u\|_{L^{\infty}[t_0, t_1] L^2}^2 & + \| u\|_{LE_{S,\delta}[t_0, t_1]}^2  \lesssim  \int_{\M_{[t_0,t_1]}^{ps}} \frac{\epsilon}{t}|\partial u|^2 dV_g + \|\partial u(t_0)\|_{L^2}^2 + \\ & \inf_{F=F_1+F_2} \int_{\M_{[t_0,t_1]}} |F_1| (r^{\delta} |\partial u| + r^{-1+\delta} |u|) + |\partial F_2| |u| dV_g
\end{split}\end{equation}

 We use the same $X$ and $q$ from \eqref{Xdef} as a multiplier for our metric $g$. We obtain
\begin{equation}\label{intdivg}
    \int_{\M_{[t_0,t_1]}} Q[g,X,q,m] dt dx=
    - \int_{\M_{[t_0,t_1]}} \Box_g u \cdot  Xu \  dV_{g}
    - \left. BDR^{g}[u]\right|_{t = t_0}^{t = t_1} - \left. BDR^{g}[u] \right|_{r=r_e}
  \end{equation}

 Due to the smallness condition \eqref{diff} we see that the boundary terms still satisfy \eqref{bdrpos}. Moreover,
\[\begin{split}
\Bigl |\int_{\M_{[t_0,t_1]}}\!\!\!\!\!\!\!\!\!\!\!\! F_1 \cdot  Xu \  dV_{g} \Bigr|  \lesssim \int_{\M_{[t_0,t_1]}}\!\!\!\!\!\!\!\!\!\!\!\!  |F_1| (|\partial u| + r^{-1} |u|) dV_g \leq \delta_0 \| u\|_{LE_{S,\delta}[t_0, t_1]}^2   + C(\delta_0) \int_{\M_{[t_0,t_1]}} \!\!\!\!\!\!\!\!\!\!\!\! |F_1| (r^{\delta} |\partial u| + r^{-1+\delta} |u|)  dV_g
\end{split}\]
and we can absorb the first term to the LHS of \eqref{wkdecay} for a small enough $\delta_0$.

 For the term involving $F_2$, the only problematic component is $CK$. We integrate by parts in time and use the trace theorem for the boundary terms at $t=t_0$ and $t=t_1$. Since $F_2$ is compactly supported, we obtain
\[
\Bigl |\int_{\M_{[t_0,t_1]}} F_2 \cdot  Xu \  dV_{g} \Bigr| \leq \delta_0 (\|\partial  u\|_{L^{\infty}[t_0, t_1] L^2}^2 + \| u\|_{LE_{S,\delta}[t_0, t_1]}^2) + C(\delta_0) \int_{\M_{[t_0,t_1]}}|\partial_{\leq 1} F_2|^2 dV_g
\]

 By letting $\delta_0$ be small enough, we can absorb the terms involving $\|\partial  u\|_{L^{\infty}[t_0, t_1] L^2}^2$ and $\| u\|_{LE_{S,\delta}[t_0, t_1]}^2$ to the left hand side of \eqref{wkdecay}.

 In order to finish the proof of \eqref{wkdecay}, it remains to show that
\begin{equation}\label{diffQint}
\begin{split}
& \int_{\M_{[t_0,t_1]} } Q[g,X,q,m]\sqrt{|g|} - Q[g_S,X,q,m]\sqrt{|g_S|} dt dx \\ & \lesssim \epsilon \Bigl(\int_{\M_{[t_0,t_1]} } Q[g_S,X,q,m]\sqrt{|g_S|}dt dx + \int_{\M_{[t_0,t_1]}^{ps}} t^{-1} |\partial u|^2 dt dx+ \|\partial u\|^2_{L^{\infty}[t_0, t_1]L^2}\Bigr)
\end{split}
\end{equation}
 We note here that it is the presence of the term $\int_{\M_{[t_0,t_1]}^{ps}} t^{-1} |\partial u|^2 dt dx$ that yields increase in time of the norms; without it, Gronwall's inequality gives boundedness. We will come back to this in Section 6.

 It will be convenient to use the notation of \cite{MMTT}. We write down more explicitly the vector $\tilde X$ as
 \begin{equation}\label{tXdef}
   \tilde X = X_1 + X_2 + CK
  \end{equation}
Here $X_2$ is a smooth vector field supported when $r<\frac{5M}2$, and
\[
 X_1 = b(r)\partial_{r} = \weight a(r)\partial_{r}
\]
 where $a: [r_e, \infty) \to \R$ is a smooth, bounded function also satisfying
\begin{equation}\label {condona}
 a'(r)\approx r^{-2}, \qquad a(3M)=0
 \end{equation}
 The exact formulas for $a$ and $X_2$ are not important, only the properties listed (see \cite{MMTT} for more details).

 The scalar function $\tq$ is (see \cite{MMTT})
 \[
  \tq(r) = \frac12 \weight \frac{1}{r^2}\partial_r\Bigl(r^2 a(r)\Bigr) + \delta_1 \chi_{\{ r > 5M/2\}} \frac{(r-3M)^2}{r^4}
 \]
 for some $\delta_1 \ll 1$, and it satisfies
\begin{equation}\label{condonq}
|\partial_r^\alpha \tq| \leq c_\alpha r^{-1-\alpha} \quad \alpha = 0,1,2
\end{equation}

 We will prove \eqref{diffQint} by splitting the integrating region into several parts.

 The error terms are easily handled near the event horizon. Indeed, due to \eqref{diff} we have
\begin{equation}\label{eherr}
\int_{\M_{[t_0,t_1]}^{[r_e, \frac{5M}2)}} |Q[g,X,q,m]\sqrt{|g|} - Q[g_S,X,q,m]\sqrt{|g_S|}| dt dx \lesssim \int_{\M_{[t_0,t_1]}^{[r_e, \frac{5M}2)}} \epsilon (|\partial u|^2 + u^2) dV_g
\end{equation}
 In particular this handles the error coming from $X_2$ and $m$.

 In the region $\frac{5M}2 < r < R_1$ we will prove that
\begin{equation}\label{diffQ}
\begin{split}
\int_{\M_{[t_0,t_1]}^{[\frac{5M}2, R_1)}}& |Q[g,X,q,m]\sqrt{|g|} -  Q[g_S,X,q,m]\sqrt{|g_S|}| dtdx \lesssim \\
& \epsilon \int_{\M_{[t_0,t_1]}^{[\frac{5M}2, R_1)}} Q[g_S,X,q,m]\sqrt{|g_S|} + t^{-1-\delta} |\partial u|^2 dt dx + \int_{\M_{[t_0,t_1]}^{ps}} \frac{\epsilon}{t}|\partial u|^2 dV_g
\end{split}
\end{equation}

In the region $r>R_1$ we will prove
\begin{equation}\label{intfar}
\begin{split}
\int_{\M_{[t_0,t_1]}^{[R_1, \infty)}} &|Q[g,X,q,m]\sqrt{|g|} - Q[g_S,X,q,m]\sqrt{|g_S|}| dt dx \lesssim \epsilon \int_{\M_{[t_0,t_1]}^{[R_1, \infty)}} r^{-1-\delta} |\partial u|^2 \\ &+ r^{-3-\delta}u^2
+ \la t-\rs\ra^{-1-\delta}|\overline{\partial} u|^2 +  r^{-2} \la t\ra^{-\frac12-\delta} \la t-\rs\ra^{-\frac12-\delta} |u|^2 dV_g
\end{split}
\end{equation}

By using Lemma~\ref{tangweight} and Lemma~\ref{Hardylemma} we can bound the last two terms on the RHS of \eqref{intfar} and we obtain
\begin{equation}\label{intfar2}
\begin{split}
\int_{\M_{[t_0,t_1]}^{[R_1, \infty)}} & |Q[g,X,q,m]\sqrt{|g|} - Q[g_S,X,q,m]\sqrt{|g_S|}| dt dx \lesssim \\ & \epsilon \Bigl(\int_{\M_{[t_0,t_1]}^{[\frac12 R_1, \infty)}} r^{-1-\delta} |\partial u|^2 + r^{-3-\delta}u^2 + |F\partial_t u| dV_g + E[t_0]\Bigr)
\end{split}
\end{equation}

 The desired estimate now follows from adding \eqref{eherr}, \eqref{diffQ}, and \eqref{intfar2}.

 We already dealt with \eqref{eherr} above. In the region $r>\frac{5M}2$ we will have to carefully analyze the contributions coming from $X_1$, $CK$, $q$ and $X_4$.
 By applying \eqref{defform} we immediately obtain
\begin{equation}\label{CK}
\sqrt{|g|} Q[g,CK] = C\frac12 (\partial_{t} \th^{\alpha\beta})\partial_\alpha u \partial_\beta u
\end{equation}

For $X_2$ we obtain from \eqref{QpiX2g}:
\begin{multline}
\frac12\Bigl(\sqrt{|g|}Q[g,X_2]-\sqrt{|g_S|}Q[g_S,X_2]\Bigr) =-\frac12 b(r)\partial_r \th^{\alpha\beta}\,
\partial_\alpha u\,\partial_\beta u \\
+\frac{b(r)}{r} \big(\th^{\alpha j}-\th^{\alpha r} \omega^j\big) \partial_\alpha u\, \partial_j u
+b^\prime(r) \th^{\alpha r}\partial_r u\,\partial_\alpha u- \frac{1}{2r^2} \partial_r(r^2 b(r)) \, \th^{\alpha\beta}\partial_\alpha u\,\partial_\beta u
\end{multline}

On the other hand, for a scalar function $q$ we calculate
\begin{equation}\label{qdiff}
\sqrt{|g|}Q[g,0, q, 0]-\sqrt{|g_S|}Q[g_S, 0, q, 0] = q (\th^{\alpha\beta}\partial_\alpha u\,\partial_\beta u) -\frac12 [(\sqrt{|g|}\Box_g - \sqrt{|g_S|}\Box_{g_S}) q] u^2
\end{equation}

 Since $b=\weight a$, and from the definition of $\tq$, we obtain that
\begin{multline}\label{X2diff}
\sqrt{|g|}Q[g,X_2, \tq, 0]-\sqrt{|g_S|}Q[g_S,X_2, \tq, 0] = -\frac12 b(r)\partial_r \th^{\alpha\beta}\,
\partial_\alpha u\,\partial_\beta u +\frac{b(r)}{r} \big(\th^{\alpha j}-\th^{\alpha r} \omega^j\big) \partial_\alpha u\, \partial_j u\\
+ b^\prime(r) \th^{\alpha r}\partial_r u\,\partial_\alpha u- \frac{M}{r^2} a(r) \, \th^{\alpha\beta}\partial_\alpha u\,\partial_\beta u
-\frac12 [(\sqrt{|g|}\Box_g - \sqrt{|g_S|}\Box_{g_S}) \tq] u^2
\end{multline}

 For the term involving $X_4$, a quick computation using \eqref{QpiX2g} and \eqref{qdiff} yields
\begin{multline}\label{X4diff}
\sqrt{|g|} Q[g, X_4, q_2, 0] - \sqrt{|g_S|} Q[g_S, X_4, q_2, 0] = -\frac12 \chi_{R_1}f(r)\partial_r \th^{\alpha\beta}\,
\partial_\alpha u\,\partial_\beta u +\frac{\chi_{R_1}f(r)}{r} \big(\th^{\alpha j}-\th^{\alpha r} \omega^j\big) \partial_\alpha u\, \partial_j u\\
+(\chi_{R_1}f)^\prime(r) \th^{\alpha r}\partial_r u\,\partial_\alpha u- \frac12 (\chi_{R_1}f)^\prime(r) \th^{\alpha\beta}\partial_\alpha u\,\partial_\beta u
-\frac12 [(\sqrt{|g|}\Box_g - \sqrt{|g_S|}\Box_{g_S}) q_2] u^2
\end{multline}

  Let us start with the analysis in the compact region $\frac{5M}2\leq r\leq R_1$.
 We first note that due to \eqref{CK}, \eqref{cpt2} and \eqref{derivmetdiff1}, we have
\begin{equation}\label{Kestcpt}
 |\sqrt{|g|} Q[g,CK] - \sqrt{|g_S|} Q[g_S,CK]| \lesssim \epsilon t^{-1-\delta} |\partial u|^2
\end{equation}

 For the $X_2$ term, we obtain from \eqref{X2diff}, \eqref{diff}, \eqref{metdiff}, \eqref{derivmetdiff1}, \eqref{condona} and Cauchy Schwarz that
\begin{equation}\label{X2estcpt}
 |\sqrt{|g|}Q[g,X_2, \tq, 0]-\sqrt{|g_S|}Q[g_S,X_2, \tq, 0]| \lesssim \epsilon \Bigl[t^{-1} |\partial u|^2 + (r-3M)^2 |\partial u|^2 + (\partial_r u)^2 + u^2 \Bigr]
\end{equation}

 Combining \eqref{Kestcpt}, \eqref{X2estcpt},  and the fact that $X_4$ and $q_2$ are supported in $\{r\geq R_1\}$, we obtain \eqref{diffQ} after integration.

 Let us now prove \eqref{intfar}. It is enough to prove the pointwise bound for $\{r\geq R_1\}$
 \begin{equation}\label{ptwseerrfar}
 \begin{split}
 |Q[g,X,q,m]\sqrt{|g|} - & Q[g_S,X,q,m]\sqrt{|g_S|}| \lesssim \epsilon \Bigl[r^{-1-\delta} |\partial u|^2 + r^{-3-\delta}u^2 \\ & + \la t-\rs\ra^{-1-\delta}|\overline{\partial} u|^2 +  r^{-2} \la t\ra^{-\frac12-\delta} \la t-\rs\ra^{-\frac12-\delta} |u|^2\Bigr]\sqrt{|g|}.
 \end{split}
 \end{equation}

 We start by estimating the contribution of $CK$ given in \eqref{CK}. Note that by Cauchy Schwarz and using \eqref{badpart}, \eqref{goodpart} we obtain
\[
 |(\partial_{t} \th^{\alpha\beta}) \partial_{\alpha} u \partial_{\beta} u| \lesssim \Bigl(|\partial_{t} \th^{\uL\uL}|(\underline{L}u)^2 + |\partial_t \th||\overline{\partial} u|^2\Bigr)\sqrt{|g|} \lesssim \epsilon\sqrt{|g|}\Bigl(r^{-1-\delta}(\underline{L}u)^2 + \la t-\rs\ra^{-1-\delta}|\overline{\partial} u|^2\Bigr)
\]
which suffices.

 For the term $Q[g, X_2, \tq, 0] + Q[g, X_4, q_2, 0]$ we estimate each term in \eqref{X2diff}, \eqref{X4diff}.

 For the first term we use \eqref{intrm}, \eqref{badpart}, \eqref{goodpart}, \eqref{derivmetdiff2} and Cauchy Schwarz:
\begin{multline}
| b(r)\partial_r \th^{\alpha\beta}\,\partial_\alpha u\,\partial_\beta u| + |\chi_{R_1} f(r)\partial_r \th^{\alpha\beta}\,\partial_\alpha u\,\partial_\beta u|
\lesssim \Bigl(|\partial_{r} \th^{\uL\uL}|(\underline{L}u)^2 + |\partial_r \th||\overline{\partial} u|^2\Bigr) \sqrt{|g|} \\
\lesssim \epsilon\sqrt{|g|}\Bigl(t^{-1-\delta}(\underline{L}u)^2 + \la t-\rs\ra^{-1-\delta}|\overline{\partial} u|^2 + r^{-1-\delta}|\partial u|^2\Bigr)
\end{multline}
Similarly the second term is estimated from the boundedness of $b$ and $f$, and \eqref{intrm}, \eqref{goodpart} and \eqref{metdiff}
\begin{equation}
|\frac{b(r)}{r} \big(\th^{\alpha j}-\th^{\alpha r} \omega^j\big) \partial_\alpha u\, \partial_j u| + |\frac{\chi_{R_1}f(r)}{r} \big(\th^{\alpha j}-\th^{\alpha r} \omega^j\big) \partial_\alpha u\, \partial_j u| \lesssim
r^{-1}|\th||\partial u|^2 \lesssim r^{-1-\delta}|\partial u|^2
\end{equation}
For the third term , we use that
\[
|b'| + |(\chi_{R_1}f)'| \lesssim r^{-1-\delta}
\]
in conjunction with \eqref{goodpart} and \eqref{metdiff}
\[
|b^\prime(r) \th^{\alpha r}\partial_r u\,\partial_\alpha u| + |(\chi_{R_1}f)^\prime(r) \th^{\alpha r}\partial_r u\,\partial_\alpha u| \lesssim r^{-1-\delta}|\partial u|^2
\]
Similarly for the fourth term,
\[
|\frac{M}{r^2} a(r) \, \th^{\alpha\beta}\partial_\alpha u\,\partial_\beta u| + |(\chi_{R_1}f)^\prime(r) \th^{\alpha\beta}\partial_\alpha u\,\partial_\beta u| \lesssim r^{-1-\delta}|\partial u|^2
\]

We are left with estimating the terms involving the Lagrangian corrections $q=\tq + q_2$. We have
\[
|\sqrt{|g|} \Box_g q - \sqrt{|g_S|} \Box_S q| \lesssim |\th^{\alpha\beta} \partial_{\alpha}\partial_{\beta} q| + |\partial_{\alpha}\th^{\alpha\beta}\partial_{\beta} q|
\]

 For the first term, we have from \eqref{condonq} and \eqref{imprvinfty} that
\[
 |\partial_{\alpha}\partial_{\beta} q| \lesssim r^{-3}
\]
which combined with \eqref{metdiff}, \eqref{goodpart} yields
\begin{equation}\label{q2estfar}
|\th^{\alpha\beta} \partial_{\alpha}\partial_{\beta} q| \lesssim \epsilon r^{-3-\delta}
\end{equation}

 For the second term, we use \eqref{condonq}, \eqref{derivmetdiff1}, and \eqref{goodpart} :
\[
|\partial_{\alpha}\th^{\alpha\beta}\partial_{\beta} q| \lesssim r^{-2}|\partial\th| \lesssim \epsilon \sqrt{|g|} (r^{-3-\delta} + r^{-2} \la t\ra^{-\frac12-\delta} \la t-\rs\ra^{-\frac12-\delta})
\]

 \eqref{intfar} now follows in the region $\{r\geq R_1\}$.

 This concludes the proof of \eqref{wkdecay}. The transition to the stronger estimate \eqref{LES} is now straightforward, following the methods of \cite{MT} and \cite{MMTT}. Indeed, it is enough to improve the estimate for
 $w =\chi_{R_1} u$, which satisfies the equation
\[
\Box_g w = \chi_{R_1} F + [\Box_g ,\chi_{R_1}] u := \tilde F
\]

 For some fixed dyadic number $\rho\geq R_1$ we use a multiplier of the form
\[
 X_{\rho} = C\partial_t + \tf(r)\partial_r , \qquad q_{\rho} = \frac{\tf(r)}r , \qquad \tf(r)= \frac{r}{r+\rho}
\]

 A quick computation yields (see also \cite{MT} Proposition 8):
\[
\left. BDR^{g}[u]\right|_{t = t_i} \approx \|\partial  u(t_i)\|_{L^2}^2
\]
\[
\int_{\M_{[t_0,t_1]}} Q[g_S,X_{\rho},q_{\rho},0] dV_g \gtrsim \| \la r\ra^{-\frac12} w\|_{L^2 ([t_0, t_1] \times A_{\rho})}^2 + \| \la r\ra^{-\frac32} w\|_{L^2 ([t_0, t_1] \times A_{\rho})}^2
\]
 On the other hand, the analogue estimate to \eqref{ptwseerrfar} also holds in this case by the same proof. We thus have
 \begin{equation}\label{ptwseerrfar2}
 \begin{split}
 \Bigl |Q[g,X_{\rho},q_{\rho},0]\sqrt{|g|} - & Q[g_S,X_{\rho},q_{\rho},0]\sqrt{|g_S|}\Bigr| \lesssim \epsilon \Bigl[r^{-1} |\partial w|^2 + r^{-3}w^2 \\ & + \la t-\rs\ra^{-1-\delta}|\overline{\partial} w|^2 +  r^{-2} \la t\ra^{-\frac12-\delta} \la t-\rs\ra^{-\frac12-\delta} |w|^2\Bigr]\sqrt{|g|}
 \end{split}
 \end{equation}
and by Lemmas \ref{tangweight} and \ref{Hardylemma} we get
\[\begin{split}
\int_{\M_{[t_0,t_1]}} \Big|Q[g,X_{\rho},q_{\rho},0]\sqrt{|g|} -  Q[g_S,X_{\rho},q_{\rho},0]\sqrt{|g_S|}\Big| dV_g \lesssim \epsilon E[t_0] + \epsilon \|w\|_{LE[t_0, t_1]}^2 + \int_{\M_{[t_0,t_1]}} \!\!\!\!\!\!\!\!\!\!\!\! |\tilde F|(|\partial u| + r^{-1}|u|) dV_g
\end{split}\]
The desired estimate \eqref{LES} follows after taking the supremum over the dyadic numbers $\rho$.
\end{proof}

\newsection{Commuting with derivatives and vector fields}
We will now prove the equivalent of Theorem~\ref{LE} for higher order derivatives and vector fields.
The goal of this section is to prove the following:
\begin{theorem}\label{higherLE} Let $u$ solve \eqref{maineeqn}, where $g$ is a Lorentzian metric satisfying the conditions from Section 3.
 Then for some constant $C_N$ independent of $t_0$, $t_1$ we have:
\begin{equation}\label{higherLES}
 \|\partial  u_{\leq N}\|_{L^{\infty}[t_0, t_1] L^2} + \|u_{\leq N}\|_{LE_S^1[t_0, t_1]} \leq C_N t_1^{C_N\epsilon} \|\partial u_{\leq N}(t_0)\|_{L^2}
\end{equation}
\end{theorem}
\begin{proof}
 We start by proving the following estimate for $W$:
\begin{lemma}\label{wavecoordscommut}
 If $W$ is defined as in \eqref{W-def}, then
\begin{equation}\label{wavecoords2}
W_{\Lambda} = S^Z(1) \partial u_{\Lambda} + W_{\Lambda}^{low}
\end{equation}
where the lower order term $W_{\Lambda}^{low}$ satisfies
\begin{equation}\label{Wlot3M}
|W_{\Lambda}^{low}| \lesssim |u_{\leq\Lambda}|, \qquad r\leq R_1
\end{equation}
\begin{equation}\label{Wlotfar}
|W_{\Lambda}^{low}| \lesssim |\partial u_{<\Lambda}| + (\frac1{r^2}+ \frac{t^{1+\delta}}{r^{1+\delta}\la t-\rs\ra^{1/2+\delta}}) |u_{\leq \Lambda}|, \qquad R_1^* \leq r^*
\end{equation}
\end{lemma}
\begin{proof}
 We can write
\[
W^{\beta} = \partial_{\alpha} h^{\alpha\beta} + \frac12 \Bigl(g_S^{\alpha\beta} \frac{\partial_{\alpha}|g_S|}{|g_S|} - g^{\alpha\beta} \frac{\partial_{\alpha}|g|}{|g|}\Bigr)
\]

 We clearly have that
\[
(\partial_{\alpha} h^{\alpha\beta})_{\Lambda} = \partial_{\alpha} h^{\alpha\beta}_{\Lambda} + S^Z(1) \partial h_{<\Lambda}
\]

 On the other hand, the second term satisfies
\[
\Bigl(g_S^{\alpha\beta} \frac{\partial_{\alpha}|g_S|}{|g_S|} - g^{\alpha\beta} \frac{\partial_{\alpha}|g|}{|g|}\Bigr)_{\Lambda} = S^Z(1) \Bigl(((g_S)_{\leq \frac{|\Lambda|}2} + h_{\leq \frac{|\Lambda|}2}) \partial h_{\leq \Lambda} + ((\partial g_S)_{\leq \frac{|\Lambda|}2} + \partial h_{\leq \frac{|\Lambda|}2}) h_{\leq \Lambda}\Bigr)
\]

 We also know that
\[
\Bigl|(g_S)_{\leq \frac{|\Lambda|}2}\Bigr| \lesssim 1, \quad  \Bigl|(\partial g_S)_{\leq \frac{|\Lambda|}2}\Bigr| \lesssim r^{-2}
\]

 The conclusion now follows by using \eqref{metricdecay} and \eqref{metricvsu}.
\end{proof}

  Our next remark is that, due to the (proof of) Lemma~\ref{wavecoordscommut} combined with \eqref{metricdecay}, \eqref{t-derivest} and the fact that $\partial_t g_S = 0$ we obtain that
\begin{equation}\label{tderivglobal}
|\partial_t h_{\Lambda}| + |\partial_t W_{\Lambda}| \lesssim \frac{\epsilon}{\la t\ra}, \qquad |\Lambda|\leq  \frac{N}2
\end{equation}

 We will now prove the following very useful elliptic estimate:
\begin{lemma}\label{elliptic}
 Let $u$ be as in the theorem above, and $J$ be a multiindex with $0\leq |J|\leq N-1$. Let $t_0\leq t\leq t_1$. We have
\begin{equation}\label{ell-est} \begin{split}
\|\partial^2 u_J(t)\|_{L^2} + \|\partial u_J\|_{LE_S^1[t_0, t]} & \lesssim \|\partial \partial_{\leq 1}u_J(t_0)\|_{L^2} + \|\partial_t \partial_{\leq 1} u_{\leq |J|}(t)\|_{L^2} + \|\partial u_{\leq |J|}(t)\|_{L^2} \\ & + \|\partial_t u_{\leq |J|}\|_{LE_S^1[t_0, t]} +  \|u_{\leq |J|}\|_{LE_S^1[t_0, t]}
\end{split} \end{equation}
\end{lemma}
\begin{proof}
 We start with the case $|J|=0$. Clearly $u$ satisfies the equation
\begin{equation}\label{ell}
\mathscr{L}_{0} u = \mathscr{L}_t u + \Box_g u
\end{equation}
where
\begin{equation}\label{ellt}
\mathscr{L}_t u = -\frac{1}{\sqrt{|g|}}\partial_t (\sqrt{|g|} g^{tj}\partial_{j} u) -\frac{1}{\sqrt{|g|}}\partial_{j} (\sqrt{|g|} g^{tj}\partial_t u) -\frac{1}{\sqrt{|g|}}\partial_t (\sqrt{|g|} g^{tt}\partial_{t} u)
\end{equation}
\begin{equation}\label{ell0}
\mathscr{L}_0 u = \frac{1}{\sqrt{|g|}}\partial_i (\sqrt{|g|} g^{ij}\partial_{j} u)
\end{equation}

The most important thing here is that $\mathscr{L}_0$ is a second order differential operator which is elliptic in the region $r>2M+\varepsilon$ for some $\varepsilon \ll M$.

Consider now smooth cutoffs $\chi_{eh}$ and $\chi_{out}$, so that $\chi_{eh} = 1$ when $r_e\leq r\leq 2M+2\varepsilon$, $\chi_{eh} = 0$ when $r\geq 2M+3\varepsilon$, and $\chi_{out}=1$ when $2M+2\varepsilon\leq r$, $\chi_{out}=0$ when $r<2M+\varepsilon$.

We will first prove that
\begin{align}\label{spatderivout1}
\|\chi_{out}\partial^2 u(t)\|_{L^2}  &\lesssim  \|\partial_t \partial_{\leq 1} u(t)\|_{L^2} +\|\partial u(t)\|_{L^2} + \|\Box_g u(t)\|_{L^2}\\
\label{spatderivout2}
\|\chi_{out}\partial u\|_{LE_S^1[t_0, t]} &\lesssim \|\partial \partial_{\leq 1} u(t_0)\|_{L^2} + \|\partial_t u\|_{LE_S^1[t_0, t]} +  \|u\|_{LE_S^1[t_0, t]}  + \bigl\|\Box_g u\|_{{LE[t_0, t]}}
\end{align}

To prove \eqref{spatderivout1} we use standard elliptic estimates on constant time slices. Indeed, multiplying \eqref{ell0} by $\partial_{k} (\chi_{out}^2 \partial_{k} u)$ and integrating by parts yields
\[\begin{split}
\int \chi_{out}^2 g^{ij}(\partial_{jk}u)(\partial_{ik}u) \sqrt{|g|} dx =  \int (\mathscr{L}_0 u) \partial_{k} (\chi_{out}^2 \partial_{k} u) \sqrt{|g|} dx - \int \partial_{i} (\chi_{out}^2) \partial_k (\sqrt{|g|} g^{ij}\partial_{j} u) (\partial_k u) dx
\end{split}\]
 There are no boundary terms at infinity since $u$ is compactly supported.

 We can now use \eqref{ell} and Cauchy-Schwarz to bound the right hand side. The second term can easily be estimated by using \eqref{metricdecay}
\[\begin{split}
& \Bigl|\int \partial_{i} (\chi_{out}^2) \partial_k (\sqrt{|g|} g^{ij}\partial_{j} u) (\partial_k u) dx \Bigl| \leq \int \Bigl|\partial_{i} (\chi_{out}^2) \partial_k (\sqrt{|g|} g^{ij})(\partial_{j} u)(\partial_k u)\Bigr| dx \\
& + \frac12 \int \Bigl|\partial_j [\partial_{i} (\chi_{out}^2) (\sqrt{|g|} g^{ij})] (\partial_{k} u)^2\Bigr| dx \lesssim \|\partial u(t)\|_{L^2}^2
\end{split}\]

 On the other hand, by \eqref{ellt} we have
\[\begin{split}
|\mathscr{L}_t u| \lesssim |\partial \partial_t u| + |\partial u|,\qquad\quad
|\partial_{k} (\chi_{out}^2 \partial_{k} u)| \lesssim \chi_{out}^2 |\partial^2 u| + |\chi_{out}' \partial u|
\end{split}\]

We now obtain by Cauchy-Schwarz that
\[\begin{split}
\Bigl| \int (\mathscr{L}_0 u) \partial_{k} (\chi_{out}^2 \partial_{k} u) \sqrt{|g|} dx \Bigr| \leq C_{\td}\Bigl(\|\partial_t \partial_{\leq 1} u(t)\|_{L^2}^2 +\|\partial u(t)\|_{L^2}^2 + \|\Box_g u(t)\|_{L^2}^2\Bigr) + \td \|\chi_{out}\partial^2 u(t)\|_{L^2}^2
\end{split}\]
for any small $\td$ and some large $C_{\td}$. After summing over $k$, using the ellipticity of $\mathscr{L}_0$ on the support of $\chi_{out}$ and taking $\td$ small enough we can absorb the last term to the left hand side, and \eqref{spatderivout1} follows.

 We are left with estimating the space-time term, \eqref{spatderivout2}.

 We start by estimating the lower order term on the left hand side. Let $\chi_{ps}$ denote a cutoff function near $r=3M$. Away from the photosphere we already have the better estimate
\begin{equation}\label{lotout}
\|(1-\chi_{ps})r^{-1} \partial u\|_{LE[t_0, t]} \lesssim \|(1-\chi_{ps}) \partial u\|_{LE[t_0, t]} \lesssim \|\partial u\|_{LE[t_0, t]}
\end{equation}

 Near the photonsphere, we multiply \eqref{ell} by $\chi_{ps}^2 u$, integrate by parts and use Cauchy-Schwarz. We obtain
\begin{equation}\label{lotps}
 \| \chi_{ps} \partial u\|_{L^2[t_0, t]L^2} \lesssim \|\partial u(t_0)\|_{L^2} + \|\partial u(t)\|_{L^2} + \|\partial_t u\|_{LE_S^1[t_0, t]} +  \|u\|_{LE_S^1[t_0, t]} + \|\chi_{ps} \Box_g u\|_{L^2[t_0, t]L^2}
\end{equation}
which takes care of the lower order term in $\|\partial u\|_{LE_S^1[t_0, t]}$ near the photosphere and hence everywhere.

 For $R>4M$ dyadic we can multiply \eqref{ell} by $\frac1R \partial_i \chi_R^2 \partial_i u$, integrate by parts and apply Cauchy-Schwarz to obtain
\[\begin{split}
\|R^{-1/2} \chi_R \partial\partial_i u\|_{L^2[t_0, t]L^2}  \leq C \Bigl(\|\partial_t u\|_{LE_S^1[t_0, t]} +  \|u\|_{LE_S^1[t_0, t]} + \|\Box_g u\|_{LE(r>7M/2)}\Bigr)  + \frac12 \|R^{-1/2} \chi_R\partial\partial_i u\|_{L^2[t_0, t]L^2}
\end{split}\]
where the constant $C$ is independent of $R$. After absorbing the last term to the LHS, summing over $i$ and taking the supremum over $R$ we obtain
\begin{equation}\label{hotout}
\| \partial^2 u\|_{LE(r>4M)}\| \lesssim \|\partial_t u\|_{LE_S^1[t_0, t]} +  \|u\|_{LE_S^1[t_0, t]} + \|\Box_g u\|_{LE(r>7M/2)}
\end{equation}

 In the compact region $r<4M$, we first multiply \eqref{ell} by $\partial_i \chi_{M}^2 (r-3M)^2 \partial_i u$, where $\chi_{M}$ is a cutoff function supported in $2M+\varepsilon \leq r\leq 5M$ which is identically $1$ when $2M+2\varepsilon \leq r\leq 4M$. After integrating by parts and applying Cauchy-Schwarz we obtain
\begin{equation}\label{hotps}
\|\chi_{M}(r-3M)\partial^2 u\|_{L^2L^2} \lesssim  \|\partial_t u\|_{LE_S^1[t_0, t]} +  \|u\|_{LE_S^1[t_0, t]} + \|\chi_{M} (r-3M) \Box_g u\|_{L^2[t_0, t]L^2}
\end{equation}

 Finally, we can multiply \eqref{ell} by $\partial_r \chi_{M}^2 \partial_r u$, integrate by parts and apply Cauchy-Schwarz to obtain
\begin{equation}\label{hotr}\begin{split}
 \|\chi_{M}\partial \partial_r u\|_{L^2[t_0, t]L^2} & \lesssim \|\partial_t \partial_{\leq 1} u\|_{L^{\infty}[t_0, t]L^2} + \|\chi_{M} \partial u\|_{L^2[t_0, t]L^2} + \|\partial_t u\|_{LE_S^1[t_0, t]} +\\ &  \|u\|_{LE_S^1[t_0, t]}  + \|\chi_{M} \Box_g u\|_{L^2[t_0, t]L^2} + \|\chi'_{M} \Box_g u\|_{L^2[t_0, t]L^2}
\end{split}\end{equation}

 The proof of \eqref{spatderivout2} is complete by \eqref{lotout}, \eqref{lotps}, \eqref{hotout}, \eqref{hotps}, and \eqref{hotr}.

  We now estimate the part near the event horizon. Let $v=\chi_{eh} u$. The function $v$ satisfies the equation
\[
 \Box_g v = [\Box_g, \chi_{eh}] u + \chi_{eh}\Box_g u
\]

 By using, for example, Lemma 4.4 of \cite{MTT} we obtain that
\[
\|\partial^2 v\|_{L^{\infty}L^2} + \|\partial v\|_{H^1} \lesssim \|\partial v(t_0)\|_{H^1} + \|\Box_g v\|_{H^1}
\]
and as a consequence
\begin{equation}\label{spatderivin}
\begin{split}
\|\chi_{eh}\partial^2 u(t)\|_{L^2} + \|\chi_{eh}\partial u\|_{LE_S^1[t_0, t]}\lesssim \|\partial\partial_{\leq 1} u(t_0)\|_{L^2} + \|\partial u\|_{H^1(2M+2\varepsilon<r<2M+3\varepsilon)} + \|\Box_g u\|_{H^1(r<2M+3\varepsilon)}
\end{split}\end{equation}

 The middle term on the RHS is supported in the region where $\chi_{out}=1$, so it can be controlled by \eqref{spatderivout2}.The desired conclusion \eqref{ell-est} when $|J|=0$ now follows from \eqref{spatderivout1}, \eqref{spatderivout2} and \eqref{spatderivin}

 Assume now that $J\!=\!(i, j, k)$ is a multi-index with $1\!\leq |J|\! = i+3j\!+3k\leq\! N$ and proceed by induction on $|J|$.
 We have
\[
\mathscr{L}_0 u_J = \mathscr{L}_t u_J + \Box_g u_J
\]

In order to estimate the last term, we observe that since $\Box_g u =0$ we have
\[
|\Box_g u_J| \lesssim (|(g_S)_{\leq |J|}| + |h_{\leq |J|/2}|) \cdot |\partial u_{\leq |J|}| + |\partial u_{\leq |J|/2}| \cdot |h_{\leq |J|}|
\]
Due to \eqref{u-decay} and \eqref{metricdecay} we obtain that
\begin{equation}\label{locommut}
|\Box_g u_J| \lesssim |\partial u_{\leq |J|}| + \frac{1}{r} |u_{\leq |J|}|
\end{equation}

Since by Hardy's inequality we have
\begin{equation}\label{elllot}
\|\frac{1}{r} u_{\leq |J|}(t)\|_{L^2} + \|\frac{1}{r} u_{\leq |J|}\|_{LE[t_0, t]} \lesssim \| \partial u_{\leq |J|}(t)\|_{L^2} + \|\partial u_{\leq |J|}\|_{LE[t_0, t]}
\end{equation}
and one can prove elliptic estimates as above. Away from the event horizon, using \eqref{spatderivout1} and \eqref{spatderivout2} applied to $u_J$ in conjunction with \eqref{locommut} and \eqref{elllot}, we obtain
\begin{equation}\label{hghspatderivout} \begin{split}
& \|\chi_{out}\partial^2 u_J\|_{L^{\infty}[t_0, t]L^2} + \|\chi_{out}\partial u_J\|_{LE_S^1[t_0, t]} \lesssim \\ & \|\partial_t \partial u_{\leq |J|}\|_{L^{\infty}[t_0, t]L^2} + \|\partial u_{\leq |J|}\|_{L^{\infty}[t_0, t]L^2} + \|\partial_t u_{\leq |J|}\|_{LE_S^1[t_0,t]} +  \|u_{\leq |J|}\|_{LE_S^1[t_0,t]}
\end{split} \end{equation}
which suffices away from the event horizon by the induction hypothesis.

On the other hand, near the event horizon, let $v=\chi_{eh} u$ as before. Our goal will be to prove that
\begin{equation}\label{hghspatderivin}\begin{split}
\|\partial^2 v_J\|_{L^{\infty}[t_0, t]L^2} + \|\partial v_J\|_{LE_S^1[t_0, t]} & \lesssim \|\partial^2 v_J(t_0)\|_{L^2} + \epsilon \|\partial v_J\|_{LE_S^1[t_0, t]} + \\ & \|u_{\leq |J|}\|_{LE_S^1[t_0,t]} + \| \chi_{out} \partial u_J\|_{LE_S^1[t_0, t]}
\end{split}\end{equation}

For small enough $\epsilon$ the second term on the RHS can be absorbed on the left, while the last term on the RHS can be estimated from \eqref{hghspatderivout}. The conclusion \eqref{ell-est} now follows from \eqref{hghspatderivout} and \eqref{hghspatderivin}.

By Lemma 4.4 of \cite{MTT} we get
\begin{equation}\label{lemmaMTTJ}
\|\partial^2 v_{J}\|_{L^{\infty}[t_0, t]L^2} + \|\partial v_{J}\|_{LE_S^1[t_0, t]} \lesssim \|\partial^2 v_{J}(t_0)\|_{L^2} + \|\Box_g v_{J}\|_{H^1[t_0,t]}
\end{equation}
Clearly $v_J$ satisfies
\begin{equation}\label{incommut}
 \Box_{g} v_J = [\Box_{g}, Z^J] v + (\Box_g v)_J
\end{equation}
Since $[\Box_g, \chi_{eh}]$ is supported in the where $\chi_{out}=1$, the second term on the RHS of \eqref{incommut} can be controlled:
\[
\|(\Box_g v)_J\|_{H^1[t_0,t]} \lesssim \| \chi_{out} \partial u_J\|_{LE_S^1[t_0, t]} + \| \chi_{eh} (\Box_g u)_J\|_{H^1[t_0,t]}
\]
 For the first term we have
\[
\|[\Box_{g}, Z^J] v\|_{H^1[t_0,t]} \lesssim \|\partial v_J\|_{H^1[t_0,t]} + \|v_{\leq |J|}\|_{H^1[t_0,t]}
\]
The second term on the RHS is controlled by the induction hypothesis, so \eqref{hghspatderivin} will follow if we prove
\begin{equation}\label{boxestin}
\|\partial v_{J}\|_{H^1[t_0,t]} \lesssim \|\partial^2 v_J(t_0)\|_{L^2} + \|v_{\leq |J|}\|_{H^1[t_0,t]} + \|(\Box_g v)_J\|_{H^1[t_0,t]}
\end{equation}

 Let us start with the case of no vector fields, namely $j=k=0$.
We will mimic the proof of Lemma 4.4 of \cite{MTT} in our case. The first key observation is that, since
\[
[\Box_{g_S}, \partial_t] = [\Box_{g_S}, \partial_{\omega}] = 0
\]
we obtain that
\begin{equation}\label{tomega}
\Bigl| [\Box_g, \partial_{t,\omega}]u \Bigr| \lesssim \epsilon |\partial_{\leq 2} u|
\end{equation}

 Let $v_{\tJ}$ denote the set $\{ \partial_t^{i_1} \partial_{\omega}^{i_2} v, i_1+i_2\leq i \}$. By Lemma 4.4 of \cite{MTT} we get
\[
 \|\partial v_{\tJ}\|_{H^1[t_0, t]} \lesssim \|\partial^2 v_{\tJ}(t_0)\|_{L^2} + \|\Box_g v_{\tJ}\|_{H^1[t_0,t]}
\]
and by using \eqref{incommut} and \eqref{tomega}
\[
\|\Box_g v_{\tJ}\|_{H^1[t_0,t]} \lesssim \epsilon \|\partial^2 v_{\tJ} \|_{H^1[t_0, t]} +\|\partial_{\leq |\tJ|} v\|_{H^1[t_0, t]} + \|(\Box_g v)_{\tJ}\|_{H^1[t_0,t]}
\]

 After absorbing the first term on the RHS to the left, we obtain \eqref{boxestin} when all derivatives are either time or angular derivatives.

 On the other hand, the $\partial_r$ derivatives do not commute nicely with the Schwarzschild metric, so we need to use the red shift effect near the event horizon. We add $\partial_r$ derivatives one by one by induction. More specifically, we use the proof of (4.21) from \cite{MTT}, which asserts that for all functions $w$ supported near the event horizon and $\gamma_1>0$ on the support of $w$, then
\begin{equation}\label{rderivin}\begin{split}
\|\partial \partial_r w\|_{L^{\infty}[t_0, t]L^2} + \| \partial_r w\|_{H^1[t_0, t]} & \lesssim \|\partial^2 w(t_0)\|_{L^2} + \| (\Box_g + \gamma_1 \partial_r) \partial_r w\|_{H^1[t_0, t]}
\end{split}\end{equation}

 Assume now that $v_J = \partial_{r}^l u_{\tJ}$ with $l>0$. After commuting we obtain
\[
\Box_g v_J = - l (\partial_r g^{rr}) \partial_{r}^{l+1} v_{\tJ} + O(\epsilon) \partial_{\leq |J|+1} v + \partial_{\leq |J|} v
\]
 The key observation is that, since the metric $g$ is $O(\epsilon)$ perturbation of $g_S$ in $C^1$ and $\partial_r g_S^{rr}(2M)>0$ one has
\[
l (\partial_r g^{rr}) > 0
\]
near $r=2M$. By applying \eqref{rderivin} to $w = \partial_{r}^{l-1} v_{\tJ}$ with $\gamma_1 = l (\partial_r g^{rr})$ we get
\[
\|\partial v_{J}\|_{H^1[t_0, t]} \lesssim \|\partial v_{J}(t_0)\|_{L^2} + \epsilon \|\partial_{\leq |J|} v\|_{H^1[t_0, t]} + \|\partial_{\leq |J|-1} v\|_{H^1[t_0, t]}
\]
 The middle term on the RHS can be absorbed to the left, and \eqref{boxestin} in the case $j\!=\!k\!=\!0$ follows by induction.

 For the general case, the key observation is that commuting $\Box_{g}$ with $\Omega$ or $S$ yields lower order terms, since each one of these vectors counts as three derivatives, while the commutator is of order at most $2$. We can now apply \eqref{boxestin} to $i$ derivatives of the function $v_{\Omega S}:=\Omega^j S^k v$ . We obtain
\[\begin{split}
\|\partial v_{J}\|_{H^1[t_0,t]} = & \|\partial \partial^i (v_{\Omega S})\|_{H^1[t_0,t]}  \lesssim \|\partial^2 v_{J}(t_0)\|_{L^2} + \|v_{\leq |J|}\|_{H^1[t_0,t]} + \|\partial^i(\Box_{g} v_{\Omega S})\|_{H^1[t_0, t]} \\ & \lesssim \|\partial^2 v_{J}(t_0)\|_{L^2} + \|v_{\leq |J|}\|_{H^1[t_0,t]} + \|(\Box_g v)_J\|_{H^1[t_0,t]}
\end{split}\]
where in the last line we used that
\[
[\Box_g, \Omega^jS^k] v \in S^Z(1) v_{\leq 3j+3k}
\]
This finishes the proof of \eqref{boxestin}.
\end{proof}

 The purpose of the above lemma is to replace spatial derivatives by time derivatives when performing the commutations. This is crucial near the photosphere, where the time derivative comes with an extra decay factor of $t^{-1}$ compared to spatial derivatives.

 The following Klainerman-Sideris type estimate will provide better control of $\partial^2 u$.

\begin{lemma}\label{KS1}  Assume that $g$ is a Lorentzian metric satisfying the conditions \eqref{metricdecay}.

 Then for any multi index $|\Lambda|\leq \frac{N}2$ we have when $r\geq 2R_1$:
\begin{equation}\label{improvetwoderivs}
|\partial^2 u_{\Lambda}| \lesssim \frac{t}{r\la t-\rs\ra} |\partial u_{\leq |\Lambda|+3}| + \frac{t}{\la t-\rs\ra} |(\Box_g u)_{\leq|\Lambda|}|
\end{equation}
\end{lemma}

\begin{proof}

 For the duration of this proof, we will let $\Box$ denote the Minkowski d'Alembertian with respect to the $x^*$ coordinates,
\[
\Box = -\partial_t^2 + \sum_{i=1}^3 \partial_{x_i^*}^2
\]

 Let us start by proving the case $\Lambda=0$. Since
 \[
 \partial_{A,B} \approx \frac1{r} \Omega, \qquad \partial_t = \frac1{t} S - \frac{\rs}{t}\partial_{\rs}
 \]
 we get that
 \begin{equation}\label{KStwoderivs}
 |\partial^2 u| \lesssim \frac{1}r |\partial u_{\leq 3}| + |\partial_{\rs}^2 u|
 \end{equation}

  On the other hand, we have for $\partial_{\rs}^2 u$ (see, for example, \cite{KS}):
 \begin{equation}\label{KSest}
 |\partial_{\rs}^2 u| \lesssim \frac1{\la t-\rs\ra} |\partial u_{\leq 3}| + \frac{t}{\la t-\rs\ra} |\Box u|
 \end{equation}

 We estimate
\[
|(\Box_g - \Box) u| \lesssim |\Box_g u - \Box_{S} u| + |\Box_{S} u - \Box u| \lesssim |h||\partial^2 u| + |\partial h||\partial u| + \frac{1}{r} |\partial u_{\leq 1}|
\]
and taking into account \eqref{metricdecay}, \eqref{KStwoderivs} we get
\[ \begin{split}
\frac{t}{\la t-\rs\ra} |\Box u| &  \lesssim \frac{t}{\la t-\rs\ra} |\Box_g u| + \frac{\epsilon}{\la t-\rs\ra^{1/2}} |\partial^2 u| + \frac{\epsilon t}{r \la t-\rs\ra}|\partial u| + \frac{t}{r\la t-\rs\ra} |\partial u_{\leq 3}| \\
& \lesssim \frac{t}{\la t-\rs\ra} |\Box_g u| + \frac{\epsilon}{\la t-\rs\ra^{1/2}} |\partial_{\rs}^2 u| + \frac{t}{r\la t-\rs\ra} |\partial u_{\leq 3}|
\end{split} \]

We now plug back into \eqref{KSest} and absorb the $\partial_{\rs}^2 u$ term to the left side. We obtain
\[
|\partial_{\rs}^2 u| \lesssim \frac{t}{r\la t-\rs\ra} |\partial u_{\leq 3}| + \frac{t}{\la t-\rs\ra} |\Box_g u|
\]

which in conjunction with \eqref{KStwoderivs} implies \eqref{improvetwoderivs} for $\Lambda = 0$.

 The proof now follows by induction. Let us assume that \eqref{improvetwoderivs} holds for all multiindices with  $|\Lambda_1| < |\Lambda|$. We will prove \eqref{improvetwoderivs} for $\Lambda$.

 By applying \eqref{KSest} to $u_{\Lambda}$ we obtain
\begin{equation}\label{KSLbdest}
|\partial_{\rs}^2 u_{\Lambda}| \lesssim \frac1{\la t-\rs\ra} |\partial u_{\leq |\Lambda|+3}| + \frac{t}{\la t-\rs\ra} |\Box u_{\Lambda}|
\end{equation}

 After commuting with the vector fields we obtain
 \[
|\Box u_{\Lambda}| \lesssim |(\Box u)_{\Lambda}| + |(\Box u)_{\leq |\Lambda|-3}|
 \]

Clearly
\[
|(\Box_S u - \Box u)_{\Lambda}| \lesssim r^{-1} (|\partial^2 u_{\Lambda}| + |\partial u_{\leq |\Lambda|}|)
\]
and using \eqref{metricdecay}
\[
|(\Box_g u - \Box_S u)_{\Lambda}| \lesssim \frac{\epsilon\la t-\rs\ra^{1/2}}{t} (|\partial^2 u_{\Lambda}| + |\partial^2 u_{\leq|\Lambda|-1}|) + \frac{\epsilon}{r}|\partial u_{\leq |\Lambda|}||
\]

 The last two inequalities combined with \eqref{KSest} imply
\[\begin{split}
\frac{t}{\la t-\rs\ra} |(\Box u)_{\Lambda}| \lesssim & \frac{\epsilon}{\la t-\rs\ra^{1/2}} (|\partial^2 u_{\Lambda}| +|\partial^2 u_{\leq|\Lambda|-1}|) + \frac{\epsilon t}{r \la t-\rs\ra}|\partial u_{\leq |\Lambda|}|| \\ & + \frac{t}{r\la t-\rs\ra} |\partial u_{\leq |\Lambda|+3}|  + \frac{t}{\la t-\rs\ra} |(\Box_g u)_{\Lambda}|
\end{split}\]

 A similar estimate holds for $|(\Box u)_{\leq |\Lambda|-3}|$. By using \eqref{KStwoderivs} (applied to $u_{\Lambda}$) we can plug back into \eqref{KSLbdest} and obtain
\[
 |\partial_{\rs}^2 u_{\Lambda}| \lesssim \frac{t}{r\la t-\rs\ra} |\partial u_{\leq |\Lambda|+3}| + \frac{t}{\la t-\rs\ra} |(\Box_g u)_{\Lambda}|
\]
which combined with \eqref{KStwoderivs} finishes the proof of \eqref{improvetwoderivs}.
\end{proof}

 Coming back to proving \eqref{higherLES}, we will proceed by induction. We first prove the result for derivatives only. Let us start with commuting with the time derivative. A simple computation gives, as $\partial_t$ and $\Box_{g_S}$ commute:
\[
\Box_g \partial_t u = F_1 + \partial_t \Box_g u
\qquad\text{where}\qquad
F_1 =  (\partial_t h^{\alpha\beta}) \partial_{\alpha}\partial_{\beta} u + (\partial_t W^{\alpha}) \partial_{\alpha} u
\]

 Let
\[
 E_1(t) = \|\partial \partial_t u(t)\|_{L^2}^2 + \|\partial_t u\|_{LE_S^1[t_0, t]}^2
\]
 By \eqref{LES} applied to $\partial_t u$ we obtain:
\begin{equation}\label{partialtcommut}
E_1(t) \lesssim \int_{\M_{[t_0,t]}} |F_1| (|\partial \partial_t u| + r^{-1} |\partial_t u|)dV_g + \int_{\M_{[t_0,t]}^{ps}}\frac{\epsilon}{\la\tau\ra} |\partial \partial_t u|^2 dV_g + \|\partial \partial_t u(t_0)\|_{L^2}^2
\end{equation}

Due to \eqref{t-derivest} we have
\[
 |\partial_t h^{\alpha\beta}| \lesssim \epsilon\la t\ra^{-1}
\]
and thus by taking Lemma~\ref{elliptic} into account
\begin{equation}\label{onetderiv1}\begin{split}
\int_{\M_{[t_0,t]}} |(\partial_t h^{\alpha\beta}) \partial_{\alpha}\partial_{\beta} u| (|\partial \partial_t u| + r^{-1} |\partial_t u|) dV_g & \lesssim \int_{\M_{[t_0,t]}} \frac{\epsilon} {\la\tau\ra} |\partial\partial_{\leq 1} u|^2 dV_g \\ & \lesssim \int_{\M_{[t_0,t]}}  \frac{\epsilon} {\la\tau\ra}  |\partial \partial_t u|^2dV_g + t^{C\epsilon} \|\partial \partial_{\leq 1} u(t_0)\|_{L^2}^2
\end{split}\end{equation}
On the other hand, by Lemma~\ref{wavecoordscommut} we can write
\[
\partial_t W^{\alpha} = S^Z(1) \partial \partial_t u + W^{low}, \qquad |W^{low}| \lesssim \epsilon |\partial_{\leq 1} u|
\]

 In the region $\frac{5M}2\leq r\leq \frac{7M}2$ we additionally have by our first technical assumption \eqref{W-est} that
\[
\partial_t W^{\alpha} = S^Z(t^{-1/2}) \partial \partial_t u + W^{low}
\]
We get
\[\begin{split}
& \int_{\M_{[t_0,t_1]}^{[r_e, R_1)}} |(\partial_t W^{\alpha}) \partial_{\alpha} u| (|\partial \partial_t u| + r^{-1} |\partial_t u|) dV_g \lesssim \int_{\M_{[t_0,t]}^{ps}} \frac{\epsilon} {\la\tau\ra} |\partial \partial_t u|^2 + \frac{\epsilon}{\la\tau\ra^{1/2}}|\partial_{\leq 1} u|^2 dV_g
\\ &+ \int_{\M_{[t_0,t_1]}^{[r_e, R_1)}\setminus \M_{[t_0,t]}^{ps}} \frac{\epsilon}{\la\tau\ra^{1/2}}|\partial_{\leq 2} u|^2 dV_g \lesssim \int_{\M_{[t_0,t]}^{ps}} \frac{\epsilon} {\la\tau\ra} |\partial \partial_t u|^2 dV_g + \epsilon \|\partial_{\leq 1} u\|_{LE_S^1[t_0, t]}^2
\end{split}\]
which by the induction hypothesis and Lemma~\ref{elliptic} implies that
\begin{equation}\label{onetderiv2}
\int_{\M_{[t_0,t_1]}^{[r_e, R_1)}} |(\partial_t W^{\alpha}) \partial_{\alpha} u| (|\partial \partial_t u| + r^{-1} |\partial_t u|) dV_g  \lesssim \int_{\M_{[t_0,t]}^{ps}} \frac{\epsilon} {\la\tau\ra}  |\partial \partial_t u|^2dV + \epsilon E_1(t) + t^{C\epsilon} \|\partial \partial_{\leq 1} u(t_0)\|_{L^2}^2
\end{equation}

 In the intermediate region $R_1^*\leq \rs\leq\frac{\tau}2$ we know by \eqref{u-decay} that
\[
|\partial_{\alpha} u| \lesssim \frac{\epsilon}{r\la t-\rs\ra^{1/2}}
\]
and thus by Lemma~\ref{wavecoordscommut} :
\begin{equation}\label{onetderiv3}\begin{split}
\int_{\M_{[t_0,t]}^{[R_1, \frac{\tau}2)}} |(\partial_t W^{\alpha}) \partial_{\alpha} u| (|\partial \partial_t u| + r^{-1} |\partial_t u|) dV_g & \lesssim \epsilon \|\partial_t u\|_{LE_S^1[t_0, t]}^2 + \epsilon \|u\|_{LE_S^1[t_0, t]}^2 \\ & \lesssim \epsilon E_1(t) + t^{C\epsilon} \|\partial \partial_{\leq 1} u(t_0)\|_{L^2}^2
\end{split}\end{equation}

 In the region $\frac{\tau}2\leq \rs$ we have, for the higher order term in $W$, since $|\partial_{\alpha} u|\lesssim \epsilon \la t\ra^{-1}$:
\begin{equation}\label{onetderiv4}
\int_{\M_{[t_0,t]}^{[\frac{\tau}2, \infty)}} |\partial_t \partial u| |\partial_{\alpha} u| (|\partial \partial_t u| + r^{-1} |\partial_t u|) dV_g \lesssim \int_{\M_{[t_0,t]}} \frac{\epsilon} {\la\tau\ra}  |\partial \partial_t u|^2dV_g + t^{C\epsilon} \|\partial \partial_{\leq 1} u(t_0)\|_{L^2}^2
\end{equation}

 For the lower order term we use \eqref{Wlotfar}. We obtain
\begin{equation}\label{onetderiv5}
\begin{split}
& \int_{\M_{[t_0,t]}^{[\frac{\tau}2, \infty)}} |\frac{\tau^{1+\delta}}{r^{1+\delta}\la t-\rs\ra^{1/2+\delta}}u_{\leq 1} \partial_{\alpha} u| (|\partial \partial_t u| + r^{-1} |\partial_t u|) dV_g \\& \lesssim \int_{\M_{[t_0,t]}^{[\frac{\tau}2, \infty)}} \epsilon \la\tau\ra^{-1} \la\tau-\rs\ra^{-1-\delta} |u_{\leq 1}| |\partial_{\leq 1}\partial_t u| dV_g  \lesssim  \int_{\M_{[t_0,t]}} \frac{\epsilon} {\la\tau\ra}  |\partial \partial_t u|^2dV_g + t^{C\epsilon} \|\partial \partial_{\leq 1} u(t_0)\|_{L^2}^2 \\ & + \int_{\M_{[t_0,t]}^{[\frac{\tau}2, \infty)}} \frac{\epsilon} {\la\tau\ra\la t-\rs\ra^{2+\delta}} u^2 dV_g \lesssim \int_{\M_{[t_0,t]}} \frac{\epsilon} {\la\tau\ra}  |\partial \partial_t u|^2dV_g + t^{C\epsilon} \|\partial \partial_{\leq 1} u(t_0)\|_{L^2}^2
\end{split}\end{equation}

 For the last inequality we used \eqref{LRHardy} to bound
\begin{equation}\label{Hardy2}
\int_{\rs\geq\frac{\tau}2} \frac{u^2}{\la t-\rs\ra^{2+\delta}} dx \lesssim \int_{\rs\geq\frac{\tau}4} |\partial u|^2 + r^{-1} u^2 dx
\end{equation}
 and Hardy's inequality.

 By \eqref{onetderiv1}, \eqref{onetderiv2}, \eqref{onetderiv3}, \eqref{onetderiv4} and \eqref{onetderiv5} we get
\begin{equation}\label{t-est}
 \int_{\M_{[t_0,t_1]}} |F_1| (|\partial \partial_t u| + r^{-1} |\partial_t u|) \lesssim \int_{\M_{[t_0,t]}} \frac{\epsilon} {\la\tau\ra}  |\partial \partial_t u|^2dV_g + t^{C\epsilon} \|\partial \partial_{\leq 1} u(t_0)\|_{L^2}^2 + \epsilon E_1(t)
\end{equation}

Gronwall's inequality combined with \eqref{partialtcommut} and \eqref{t-est} yields
\[
E_1(t) \lesssim t^{C\epsilon} \|\partial \partial_{\leq 1} u(t_0)\|_{L^2}^2
\]

which is the desired estimate for $\partial_t u$. By applying Lemma~\ref{elliptic}, the same holds true for $\partial u$.

 We now proceed by induction. We will show how to get the most difficult case when we commute with $N$ derivatives. We assume that the conclusion \eqref{higherLES} holds for up to $N-1$ derivatives, namely
\[
\|\partial \partial_{\leq N-1} u(t)\|_{L^2} + \|\partial_{\leq N-1} u\|_{LE_S^1[t_0, t]} \lesssim t^{C \epsilon} \|\partial \partial_{\leq N-1} u(t_0)\|_{L^2}, \qquad t_0 \leq t \leq t_1
\]
 and want to prove that
 \begin{equation}\label{Nderivs}
 \|\partial_{N+1} u(t)\|_{L^2} + \|\partial_{N} u\|_{LE_S^1[t_0, t]} \lesssim t^{C \epsilon} \|\partial \partial_{\leq N} u(t_0)\|_{L^2}, \qquad t_0 \leq t \leq t_1
 \end{equation}

  We start by commuting with $N$ time derivatives. We obtain
\begin{equation}
\Box_g \partial_t^N u = F_N\label{boxuN}
\end{equation}
 where
\begin{equation}
 F_N = \sum_{m+n=N} A_{mn} \partial_t^m h^{\alpha\beta} \partial_t^n \partial_{\alpha}\partial_{\beta} u + B_{mn} \partial_t^{m} W^{\alpha} \partial_t^n \partial_{\alpha} u\label{FN}
\end{equation}
 for some constants $A_{mn}$, $B_{mn}$ with $A_{0N} = B_{0N} = 0$.

  Let
\begin{equation}\label{enuN}
 E_N(t) = \|\partial \partial_t^N u(t)\|_{L^2}^2 + \|\partial_t^N u\|_{LE_S^1[t_0, t]}^2
\end{equation}

  By \eqref{LES} applied to $\partial_t^N u$ we know that
\begin{equation}\label{partialtcommutN}
E_N(t) \lesssim \int_{\M_{[t_0,t]}} |F_N| (|\partial \partial_t^N u| + r^{-1} |\partial_t^N u|)dV_g + \int_{\M_{[t_0,t]}^{ps}} \frac{\epsilon}{\la\tau\ra}|\partial \partial_t^N u|^2 dV_g + \|\partial \partial_t^n u(t_0)\|_{L^2}^2
\end{equation}

 We will now show that
 \begin{equation}\label{FN-est}
\int_{\M_{[t_0,t]}} |F_N| |\partial_{\leq 1} \partial_t^N u| dV_g \lesssim \int_{\M_{[t_0,t]}} \frac{\epsilon} {\la\tau\ra}  |\partial \partial_t^N u|^2dV_g + \epsilon E_N(t) + t^{C\epsilon}\|\partial \partial_{\leq N} u(t_0)\|_{L^2}^2
\end{equation}

 The conclusion \eqref{Nderivs} will now follow after applying Gronwall's inequality in \eqref{partialtcommutN} and using Lemma~\ref{elliptic}.

 Clearly either $m$ or $n$ is at most $\frac{N}2$. If $1\leq m\leq \frac{N}2$ we can apply \eqref{metricform}, \eqref{t-derivest} and Lemma~\ref{wavecoordscommut} to get
\begin{equation}\label{tderivglobal}
 |\partial_t^m h^{\alpha\beta}| + |\partial_t^{m} W^{\alpha}| \lesssim \epsilon \la t \ra^{-1}
\end{equation}
which suffices.

 On the other hand, if $1\leq n\leq \frac{N}2$ we also obtain from \eqref{metricform} and \eqref{t-derivest} that
\[
 |\partial_t^n \partial_{\alpha}\partial_{\beta} u| + |\partial_t^n \partial_j u| \lesssim \epsilon \la t \ra^{-1}, \qquad n\geq 1
\]

We are left with the case $n=0$.  Near the trapped set we have by \eqref{u-decay} and Lemma~\ref{elliptic}:
\[\begin{split}
&\int_{\M_{[t_0,t]}^{[r_e, R_1)}} |(\partial_t^N h^{\alpha\beta}) (\partial_{\alpha}\partial_{\beta} u) \partial_{\leq 1} \partial_t^N u| dV_g  \lesssim \int_{\M_{[t_0,t]}^{[r_e, R_1)}} \frac{\epsilon} {\la\tau\ra^{\frac12}} |(\partial_t^N h^{\alpha\beta})  \partial_{\leq 1} \partial_t^N u| dV_g \\ & \lesssim \int_{\M_{[t_0,t]}^{[r_e, R_1)}} \frac{\epsilon} {\la\tau\ra}  |\partial \partial_t^N u|^2 + \epsilon |\partial_{\leq N}u|^2 dV_g \lesssim \int_{\M_{[t_0,t]}^{[r_e, R_1)}} \frac{\epsilon} {\la\tau\ra}  |\partial \partial_t^N u|^2dV_g + \epsilon E_N(t) + \|\partial_{\leq N-1} u\|_{LE_S^1[t_0, t]}^2
\end{split}\]
which suffices by the induction hypothesis.

On the other hand for the second term we use \eqref{metricform}, \eqref{u-decay} and \eqref{W-est} as in \eqref{onetderiv2}. We obtain
\[\begin{split}
&\int_{\M_{[t_0,t]}^{[r_e, R_1)}} |(\partial_t^{N} W^{\alpha}) (\partial_{\alpha} u) \partial_{\leq 1} \partial_t^N u| dV_g  \lesssim \int_{\M_{[t_0,t]}^{[r_e, R_1)}} \frac{\epsilon} {\la\tau\ra^{\frac12}} |(\partial_t^{N} W^{\alpha}) \partial_{\leq 1} \partial_t^N u| dV_g \\ & \lesssim \int_{\M_{[t_0,t]}^{ps}} \frac{\epsilon} {\la\tau\ra} |\partial \partial_t^N u|^2 + |\partial_{\leq N}u|^2 dV_g + \int_{\M_{[t_0,t_1]}^{[r_e, R_1)}\setminus \M_{[t_0,t]}^{ps}} \epsilon |\partial_{\leq N+1} u|^2 dV_g \\ &
\lesssim \int_{\M_{[t_0,t]}} \frac{\epsilon} {\la\tau\ra}  |\partial \partial_t^N u|^2dV_g + \epsilon E_N(t) + \|\partial_{\leq N-1} u\|_{LE_S^1[t_0, t]}^2
\end{split}\]

 In the intermediate region $R_1^*\leq\rs\leq\frac{\tau}2$ we get as in \eqref{onetderiv1} and \eqref{onetderiv3}:
\[\begin{split}
\int_{\M_{[t_0,t]}^{[R_1, \frac{\tau}2})} (|\partial_t^N h^{\alpha\beta} \partial_{\alpha}\partial_{\beta}| + |\partial_t^{N} W^{\alpha} \partial_{\alpha} u|)|\partial_{\leq 1} \partial_t^N u| dV_g  \lesssim \int_{\M_{[t_0,t]}} \frac{\epsilon} {\la\tau\ra} & |\partial \partial_t^N u|^2dV_g + \epsilon E_N(t) + \\& t^{C\epsilon}\|\partial \partial_{\leq N} u(t_0)\|_{L^2}^2
\end{split}\]

In the region $\frac{\tau}2\leq \rs$ near the cone we have as in \eqref{onetderiv1}, \eqref{onetderiv4} and \eqref{onetderiv5}
\[
\int_{\M_{[t_0,t]}^{[\frac{\tau}2, \infty)}} (|\partial_t^N h^{\alpha\beta} \partial_{\alpha}\partial_{\beta}| + |\partial_t^{N} W^{\alpha} \partial_{\alpha} u|)|\partial_{\leq 1} \partial_t^N u| dV_g \lesssim \int_{\M_{[t_0,t]}} \frac{\epsilon} {\la\tau\ra}  |\partial \partial_t^N u|^2dV_g + t^{C\epsilon} \|\partial \partial_{\leq N-1} u(t_0)\|_{L^2}^2
\]

The conclusion \eqref{FN-est} now follows.

 We now use induction to add the rest of the vector fields in $Z$. Let us start by proving the desired estimate for $Zu$, where $Z\in\{\Omega, S\}$. A quick computation gives
\[
 \Box_g Zu = F_Z
\]
 where
 \begin{equation}\label{Z-error}
 F_Z = (Zh^{\alpha\beta})\partial_{\alpha}\partial_{\beta} u + h^{\alpha\beta} [Z, \partial_{\alpha}\partial_{\beta}]u + (ZW^{\alpha})\partial_{\alpha} u  + W^{\alpha} [Z, \partial_{\alpha}]u + [\Box_S, Z] u
 \end{equation}

 We will write
\[
F_Z = F_Z^1 + F_Z^2, \qquad F_Z^2 = \chi_{ps} [\Box_{g_S}, Z] u
\]

 Let
 \[
 E_Z (t) = \|\partial Zu(t)\|_{L^2}^2 + \|Zu\|_{LE_S^1[t_0, t]}^2
 \]

We would like to show that
\begin{equation}\label{Zvfd}
E_Z (t) \lesssim t^{C \epsilon} \|\partial u_{\leq 3}(t_0)\|_{L^2}^2
\end{equation}

 By applying Theorem~\ref{LE} to $Zu$ we obtain that
\begin{equation}\label{commutZ}
E_Z (t) \lesssim \int_{\M_{[t_0,t]}^{ps}} \frac{\epsilon}{\la\tau\ra}|\partial Zu|^2 dV_g + \|\partial u(t_0)\|_{L^2}^2 + \int_{\M_{[t_0,t_1]}} |F_1^Z| (|\partial Zu| + r^{-1} |Zu|) + |\partial_{\leq 1} F_2^Z|^2 dV_g
\end{equation}

Let us first estimate the last term. Since $[\Box_{g_S}, \Omega] = 0$, we only have to worry about the case $Z=S$. Since near the trapped set we have that $[\Box_{g_S}, Z] \approx \partial^2$ we get
\begin{equation}\label{vfcpt3}
\int_{\M_{[t_0,t]}} |\partial_{\leq 1} F_Z^2|^2 dV_g \lesssim \int_{\M_{[t_0,t_1]}^{ps}} |\partial \partial_{\leq 2} u|^2 dV_g \lesssim \|\partial_{\leq 3} u\|_{LE_S^1[t_0, t]}^2 \lesssim t^{C \epsilon} \|\partial \partial_{\leq 3} u(t_0)\|_{L^2}^2
\end{equation}
by the induction hypothesis.

 We will now prove that
\begin{equation}\label{Zinhom}\begin{split}
& \int_{\M_{[t_0,t]}} |F_Z^1| (|\partial u| + r^{-1} |u|) dV_g \leq \delta_0 E_Z(t) +  C(\delta_0) \Bigl(\int_{\M_{[t_0,t]}} \frac{\epsilon}{\la\tau\ra}|\partial Zu|^2 dV_g + \delta_{ZS} E_{\Omega}(t)\Bigr) \\ &  + C(\delta_0) t^{C\epsilon} \Bigl(\|\partial_{\leq 3} u \|_{LE_S^1[t_0, t]}^2 + \|\partial \partial_{\leq 3} u\|_{L^{\infty}[t_0, t]L^2}^2\Bigr)
\end{split}\end{equation}
Here $\delta_0$ is small, $\epsilon$-independent, and $\delta_{ZS} =0$ when $Z=\Omega$, $\delta_{ZS}=1$ when $Z=S$. The first term on the RHS can be absorbed to the LHS of \eqref{commutZ}, while in view of \eqref{Nderivs} we can bound the last term:
\[
\|\partial_{\leq 3} u \|_{LE_S^1[t_0, t]}^2 + \|\partial \partial_{\leq 3} u\|_{L^{\infty}[t_0, t]L^2}^2 \lesssim t^{C\epsilon}\|\partial \partial_{\leq 3} u(t_0)\|_{L^2}^2
\]

 When $Z=\Omega$, \eqref{Zvfd} follows from the inequality above and Gronwall's inequality. When $Z=S$ the conclusion will similarly follow since we have the needed estimate for $E_{\Omega}$.

  We treat each term in \eqref{Z-error} separately.

  For the first term we will use the pointwise bounds \eqref{u-decay} and \eqref{improvetwoderivs} to bound
 \[
 |\partial_{\alpha}\partial_{\beta} u| \lesssim \epsilon \frac{t}{r^2\la t-\rs\ra^{\frac32}}
 \]
  and the fact that
 \[
 |Zh| \lesssim |Zu| + |u|
 \]

  We divide our region into three parts as before.

  In the region $r_e\leq r\leq R_1$ we estimate:
\begin{equation}\label{vfcpt}
\begin{split}
\int_{\M_{[t_0,t_1]}^{[r_e, R_1)}} & |(Zh^{\alpha\beta})(\partial_{\alpha}\partial_{\beta} u)| (|\partial Zu| + r^{-1} |Zu|) dV_g  \lesssim \int_{\M_{[t_0,t_1]}^{[r_e, R_1)}} \frac{\epsilon}{\la\tau\ra^{\frac12}} |Zh| |\partial_{\leq 1}Zu| dV_g \\ &  \lesssim \int_{\M_{[t_0,t_1]}^{[r_e, R_1)}} \frac{\epsilon}{\la\tau\ra}|\partial Zu|^2 dV_g + \epsilon E_Z(t) + \epsilon \|\partial_{\leq 3}u\|_{LE_S^1[t_0, t]}^2
\end{split} \end{equation}

 In the intermediate region $R_1^* \leq \rs\leq \frac{\tau}2$ we get:
 \begin{equation}\label{vfinter}
\begin{split}
\int_{\M_{[t_0,t]}^{[R_1, \frac{\tau}2)}} & |(Zh^{\alpha\beta})(\partial_{\alpha}\partial_{\beta} u)| (|\partial Zu| + r^{-1} |Zu|) dV_g  \lesssim \int_{\M_{[t_0,t]}^{[R_1, \frac{\tau}2)}} \frac{\epsilon}{r^{\frac52}} |Zh| (|\partial Zu| + r^{-1} |Zu|) dV_g \\ &  \lesssim  \epsilon E_Z(t) + \epsilon \|\partial_{\leq 3}u\|_{LE_S^1[t_0, t]}^2
\end{split} \end{equation}

 Near the cone, where $\frac{\tau}2 \leq \rs$, we have
\begin{equation}\label{vfcone} \begin{split}
& \int_{\M_{[t_0,t]}^{[\frac{\tau}2, \infty)}}  |(Zh^{\alpha\beta})(\partial_{\alpha}\partial_{\beta} u)| (|\partial Zu| + r^{-1} |Zu|) dV_g  \lesssim \int_{\M_{[t_0,t]}^{[\frac{\tau}2, \infty)}} \frac{\epsilon}{\la\tau\ra\la t-\rs\ra^{\frac32}} |Zh| (|\partial Zu| + \\& r^{-1} |Zu|) dV_g \lesssim \int_{\M_{[t_0,t]}^{[\frac{\tau}2, \infty)}} \frac{\epsilon}{\la\tau\ra} |\partial Zu|^2 + \frac{\epsilon}{\la\tau\ra^2}(|Zu|^2 + |u|^2) + \frac{\epsilon}{\la\tau\ra\la \tau-\rs\ra^3}(|Zu|^2 + |u|^2) dV_g  \\ & \lesssim \int_{\M_{[t_0,t]}} \frac{\epsilon}{\la\tau\ra}|\partial Zu|^2 dV + t^{C\epsilon}\|\partial \partial_{\leq 3} u\|_{L^{\infty}[t_0, t]L^2}^2
\end{split}\end{equation}
where we used Hardy's inequality and \eqref{Hardy2} for the last estimate.

The desired conclusion thus follows for the first term in \eqref{Z-error}. The same argument also works for the second term in \eqref{Z-error} since $[Z, \partial_{\alpha}\partial_{\beta}] \approx \partial^2$ and $Zh\approx h$.

 For the third term, the crucial observation is that, while in general we have (see Lemma~\ref{wavecoordscommut})
\[
|ZW| \lesssim |\partial Zu| + |\partial \partial_{\leq 2} u| + |Zu| + |u|, \quad r\leq R_1
\]
 near $r=3M$ we have due to \eqref{W-est}
\[
|ZW| \lesssim \la t\ra^{-1/2} |\partial Zu| + |Zu| + |\partial_{\leq 3} u|
\]

 We thus get :
\begin{equation}\label{vfcpt2}
\begin{split}
 & \int_{\M_{[t_0,t_1]}^{[r_e, R_1)}}  |(ZW^{\alpha})(\partial_{\alpha} u)| (|\partial Zu| + r^{-1} |Zu|) dV_g \lesssim \int_{\M_{[t_0,t_1]}^{ps}} \frac{\epsilon}{\la\tau\ra} |\partial Zu|^2 + \epsilon(|Zu|^2 + |\partial_{\leq 3} u|^2) dV_g \\ & + \int_{\M_{[t_0,t_1]}^{[r_e, R_1)}\setminus \M_{[t_0,t]}^{ps}} \epsilon (|\partial_{\leq 1} Zu|^2 + |\partial_{\leq 3} u|^2) dV_g \lesssim \int_{\M_{[t_0,t]}} \frac{\epsilon}{\la\tau\ra}|\partial Zu|^2 dV_g + \epsilon E_Z(t) + \epsilon \|\partial_{\leq 3}u\|_{LE_S^1[t_0, t]}^2
\end{split}
\end{equation}

 In the region $\rs\geq R_1^*$ we have by Lemma~\ref{wavecoordscommut}
\[
|ZW| \lesssim |\partial Zu| + |\partial \partial_{\leq 2} u| + (\frac1{r^2}+ \frac{t^{1+\delta}}{r^{1+\delta}\la t-\rs\ra^{1/2+\delta}}) (|Zu| + |u|)
\]

In the intermediate region $R_1^* \leq \rs\leq \frac{\tau}2$ we now obtain
 \begin{equation}\label{vfinter2}\begin{split}
& \int_{\M_{[t_0,t]}^{[R_1, \frac{\tau}2)}} |(ZW^{\alpha})(\partial_{\alpha} u)| (|\partial Zu| + r^{-1} |Zu|) dV_g  \\ & \lesssim \int_{\M_{[t_0,t]}^{[R_1, \frac{\tau}2)}} \frac{\epsilon} {r\la t-\rs\ra^{\frac12}} |ZW| (|\partial Zu| + r^{-1} |Zu|) dV_g  \lesssim  \epsilon E_Z(t) + \epsilon\|\partial_{\leq 3} u\|_{LE_S^1[t_0, t]}^2
 \end{split}\end{equation}

Near the cone we have, similarly to \eqref{vfcone}
\begin{equation}\label{vfcone2}\begin{split}
\int_{\M_{[t_0,t]}^{[\frac{\tau}2, \infty)}}  |(ZW^{\alpha})(\partial_{\alpha} u)| & (|\partial Zu| + r^{-1} |Zu|) dV_g  \lesssim \int_{\M_{[t_0,t]}^{[\frac{\tau}2, \infty)}}  \frac{\epsilon} {\la\tau\ra} |ZW| (|\partial Zu| + r^{-1} |Zu|) dV_g  \\ & \lesssim \int_{\M_{[t_0,t]}} \frac{\epsilon}{\la\tau\ra}|\partial Zu|^2 dV_g + t^{C\epsilon}\|\partial \partial_{\leq 2} u\|_{L^{\infty}[t_0, t]L^2}^2
\end{split}\end{equation}

The fourth term is \eqref{Z-error} is treated similarly, since $ZW\approx W$ and $[Z, \partial_{\beta}]\approx \partial$.

Finally we have to deal with the fifth term, which comes from the linear part. In the compact region $r_e \leq r \leq R_1$ away from the photonsphere we have $[\Box_{g_S}, Z] \approx \partial^2$ and so we trivially estimate
\begin{equation}\label{vfcpt3}\begin{split}
\int_{\M_{[t_0,t_1]}^{[r_e, R_1)}} & (1-\chi_{ps})([\Box_{g_S}, Z] u) (|\partial Zu| + r^{-1} |Zu|) dV_g  \leq \int_{\M_{[t_0,t_1]}^{[r_e, R_1)}} (1-\chi_{ps}) \Bigl(\delta_0 (|\partial Zu|^2 + |Zu|^2)\\ & + C(\delta_0)|\partial_{\leq 2} u|^2\Bigr) dV_g  \leq \delta_0 E_Z(t) + C(\delta_0) \|\partial_{\leq 3} u\|_{LE_S^1[t_0, t]}^2
\end{split}\end{equation}

In the region $r\geq R_1$ we use the fact that the commutator has good decay properties. More precisely, $[\Box_{g_S}, \Omega] =0$, while for $S$ we have (see, for example, \cite{Luk1}, Proposition 4):
\begin{equation}\label{S-far}
[\Box_{g_S}, S] u \in S^Z(1) \Box_{g_S} u + S^Z(r^{-2+})(\partial u, \partial\Omega u)
\end{equation}

 Since
\[
 \Box_g u - \Box_{g_S} u = h^{\alpha\beta} \partial_{\alpha}\partial_{\beta} u + W^{\alpha}\partial_{\alpha} u
\]
the right hand side can be handled as in \eqref{vfinter}, \eqref{vfcone}, \eqref{vfinter2}, and \eqref{vfcone2}. On the other hand, we also have
\begin{equation}\label{vfcone3}
\begin{split}
&\int_{\M_{[t_0,t]}^{[R_1, \infty)}} r^{-2+} (|\partial u| + |\partial\Omega u|) (|\partial Su| + r^{-1} |Su|) dV_g  \leq \delta_0 E_S(t) +  \\ &C(\delta_0) (\|\partial u\|_{LE[t_0, t]}^2 + \|\partial \Omega u\|_{LE[t_0, t]}^2)\leq \delta_0 E_S(t) + C(\delta_0) (\|\partial_{\leq 3} u\|_{LE_S^1[t_0, t]}^2 + E_{\Omega}(t))
\end{split}\end{equation}

 This completes the proof of \eqref{Zinhom}.

 We now proceed by induction. Fix a positive integer $K\leq \frac{N}3$, and define
\[
\mathcal{P} = \{\Lambda=(i, j, k), j+k=K\}
\]
\[
\mathcal{P}^l = \{\Lambda=(i, j, k)\in \mathcal{P}, k\leq l\}, \qquad 0\leq l\leq K
\]

 We will prove the theorem by induction on $K$. We have explicitly obtained the estimates in the cases $K=0$ and $K=1$, $i=0$, which are the first two steps in the induction argument.

 For $\Lambda = (i, j, k)$ we say that $\Lambda\prec \mathcal{P}$ if $j+k<K$.

We assume that the theorem holds for all $\Lambda\prec\mathcal{P}$ and we will prove it for $\Lambda\in\mathcal{P}$.
 Let us define
\begin{equation}\label{ZKen}
E_{\mathcal{P}^l}(t) = {\sum}_{\Lambda\in\mathcal{P}^l} \|\partial u_{\Lambda}(t)\|_{L^2}^2 + \|u_{\Lambda}\|_{LE_S^1[t_0, t]}^2, \quad 0\leq l\leq K
\end{equation}
\begin{equation}\label{Zen}
E_{\mathcal{P}}(t) ={ \sum}_{\Lambda\in\mathcal{P}} \|\partial u_{\Lambda}(t)\|_{L^2}^2 + \|u_{\Lambda}\|_{LE_S^1[t_0, t]}^2
\end{equation}
and
\[
E_{\prec\mathcal{P}}(t) ={ \sum}_{\Lambda_1\prec\Lambda\in\mathcal{P}} \|\partial u_{\Lambda_1}(t)\|_{L^2}^2 + \|u_{\Lambda_1}\|_{LE_S^1[t_0, t]}^2
\]

 Pick some $\Lambda\in\mathcal{P}^l$. We will show that
\begin{equation}\label{Zenest}
E_{\mathcal{P}^l}(t) \lesssim \int_{\M_{[t_0,t]}} \epsilon \la \tau\ra^{-1} |\partial u_{\Lambda}|^2 dV_g +  E_{\mathcal{P}^{l-1}}(t) + t^{C\epsilon}E_{\prec\mathcal{P}}(t)
\end{equation}
where we define $E_{\mathcal{P}^{-1}}= 0$. \eqref{Zenest} easily implies that
\[
E_{\mathcal{P}}(t) \lesssim \int_{\M_{[t_0,t]}} \epsilon \la \tau\ra^{-1} E_{\mathcal{P}}(\tau) dV_g + t^{C\epsilon}E_{\prec\mathcal{P}}(t)
\]

  Gronwall's inequality and the induction hypothesis now imply the desired result.

  We start by noticing that, due to the elliptic estimate Lemma~\ref{elliptic}, we may assume that all derivatives in $\Lambda$ are time derivatives.

 After commuting the equation we get that
\begin{equation}
\Box_g u_{\Lambda} = F_{\Lambda} + L_{\Lambda}\label{boxuLambda}
\end{equation}
where
\begin{equation}\label{FLambda}
F_{\Lambda} = \sum_{\substack{\Lambda_1+\Lambda_2 = \Lambda \\|\Lambda_1|>0}} h_{\Lambda_1}^{\alpha\beta} \partial_{\alpha}\partial_{\beta}u_{\Lambda_2} + \sum_{\Lambda_1+\Lambda_2 = \Lambda} W_{\Lambda_1}^{\alpha} \partial_{\alpha}u_{\Lambda_2} + [\Box_{g_S}, Z^{\Lambda}] u
\end{equation}
and $L_{\Lambda}$ consists of lower order terms satisfying the bound
\begin{equation}
|L_{\Lambda}| \lesssim \sum_{m+n < |\Lambda|} (|h_{\leq m}| |\partial^{2} u_{\leq n}| + |W_{\leq m}^{\alpha}| |\partial_{\alpha}u_{\leq n}|)
\end{equation}

 We would like to prove that
\begin{equation}\label{FLambdainhom}
\begin{split}
\inf_{F_{\Lambda}=F_{\Lambda}^1+F_{\Lambda}^2} &\int_{\M_{[t_0,t]}} |F_{\Lambda}^1| (|\partial u_{\Lambda}| + r^{-1} |u_{\Lambda}|) + |\partial_{\leq 1} F_{\Lambda}^2|^2 dV_g \leq \delta_0 E_{\mathcal{P}^{l}}(t) \\ & + C(\delta_0) \Bigl(\int_{\M_{[t_0,t]}} \epsilon \la \tau\ra^{-1} |\partial u_{\Lambda}|^2 dV_g +  E_{\mathcal{P}^{l-1}}(t)\Bigr) + C(\delta_0) t^{C\epsilon}E_{\prec\mathcal{P}}(t)
\end{split}\end{equation}
and a similar result for $L_{\Lambda}$.

By applying Theorem~\ref{LE}, summing after $\Lambda\in\mathcal{P}^l$, and absorbing the first term on the right hand side of \eqref{FLambdainhom} to the left hand side, we obtain the desired conclusion \eqref{Zenest}.

 We will prove \eqref{FLambdainhom}. The corresponding estimate for $L_{\Lambda}$ is similar but easier since it is lower order.

 We write as before
\[
F_{\Lambda} = F_{\Lambda}^1 + F_{\Lambda}^2, \qquad F_{\Lambda}^2 = \chi_{ps} [\Box_{g_S}, Z^{\Lambda}] u
\]

 For the second term, it is immediate to see that
\[
\int_{\M_{[t_0,t]}} |\partial_{\leq 1} F_{\Lambda}^2|^2 dV_g \lesssim E_{\prec\mathcal{P}}(t)
\]

 We now estimate $F_{\Lambda}^1$. Clearly we must have either $|\Lambda_1|\leq \frac{N}2$ or $|\Lambda_2|\leq \frac{N}2$.

 In the second case, the estimates follow similarly to \eqref{vfcpt}, \eqref{vfinter}, \eqref{vfcone}, \eqref{vfcpt2}, \eqref{vfinter2} and \eqref{vfcone2}. Indeed, when $r_e\leq r\leq R_1$ we have
\begin{equation}\label{vfcpt1hgh}
\begin{split}
\int_{\M_{[t_0,t_1]}^{[r_e, R_1)}} & |h_{\Lambda_1}^{\alpha\beta} \partial_{\alpha}\partial_{\beta}u_{\Lambda_2}| (|\partial u_{\Lambda}| + r^{-1} |u_{\Lambda}|) dV_g \lesssim \int_{\M_{[t_0,t_1]}^{[r_e, R_1)}} \frac{\epsilon}{\la\tau\ra^{\frac12}} |h_{\Lambda_1}| |\partial_{\leq 1}u_{\Lambda}| dV_g  \\ & \lesssim \int_{\M_{[t_0,t]}} \frac{\epsilon}{\la\tau\ra}|\partial u_{\Lambda}|^2 dV_g + \epsilon E_{\mathcal{P}^l}(t) + E_{\prec\mathcal{P}}(t)
\end{split}\end{equation}
\begin{equation}\label{vfcpt2hgh}\begin{split}
& \int_{\M_{[t_0,t_1]}^{[r_e, R_1)}} |W_{\Lambda_1}^{\alpha} \partial_{\alpha}u_{\Lambda_2}| (|\partial u_{\Lambda}| + r^{-1} |u_{\Lambda}|) dV_g \lesssim \int_{\M_{[t_0,t_1]}^{[r_e, R_1)}} \frac{\epsilon}{\la\tau\ra^{\frac12}} |W_{\Lambda_1}| |\partial_{\leq 1}u_{\Lambda}| dV_g \\ & \lesssim \int_{\M_{[t_0,t_1]}^{ps}}  \frac{\epsilon}{\la \tau\ra} |\partial u_{\Lambda}|^2 + \epsilon |u_{\leq\Lambda}|^2 dV_g + \int_{\M_{[t_0,t_1]}^{[r_e, R_1)}\setminus \M_{[t_0,t]}^{ps}} \epsilon |\partial_{\leq 1} u_{\leq \Lambda}|^2 dV_g \\& \lesssim \int_{\M_{[t_0,t]}} \frac{\epsilon}{\la\tau\ra}|\partial u_{\Lambda}|^2 dV_g + \epsilon E_{\mathcal{P}^l}(t) + E_{\prec\mathcal{P}}(t)
\end{split} \end{equation}

In the intermediate region $R_1^* \leq \rs\leq \frac{\tau}2$ we obtain
\begin{equation}\label{vfinter1hgh}
\begin{split}
\int_{\M_{[t_0,t]}^{[R_1, \frac{\tau}2)}} & |h_{\Lambda_1}^{\alpha\beta} \partial_{\alpha}\partial_{\beta}u_{\Lambda_2}| (|\partial u_{\Lambda}| + r^{-1} |u_{\Lambda}|) dV_g \lesssim \int_{\M_{[t_0,t]}^{[R_1, \frac{\tau}2)}}\frac{\epsilon} {r^{\frac52}} |h_{\Lambda_1}| (|\partial u_{\Lambda}| + r^{-1} |u_{\Lambda}|) dV_g \\ & \lesssim  \epsilon E_{\mathcal{P}^l}(t) + E_{\prec\mathcal{P}}(t)
\end{split}\end{equation}
 \begin{equation}\label{vfinter2hgh}
\begin{split}
& \int_{\M_{[t_0,t]}^{[R_1, \frac{\tau}2)}}  |W_{\Lambda_1}^{\alpha} \partial_{\alpha}u_{\Lambda_2}| (|\partial u_{\Lambda}| + r^{-1} |u_{\Lambda}|) dV_g
 \lesssim \int_{\M_{[t_0,t]}^{[R_1, \frac{\tau}2)}} \frac{\epsilon}{r^{\frac32}} |W_{\Lambda_1}| (|\partial u_{\Lambda}| + r^{-1} |u_{\Lambda}|) dV_g
 \\ & \lesssim  \epsilon E_{\mathcal{P}^l}(t) + \epsilon E_{\prec\mathcal{P}}(t)
\end{split}\end{equation}

 Near the cone, we get by Hardy's inequality and \eqref{Hardy2}:
\begin{equation}\label{vfconehgh} \begin{split}
& \int_{\M_{[t_0,t]}^{[\frac{\tau}2, \infty)}} |h_{\Lambda_1}^{\alpha\beta} \partial_{\alpha}\partial_{\beta}u_{\Lambda_2}| (|\partial u_{\Lambda}| + r^{-1} |u_{\Lambda}|) dV_g \lesssim \int_{\M_{[t_0,t]}^{[\frac{\tau}2, \infty)}}\frac{\epsilon}{\la\tau\ra\la \tau-\rs\ra^{\frac32}} |u_{\leq \Lambda_1}| (|\partial u_{\Lambda}| \\ & + r^{-1} |u_{\Lambda}|) dV_g  \lesssim \int_{\M_{[t_0,t]}^{[\frac{\tau}2, \infty)}} \frac{\epsilon}{\tau} |\partial u_{\leq\Lambda}|^2 + \frac{\epsilon}{\tau^2} |u_{\leq\Lambda}|^2 dV_g \lesssim \int_{\M_{[t_0,t]}} \frac{\epsilon}{\la\tau\ra}|\partial u_{\Lambda}|^2 dV_g + t^{C\epsilon} E_{\prec\mathcal{P}}(t)
\end{split}\end{equation}
and by Lemma~\ref{wavecoordscommut}
\begin{equation}\label{vfcone2hgh}\begin{split}
& \int_{\M_{[t_0,t]}^{[\frac{\tau}2, \infty)}}  |W_{\Lambda_1}^{\alpha} \partial_{\alpha}u_{\Lambda_2}| (|\partial u_{\Lambda}| + r^{-1} |u_{\Lambda}|) dV_g \lesssim \int_{\M_{[t_0,t]}^{[\frac{\tau}2, \infty)}} \frac{\epsilon} {\la\tau\ra} |W_{\Lambda_1}| (|\partial u_{\Lambda}| + r^{-1} |u_{\Lambda}|) dV_g   \lesssim \\ & \int_{\M_{[t_0,t]}} \frac{\epsilon}{\la\tau\ra}|\partial u_{\Lambda}|^2 dV_g + t^{C\epsilon}\|\partial u_{\prec \Lambda}\|_{L^{\infty}[t_0, t]L^2}^2  \lesssim \int_{\M_{[t_0,t]}} \frac{\epsilon}{\la\tau\ra}|\partial u_{\Lambda}|^2 dV_g + \epsilon E_{\mathcal{P}^l}(t) + t^{C\epsilon} E_{\prec\mathcal{P}}(t)
\end{split}\end{equation}

   Let us treat the first case, when $|\Lambda_1|\leq \frac{N}2$.  If $i_1 > 0$ we get by \eqref{tderivglobal}, since we are assuming that the derivatives are time derivatives:
\[
|h_{\Lambda_1}| + |W^{\alpha}_{\Lambda_1}|\lesssim \epsilon \la t\ra^{-1}
\]
and thus
\begin{equation}\label{vfderhgh}\begin{split}
\int_{\M_{[t_0,t]}}  (|h_{\Lambda_1}^{\alpha\beta} \partial_{\alpha}\partial_{\beta}u_{\Lambda_2}| + |W_{\Lambda_1}^{\alpha} \partial_{\alpha} u_{\Lambda_2}) (|\partial u_{\Lambda}| + r^{-1} |u_{\Lambda}|) dV_g \lesssim \int_{\M_{[t_0,t]}} \frac{\epsilon}{\la\tau\ra}|\partial u_{\Lambda}|^2 dV_g + t^{C\epsilon} E_{\prec\mathcal{P}}(t)
\end{split} \end{equation}

 On the other hand, if $i_1=0$ then a-priori we only know that $|h_{\Lambda_1}|\lesssim \epsilon\la t\ra^{-1}\la t-\rs\ra^{\frac12}$. The crucial observation for the first term is that, since $|\Lambda_1|>0$, we have that $(2, 0, 0) + \Lambda_2\prec\Lambda$, and thus $\partial_{\alpha}\partial_{\beta} u_{\Lambda_2}$ is of lower order. Near the trapped set we thus get
\begin{equation}\label{vfcpt4hgh}
\int_{\M_{[t_0,t_1]}^{[r_e, R_1)}}  |h_{\Lambda_1}^{\alpha\beta} \partial_{\alpha}\partial_{\beta}u_{\Lambda_2}| (|\partial u_{\Lambda}| + r^{-1} |u_{\Lambda}|) dV_g  \lesssim \int_{\M_{[t_0,t_1]}^{[r_e, R_1)}} \epsilon |u_{\prec\Lambda}|^2 dV_g  \lesssim \epsilon E_{\prec\mathcal{P}}(t)
\end{equation}

 In the intermediate region and near the cone we use Klainerman-Sideris:
\begin{equation}\label{vfinter3hgh}\begin{split}
& \int_{\M_{[t_0,t]}^{[R_1, \infty)}}  |h_{\Lambda_1}^{\alpha\beta} \partial_{\alpha}\partial_{\beta}u_{\Lambda_2}| (|\partial u_{\Lambda}| + r^{-1} |u_{\Lambda}|) dV_g  \lesssim \int_{\M_{[t_0,t]}^{[R_1, \infty)}}  \frac{\epsilon\la \tau-\rs\ra^{\frac12}}{\la \tau\ra} |\partial_{\alpha}\partial_{\beta} u_{\Lambda_2}|  (|\partial u_{\Lambda}| \\ & + r^{-1} |u_{\Lambda}|) dV_g  \lesssim \int_{\M_{[t_0,t]}^{[R_1, \infty)}} \sum_{Z\in\{\Omega, S\}} \Bigl(\frac{\epsilon} {\la \tau\ra^\la \tau-r\ra^{\frac12}} |\partial Zu_{\Lambda_2}| + \epsilon \la \tau-r\ra^{-\frac12} |\Box u_{\Lambda_2}|\Bigr) (|\partial u_{\Lambda}| \\ & + r^{-1} |u_{\Lambda}|) dV_g  \lesssim \int_{\M_{[t_0,t]}}  \frac{\epsilon} {\la \tau\ra} |\partial u_{\Lambda}|^2 dV_g + t^{C\epsilon} E_{\prec\mathcal{P}}(t) + \int_{\M_{[t_0,t]}^{[R_1, \infty)}} \epsilon \la \tau-r\ra^{-\frac12} |(\Box_g u - \Box u)_{\Lambda_2}| \\ & (|\partial u_{\Lambda}| + r^{-1} |u_{\Lambda}|) dV_g  \lesssim \int_{\M_{[t_0,t]}} \frac{\epsilon} {\la \tau\ra} |\partial u_{\Lambda}|^2 dV_g + t^{C\epsilon}E_{\prec\mathcal{P}}(t)
\end{split} \end{equation}
where we can use the induction hypothesis to estimate the last term.

The estimate for the $W^{\alpha}$ term is similar, but easier. Near the trapped set we have
\begin{equation}\label{vfcpt5hgh}\begin{split}
\int_{\M_{[t_0,t_1]}^{[r_e, R_1)}}  |W_{\Lambda_1}^{\alpha} \partial_{\alpha}u_{\Lambda_2}| (|\partial u_{\Lambda}| + r^{-1} |u_{\Lambda}|) dV_g & \lesssim \int_{\M_{[t_0,t_1]}^{[r_e, R_1)}} \frac{\epsilon} {\la \tau\ra} |\partial u_{\Lambda}|^2 + \epsilon |u_{\prec\Lambda}|^2 dV_g \\ & \lesssim \int_{\M_{[t_0,t]}} \frac{\epsilon} {\la \tau\ra} |\partial u_{\Lambda}|^2 dV_g + \epsilon E_{\prec\mathcal{P}}(t)
\end{split}\end{equation}
In the intermediate region we get from Lemma~\ref{wavecoords2} and \eqref{u-decay} that $|W_{\Lambda_1}| \lesssim \epsilon r^{-1-\delta}$ and thus
\begin{equation}\label{vfinter3hgh}\begin{split}
\int_{\M_{[t_0,t]}^{[R_1, \frac{\tau}2)}}  |W_{\Lambda_1}^{\alpha} \partial_{\alpha}u_{\Lambda_2}| (|\partial u_{\Lambda}| + r^{-1} |u_{\Lambda}|) dV_g &  \lesssim \int_{\M_{[t_0,t]}^{[R_1, \frac{\tau}2)}} \epsilon r^{-1-\delta} |\partial_{\alpha} u_{\Lambda_2}| (|\partial u_{\Lambda}| + r^{-1} |u_{\Lambda}|) dV_g \\ & \lesssim \epsilon E_{\mathcal{P}^l}(t) + E_{\prec\mathcal{P}}(t)
\end{split}\end{equation}
Near the cone Lemma~\ref{wavecoords2} and \eqref{u-decay} imply that $|W_{\Lambda_1}| \lesssim \frac{\epsilon}{\la t\ra \la t-\rs\ra^{\delta}}$ and thus
\begin{equation}\label{vfcone3hgh}\begin{split}
\int_{\M_{[t_0,t]}^{[\frac{\tau}2, \infty)}} |W_{\Lambda_1}^{\alpha} \partial_{\alpha}u_{\Lambda_2}| (|\partial u_{\Lambda}| + r^{-1} |u_{\Lambda}|) dV_g  & \lesssim \int_{\M_{[t_0,t]}^{[\frac{\tau}2, \infty)}} \frac{\epsilon}{\la\tau\ra} |\partial_{\alpha} u_{\Lambda_2}| (|\partial u_{\Lambda}| + r^{-1} |u_{\Lambda}|) dV_g \\ & \lesssim \int_{\M_{[t_0,t]}} \frac{\epsilon} {\la \tau\ra} |\partial u_{\Lambda}|^2 dV_g + t^{C\epsilon} E_{\prec\mathcal{P}}(t)
\end{split}\end{equation}

 Finally, for the last term in \eqref{FLambdainhom}, we see, due to \eqref{S-far}, that
\begin{equation}\label{S-farhgh}
[\Box_{g_S}, Z^{\Lambda}] u \in S^Z(r^{-2+}) (\partial u_{\tilde\Lambda} + |\partial u_{\prec\Lambda}|) + F_{\prec\Lambda}, \quad  \tilde\Lambda\in \mathcal{P}^{l-1}
\end{equation}
  The last term can be estimated by our induction hypothesis. The highest order term can be treated as in \eqref{vfcpt3} and \eqref{vfcone3}:
\begin{equation}\label{lincomm}
 \int_{\M_{[t_0,t]}}  (1-\chi_{ps}) r^{-2+}|\partial u_{\tilde\Lambda}| (|\partial u_{\Lambda}| + r^{-1} |u_{\Lambda}|) dV_g \leq \delta_0 E_{\mathcal{P}^{l}}(t) + C(\delta_0) (E_{\mathcal{P}^{l-1}}(t) + E_{\prec\mathcal{P}}(t))
 \end{equation}
while for the lower order term we immediately get
\begin{equation}\label{lincommlot}
 \int_{\M_{[t_0,t]}}  (1-\chi_{ps}) r^{-2+}|\partial u_{\prec\Lambda}| (|\partial u_{\Lambda}| + r^{-1} |u_{\Lambda}|) dV_g \leq \delta_0 E_{\mathcal{P}^{l}}(t) + C(\delta_0)E_{\prec\mathcal{P}}(t)
\end{equation}
The estimate \eqref{FLambdainhom} now follows from \eqref{vfcpt1hgh} - \eqref{lincommlot}.
\end{proof}

\newsection{Pointwise decay}
In this section we will establish the required pointwise decay for $u$ and vector fields applied to $u$. More precisely, we will prove the following
\begin{theorem}\label{ptwsedcy} Let $T$ be a fixed time and $u$ solve \eqref{maineeqn} in the time interval $T\leq t\leq 2T$. Assume that $g(u,t,x)$ satisfies the conditions from Section 3.
Then for any multi index $|\Lambda|\leq N-13$ we have for $T\leq t\leq 2T$:
\begin{equation}\label{ptdecayu}
|u_{\Lambda}| \leq C'_{|\Lambda|} \la t\ra^{-1} \la t-\rs\ra^{1/2} \|u_{\leq |\Lambda|+13}\|_{LE_S^1[T, 2T]}
\end{equation}
\begin{equation}\label{ptdecaydu}
|\partial u_{\Lambda}| \leq C'_{|\Lambda|} \la r\ra^{-1} \la t-\rs\ra^{-1/2} \|u_{\leq |\Lambda|+13}\|_{LE_S^1[T, 2T]}
\end{equation}
\end{theorem}
 In particular we obtain by Theorem~\ref{higherLE} that
\begin{equation}\label{ptdecayu2}
|u_{\leq N-13}| \lesssim \la t\ra^{-1+C\epsilon} \la t-\rs\ra^{1/2}
\end{equation}
\begin{equation}\label{ptdecaydu2}
|\partial u_{\leq N-13}| \lesssim \la t\ra^{C\epsilon} \la r\ra^{-1} \la t-\rs\ra^{-1/2}
\end{equation}
\begin{proof}
 We follow notation from \cite{MTT}. For the region
 \[
C_{T} = \{ T \leq t \leq 2T, \ \ r \leq t\}.
\]
 we use a double dyadic decomposition of it
with respect to either the size of $t-\rs$ or the size of $r$, depending
on whether we are close or far from the cone,
\[
C_{T} ={\bigcup}_{1\leq  R \leq T/4}  C_{T}^{R} \, \,\,\bigcup\,\,\, {\bigcup}_{1\leq  U < T/4} C_T^{U}
\]
where for $R,U > 1$ we set
\[
 C_{T}^{R} = C_T \cap \{ R < r < 2R \},
\qquad
C_{T}^{U} = C_T \cap \{ U < t-\rs < 2U\}
\]
while for $R=1$ and $U= 1$ we have
\[
 C_{T}^{R=1} = C_T \cap \{ 0 < r < 2 \},
\qquad
C_{T}^{U=1} = C_T \cap \{ 0 < t-\rs < 2\}
\]
The sets $C_T^R$ and $C_T^U$ represent the setting in which we apply
Sobolev embeddings, which allow us to obtain pointwise bounds from
$L^2$ bounds. Precisely, we have

\begin{lemma} \label{l:l2tolinf}
 For any function $w$ and all $T \geq 1$ and $1 \leq R,U \leq T/4$  we have
 \begin{equation}
  \!  \| w\|_{L^\infty(C_T^{R})} \lesssim
\frac{1}{T^{\frac12} R^{\frac32}} \sum_{i\leq 1, j \leq 2}
    \|S^i \Omega^j w\|_{L^2( C_T^{R})} +  \frac{1}{T^{\frac12} R^{\frac12}}
\sum_{i\leq 1, j \leq 2}  \|S^i \Omega^j \nabla w\|_{L^2( C_T^{R})},
    \label{l2tolinf-r}\end{equation}
respectively
  \begin{equation}
    \| w\|_{L^\infty(C_T^{U})} \lesssim \frac{1}{T^{\frac32} U^{\frac12}} \sum_{i\leq 1, j \leq 2}
    \|S^i \Omega^j w\|_{L^2( C_T^{U})} +  \frac{U^{\frac12}}{T^{\frac32}}
 \sum_{i\leq 1, j \leq 2} \|S^i \Omega^j \nabla w\|_{L^2(C_T^{U})}.
    \label{l2tolinf-u}\end{equation}
\end{lemma}
\begin{proof}
  In exponential coordinates $(s,\rho,\omega)$ with $t = e^s$ and $r =
  e^{s+\rho}$ the region $C_T^R$ becomes a region of size $1$ in all directions. The bound \eqref{l2tolinf-r} follows from applying the fundamental theorem of calculus in the $s$ and $\rho$ directions,
  combined with the usual
  Sobolev embeddings on the sphere $\S^2$.  The
  same applies for \eqref{l2tolinf-u} in exponential coordinates
  $(s,\rho,\omega)$ with $t = e^s$ and $t-\rs = e^{s+\rho}$.
\end{proof}

  By applying \eqref{l2tolinf-r} to $u_{\Lambda}$ and taking the supremum over $R\leq\frac{T}2$ we obtain \eqref{ptdecayu} in the region $\rs\leq \frac{t}2$.

  Near the cone we use \eqref{l2tolinf-u} in conjunction with the Hardy-type inequality \eqref{LRHardy}. Let $\chi$ be a cutoff so that $\chi\equiv 1$ when $x\geq \frac12$ and $\chi\equiv 0$ when $x\leq \frac14$.  For any function $w$ we have
\begin{equation}\label{Hrdycone}
 U^{-1} \|w\|_{L^2( C_T^{U})} \lesssim \Big\|\frac{\chi(\rs/t) w}{\la t-\rs\ra}\Big\|_{L^2[T, 2T]L^2} \lesssim \|\partial_{\rs} (\chi(\rs/t) w)\|_{L^2[T, 2T]L^2} \lesssim T^{\frac12} \|w\|_{LE_S^1[T, 2T]}
\end{equation}
Using the inequality above for $w=S^i \Omega^j u_{\Lambda}$ in \eqref{l2tolinf-u} we obtain the desired estimate
 \[
 \|u_{\Lambda}\|_{L^\infty(C_T^{U})} \lesssim \frac{U^{\frac12}}{T} \|u_{\leq |\Lambda|+10}\|_{LE_S^1[T, 2T]}
 \]

  In order to estimate the derivatives, we need another Klainerman-Sideris type estimate for second derivatives of higher order terms.
\begin{lemma}\label{KS2}  Assume that $g$ is a Lorentzian metric satisfying the conditions \eqref{metricdecay}.
 Then for any multi index $|\Lambda|\leq N$ we have for $r\geq 2R_1$:
\begin{equation}\label{improvetwoderivs2}
|\partial^2 u_{\Lambda}| \lesssim \frac{t}{r\la t-\rs\ra} |\partial u_{\leq |\Lambda|+3}| + \frac{t^2}{r^2\la t-\rs\ra^{2}} |u_{\leq |\Lambda|}| + \frac{t}{\la t-\rs\ra} |(\Box_g u)_{\Lambda}|
\end{equation}
\end{lemma}
 The proof is similar to the one of Lemma~\ref{KS1}. We start with the estimate \eqref{KSLbdest}, and we need to bound $|\Box u_{\Lambda}|$. We again have
\[
|\Box u_{\Lambda}| \lesssim |(\Box u)_{\Lambda}| + |(\Box u)_{\leq |\Lambda|-3}|
\]
 and
\[
 |(\Box_{g_S} u - \Box u)_{\Lambda}| \lesssim r^{-1} (|\partial^2 u_{\Lambda}| + |\partial u_{\leq |\Lambda|}|)
\]
  We now have
\[ (\Box_g u - \Box_{g_S} u)_{\Lambda} = \sum_{\Lambda_1+\Lambda_2 = \Lambda} h_{\Lambda_1}^{\alpha\beta} \partial_{\alpha}\partial_{\beta}u_{\Lambda_2} + \sum_{\Lambda_1+\Lambda_2 = \Lambda} W_{\Lambda_1}^{\alpha} \partial_{\alpha}u_{\Lambda_2} \]

 Either $|\Lambda_1|\leq \frac{N}2$ or $|\Lambda_2|\leq \frac{N}2$.
  In the first case we have, similarly to the proof of Lemma~\ref{KS1}
\[
|h_{\Lambda_1}|\lesssim \frac{\epsilon\la t-\rs\ra^{1/2}}{\la t\ra}, \qquad |W_{\Lambda_1}|\lesssim \frac{1}{\la r\ra\la t-\rs\ra^{\delta}}
\]
 so
\[
\frac{t}{\la t-\rs\ra} |h_{\Lambda_1}^{\alpha\beta} \partial_{\alpha}\partial_{\beta}u_{\Lambda_2}| \lesssim \frac{\epsilon}{\la t-\rs\ra^{1/2}} (|\partial^2 u_{\Lambda}| + |\partial^2 u_{\leq |\Lambda|-1}|)
\]
\[
\frac{t}{\la t-\rs\ra} |W_{\Lambda_1}^{\alpha} \partial_{\alpha}u_{\Lambda_2}| \lesssim \frac{t}{r\la t-\rs\ra} |\partial u_{\leq |\Lambda|}|
\]

  In the second case, we have by Lemma~\ref{KS1}
\[
|\partial u_{\Lambda_2}| \lesssim \frac{\epsilon}{r\la t-\rs\ra^{\frac12}}, \qquad |\partial^2 u_{\Lambda_2}|\lesssim \frac{\epsilon t}{r^2\la t-\rs\ra^{\frac32}}
\]
so
\[
\frac{t}{\la t-\rs\ra} |h_{\Lambda_1}^{\alpha\beta} \partial_{\alpha}\partial_{\beta}u_{\Lambda_2}| \lesssim \frac{\epsilon t^2}{r^2\la t-\rs\ra^{5/2}} |u_{\leq |\Lambda|}|
\]
\[
\frac{t}{\la t-\rs\ra} |W_{\Lambda_1}^{\alpha} \partial_{\alpha}u_{\Lambda_2}| \lesssim \frac{t}{r\la t-\rs\ra^{3/2}} |W_{\leq |\Lambda|}| \lesssim \frac{t}{r\la t-\rs\ra^{3/2}} |\partial u_{\leq |\Lambda|}| +
\frac{t^2}{r^2\la t-\rs\ra^{2}} |u_{\leq |\Lambda|}|
\]

The conclusion \eqref{improvetwoderivs2} now follows as in Lemma~\ref{KS1}.

Coming back to the proof of \eqref{ptdecaydu}, we apply \eqref{l2tolinf-r} and \eqref{l2tolinf-u} to $\partial u_{\Lambda}$. We obtain
\[ \begin{split}
\| \partial u_{\Lambda}\|_{L^\infty(C_T^{R})} \lesssim & \frac{1}{T^{\frac12} R^{\frac32}} \sum_{i\leq 1, j \leq 2}
    \|S^i \Omega^j \partial u_{\Lambda}\|_{L^2( C_T^{R})} +  \frac{1}{T^{\frac12} R^{\frac12}}
\sum_{i\leq 1, j \leq 2}     \|S^i \Omega^j \partial^2 u_{\Lambda}\|_{L^2( C_T^{R})} \\ & \lesssim \frac{1}{T^{\frac12}R} \|u_{\leq |\Lambda|+13}\|_{LE_S^1[T, 2T]}
\end{split}\]
where we used \eqref{improvetwoderivs2} to bound the last term.
Similarly, by using \eqref{improvetwoderivs2} and \eqref{Hrdycone}, we obtain
\[ \begin{split}
\| \partial u_{\Lambda}\|_{L^\infty(C_T^{U})} \lesssim & \frac{1}{T^{\frac32} U^{\frac12}} \sum_{i\leq 1, j \leq 2}
    \|S^i \Omega^j \partial u_{\Lambda}\|_{L^2( C_T^{U})} +  \frac{U^{\frac12}}{T^{\frac32}}
\sum_{i\leq 1, j \leq 2}     \|S^i \Omega^j \partial^2 u_{\Lambda}\|_{L^2(C_T^{U})} \\ & \lesssim \frac{1}{T^{\frac12}U^{\frac12}} \|u_{\leq |\Lambda|+13}\|_{L^2( C_T^{U})} +
 \frac{1}{(UT)^{\frac32}} \|u_{\leq |\Lambda|+10}\|_{L^2( C_T^{U})}  \lesssim \frac{1}{TU^{\frac12}} \|u_{\leq |\Lambda|+13}\|_{LE_S^1[T, 2T]}
\end{split}\]
\end{proof}

\newsection{Boundedness of the lower order norms}
 In order to close our bootstrap, having growing local energy norms does not suffice. In this section we will show that the lower order norms are actually bounded. We will denote by $\mathcal T = \{\overline\partial\}$ the set of tangential derivatives, and let $\partial_T = T^{\alpha}\partial_{\alpha}$.
We start with the following analogous result to Theorem~\ref{LE}.
\begin{theorem}\label{LESbded} Let $u$ solve $\Box_g u = F$, where $g$ is as in Theorem~\ref{LE}.  Then we have
 \begin{equation}\label{0lowLES}
 \|\partial  u\|_{L^{\infty}[t_0, t_1] L^2}^2 + \|u\|_{LE_S^1[t_0, t_1]}^2 \lesssim \|\partial u_{\leq 1}(t_0)\|_{L^2}^2 + \|F_{\leq 1}\|_{L^1L^2+LE^*_S}^2
\end{equation}
\end{theorem}
\begin{proof}
 We have by Theorem~\ref{higherLE}
 \[ \begin{split}
& \int_{\M_{[t_0,t_1]}^{ps}} \frac{\epsilon}{t}|\partial u|^2 dV_g  \lesssim \sum_{T dyadic} \int_{\M_{[T, 2T]}^{ps}} t^{-1} |\partial u|^2 dV_g  \lesssim \sum_{T dyadic} T^{-1} T^{C\epsilon} (\|\partial u_{\leq 1}(t_0)\|_{L^2}^2 \\ & + \|F_{\leq 1}\|_{L^1L^2+LE^*_S}^2)  \lesssim \|\partial u_{\leq 1}(t_0)\|_{L^2}^2 + \|F_{\leq 1}\|_{L^1L^2+LE^*_S}^2
 \end{split}\]
 The result now follows from \eqref{LES} and Gronwall's inequality.
\end{proof}
 We now prove a higher order version of Theorem~\ref{LESbded} which is a refinement of
 Theorem~\ref{higherLE} for low norms.
\begin{theorem}\label{lowerLE} Let $u$ solve \eqref{maineeqn}, where $g$ is as in Theorem~\ref{higherLE}.
Then for $\tN \leq N-3$, we have:
\begin{equation}\label{lowLES}
 \|\partial  u_{\leq \tN}\|_{L^{\infty}[t_0, t_1] L^2}^2 + \|u_{\leq \tN}\|_{LE_S^1[t_0, t_1]}^2 \lesssim \|\partial u_{\leq N}(t_0)\|_{L^2}^2
\end{equation}
\end{theorem}

  We need a technical lemma to write the difference $\Box_g - \Box_{g_S}$ with respect to the null frame:
\begin{lemma}\label{waveopnullfraame}  We have
\begin{equation}\label{null}
(\Box_g - \Box_{g_S})u = h^{\uL\uL}\partial_{\uL}^2 u +T^2_g(h)u+ \frac{1}{\sqrt{|g|}}\partial_{\uL}(\sqrt{|g|}h^{\uL\uL})\partial_{\uL} u + T^1_g(h)u
+T^0_g(h)u ,
\end{equation}
where $\partial_U=U^\alpha\partial_\alpha$,
\begin{equation}
T^2_g(h)u = {\sum}_{T\in\mathcal{T}}( h^{T\uL}\uL^{\beta}+h^{T\beta} )\partial_T\partial_{\beta} u,
\end{equation}
\begin{multline}
T^1_g(h)u=\frac{1}{\sqrt{|g|}}\Big(\frac{1}{g_S(\uL,L)}\uL_\alpha\partial_{L }\big(\sqrt{|g|}h^{\alpha\beta}\big)
+{\sum}_{T\in\{A,B\}}T_\alpha\partial_{T}\big(\sqrt{|g|}h^{\alpha\beta}\big)\Big)\partial_\beta u\\
+\frac{L_\alpha}{\sqrt{|g|}g_S(\uL,L)}\partial_{\uL }\big(\sqrt{|g|}h^{\alpha\beta}\big)\Big(\frac{\uL_\beta}{g_S(\uL,L)}\partial_{L }
+\sum_{T\in\{A,B\}}T_\beta\partial_{T}\Big) u
\\
+{\sum}_{T\in\mathcal{T}}
\Big(\frac{g_S^{T\uL}}{2}\partial_{\uL}\ln{\big(|g|/|g_S|\big)}\partial_{T}
+\frac{g_S^{T\beta}}{2}\partial_T\ln{\big(|g|/|g_S|\big)}\,\partial_{\beta}\Big)u,
\label{T1}
\end{multline}
and
\begin{equation}
T^0_g(h)u=-\frac{2M}{r^2}\partial_{r^*}u+\frac{2L_\alpha h^{\alpha i}\omega_i}{g_S(L,\uL)^3}\frac{2M}{r^2} \partial_{\uL} u.
\end{equation}
Here $\mathcal{T}=\{L,A,B\}$, $|g|/|g_S|=1+O(h)$
and $g^{UV}_S$ are the coefficients in the expansion of $g_S$ in the null frame
$g_S^{\alpha\beta}=g_S^{\uL \uL} {\uL}^\alpha{\uL}^\beta +\sum_{T\in\mathcal{T}}g_S^{\uL T}{\uL}^\alpha T^\beta +g_S^{\alpha T}T^\beta $.
\end{lemma}
\begin{proof}[Proof of Lemma \ref{waveopnullfraame}]
Expanding in a null frame using \eqref{nullframeexpansion} we have
\begin{equation}
h^{\alpha\beta} \partial_\alpha\partial_\beta
=h^{\uL\uL} \uL^\alpha\uL^\beta \partial_\alpha\partial_\beta+
\sum_{T\in\mathcal{T}}( h^{T\uL}\uL^{\beta}+h^{T\beta} )\partial_T\partial_{\beta} u.
\end{equation}
If we also use that
\[
\uL^\alpha\uL^\beta \partial_\alpha\partial_\beta
=\partial_{\uL}^2-(\partial_{\uL} \uL^\beta )\partial_\beta.
\]
we get
\[
 h^{\alpha\beta} \partial_\alpha\partial_\beta
=T^2_g(h)u -\frac{2M}{r^2}\partial_{r^*},
\label{T2}
\]
To prove \eqref{T1} we first rewrite
\begin{equation*}
 \frac{1}{\sqrt{|g|}}\partial_\alpha \big(\sqrt{|g|}g^{\alpha\beta}\big) - \frac{1}{\sqrt{|g_S|}}\partial_\alpha \big(\sqrt{|g_S|}g_S^{\alpha\beta}\big)
=\frac{1}{\sqrt{|g|}}\partial_\alpha \big(\sqrt{|g|}h^{\alpha\beta}\big) +\frac{1}{2}\partial_\alpha \ln{\big(|g|/|g_S|\big)}\,g_S^{\alpha\beta},
\end{equation*}
where $|g|/|g_S|=1+O(h)$.
Since $g_S^{\uL\uL}=0$, we obtain
\begin{equation*}
\partial_\alpha \ln{\big(|g|/|g_S|\big)}\,g_S^{\alpha\beta}\partial_\beta u
=\sum_{T\in\mathcal{T}}\Big(\partial_T\ln{\big(|g|/|g_S|\big)}\,g_S^{T\alpha}\partial_{\alpha}
+\partial_{\uL}\ln{\big(|g|/|g_S|\big)}\,g_S^{T\uL}\partial_{T}\Big)u
\end{equation*}
which is the last line of \eqref{T1}. Expanding the usual derivatives in a null frame we get
 \begin{equation}
 \partial_\alpha=\frac{1}{g_S(L,\uL)}L_\alpha\,\uL +\frac{1}{g_S(L,\uL)}\uL_\alpha\, L
 +A_\alpha \,A+B_\alpha\, B ,\label{partialinframe}
 \end{equation}
 which follows from the discussion before \eqref{lowernull}.
 Hence
\begin{equation*}
\partial_\alpha \big(\sqrt{|g|}h^{\alpha\beta}\big)
=\frac{1}{g_S(\uL,L)}\big(L_\alpha\partial_{\uL }\big(\sqrt{|g|}h^{\alpha\beta}\big)
+\uL_\alpha\partial_{L }\big(\sqrt{|g|}h^{\alpha\beta}\big)\big)
+\sum_{T\in\{A,B\}}T_\alpha\partial_{T}\big(\sqrt{|g|}h^{\alpha\beta}\big).
\end{equation*}
We further expand
\begin{equation*}
\partial_{\uL} \big(\sqrt{|g|}h^{\alpha\beta}\big)\partial_\beta
=\frac{\partial_{\uL }\big(\sqrt{|g|}h^{\alpha\beta}\big)}{g_S(\uL,L)}\big(L_\beta\partial_{\uL}
+\uL_\beta\partial_{L }\big)
+\sum_{T\in\{A,B\}}\partial_{\uL} \big(\sqrt{|g|}h^{\alpha\beta}\big) T_\beta\partial_{T}.
\end{equation*}
It follows that
\begin{equation}
 \frac{1}{\sqrt{|g|}}\partial_\alpha \big(\sqrt{|g|}g^{\alpha\beta}\big) - \frac{1}{\sqrt{|g_S|}}\partial_\alpha \big(\sqrt{|g_S|}g_S^{\alpha\beta}\big)
 =\frac{1}{\sqrt{|g|}}\frac{L_\alpha L_\beta}{g_S(\uL,L)^2}\partial_{\uL} \big(\sqrt{|g|}h^{\alpha\beta}\big)\partial_{\uL} u
 +T^1_g(h) u.
\end{equation}
Moreover
\begin{multline}
\frac{1}{\sqrt{|g|}}\frac{L_\alpha L_\beta}{g_S(\uL,L)^2}\partial_{\uL} \big(\sqrt{|g|}h^{\alpha\beta}\big)-
\frac{1}{\sqrt{|g|}}\partial_{\uL} \big(\sqrt{|g|}h^{\uL\uL}\big)=-h^{\alpha\beta}\partial_{\uL}\frac{L_\alpha L_\beta}{g_S(\uL,L)^2} \\
=-2h^{\alpha\beta}L_\alpha\frac{ \partial_{\uL} L_\beta}{g_S(\uL,L)^2}
+2h^{\alpha\beta}L_\alpha L_\beta\frac{ \partial_{\uL}g_S(\uL,L) }{g_S(\uL,L)^3}
=\frac{2L_\alpha h^{\alpha i}\omega_i}{g_S(L,\uL)^3}\frac{2M}{r^2}.
\end{multline}
Summing everything up proves \eqref{null}.
\end{proof}

\begin{proof}[Proof of Theorem \ref{lowerLE}]

 We use induction on a multiindex $\Lambda$, in the spirit of Theorem~\ref{higherLE}. The case $\Lambda=0$ is given by Theorem~\ref{LESbded}.
  We first commute with time derivatives, as in the proof of Theorem~\ref{higherLE}. Recall the definitions of $E_{\tN}$ and $F_{\tN}$ from \eqref{enuN} and \eqref{boxuN}.
We note that
\begin{equation}\label{FtNps}
 \int_{\M_{[t_0,t]}^{ps}} \frac{\epsilon}{\tau}|\partial \partial_t^{\tN} u|^2 dV_g  \lesssim \|\partial \partial_{\leq \tN+1} u(t_0)\|_{L^2}^2,
\end{equation}
which follows from using Theorem~\ref{higherLE} as in the in the proof of Theorem~\ref{LESbded} above.

 We will next modify the estimate \eqref{FN-est} to
\begin{equation}\label{FtN-est}
\int_{\M_{[t_0,t]}} |F_{\tN}| |\partial_{\leq 1} \partial_t^{\tN} u| dV_g \leq \delta_0 E_{\tN}(t) + C(\delta_0) \|\partial \partial_{\leq N}u(t_0)\|_{L^2}^2
\end{equation}
for a suitably small, $\epsilon$-independent constant $\delta_0$ that allows the first term to be absorbed on the left hand side. The conclusion of the theorem,  \eqref{lowLES}, for all derivatives, i.e.  with $u_{\leq \tN}$ replaced by $\partial_{\leq \tN}u$, will then follow from first  using  \eqref{FtNps} and \eqref{FtN-est} in \eqref{partialtcommutN} and then using Gronwall's inequality and Lemma~\ref{elliptic}.

  Let us now prove \eqref{FtN-est}. In the region $r_e \leq r\leq R_1$ we have by \eqref{tderivglobal}
\[
\int_{\M_{[t_0,t]}^{[r_e, R_1)}} |F_{\tN}| |\partial_{\leq 1} \partial_t^{\tN} u| dV_g \lesssim \int_{\M_{[t_0,t]}^{[r_e, R_1)}} \frac{\epsilon}{\la \tau\ra} |\partial_{\leq\tN +1} u|^2 dV_g
\]
 and so
 \begin{equation}\label{bdcpttest}
 \int_{\M_{[t_0,t]}^{[r_e, R_1)}} |F_{\tN}| |\partial_{\leq 1} \partial_t^N u| dV_g \lesssim \|\partial_{\leq\tN +1} u(t_0)\|_{L^2}^2
 \end{equation}
 by decomposing the time interval dyadically and using Theorem~\ref{higherLE} as in the proof
 of Theorem~\ref{LESbded}.

 In the region $R_1^* \leq \rs\leq\frac{t}2$ we have by \eqref{t-derivest} and \eqref{Wlotfar}
$
|\partial_t h_{\Lambda}| + |\partial_t W_{\Lambda}| \lesssim \epsilon r^{-1-\delta}
$
 and so
\begin{equation}\label{bdinttest}
\begin{split}
& \int_{\M_{[t_0,t]}^{[R_1, \frac{\tau}2)}} |F_{\tN}| (|\partial\partial_t^{\tN} u| + r^{-1} |\partial_t^{\tN} u|) dV_g  \lesssim \int_{\M_{[t_0,t]}^{[R_1, \frac{\tau}2)}} \epsilon r^{-1-\delta} (|\partial \partial_t^{\tN} u|^2 + |\partial \partial_{\leq \tN-1} u|^2) dV_g \\ & \leq \delta_0 E_{\tN}(t) + C(\delta_0) \|\partial_{\leq \tN-1} u\|_{LE_S^1[t_0, t]}^2 \leq \delta_0 E_{\tN}(t) + C(\delta_0)\|\partial \partial_{\leq\tN} u(t_0)\|_{L^2}^2
 \end{split}
 \end{equation}
by the induction hypothesis.

  Near the cone we will use the additional hypothesis \eqref{badpart2}. We clearly have for any function $w$ that
 \begin{equation}\label{tgimprov}
 \overline{\partial} w \in S^Z \Bigl(\frac{\la t-\rs\ra}{t}\Bigr)\partial w + S^Z \Bigl(\frac{1}{t}\Bigr) Zw
 \end{equation}

We will now commute time  derivatives through the equation expressed in a nullframe using Lemma
\ref{waveopnullfraame}.
Note that $T^i_g(h)u $, $i=1,2$, contain at least one tangential derivative on one of the factors $h$ and $u$:
\[
T^2_g(h)u\approx h\overline{\partial}\partial u,\qquad T^1_g(h)u\approx\partial h \overline{\partial} u + \overline{\partial} h \partial u,\qquad \overline{\partial}\in \mathcal{T},
\]
whereas $T^0_g(h)u$ decay better
$
T^0_g(h)u\sim {M}{r^{-2}}\partial u.
$

 We now commute with $\partial_{t}^\tN$. Due to \eqref{badpart2}, \eqref{metricvsu} and \eqref{u-decay} we have
\[
\Bigl| \big[\partial_{t}^\tN, h^{\uL\uL}\partial_{\uL}^2
+ \frac{1}{\sqrt{|g|}}\partial_{\uL}(\sqrt{|g|}h^{\uL\uL})\partial_{\uL}\big] u \Bigr| \lesssim t^{-1-\delta} |\partial u_{\leq \tN +1}|
\]
On the other hand, since $[\partial_t, \overline{\partial}] = 0$, we also have due to \eqref{metricdecay}, \eqref{metricvsu}, \eqref{u-decay} and \eqref{tgimprov}
 \[
\Bigl| \big[\partial_{t}^\tN, T^i_g(h)\big]u\Bigr| \lesssim t^{-1-\delta} |\partial u_{\leq \tN +3}| + t^{-2}\la t-\rs\ra^{-\delta} |u_{\leq \tN +3}|
 \]
 Therefore near the cone we have by Theorem~\ref{higherLE} and \eqref{LRHardy}
 \begin{equation}\label{bdconetest}\begin{split}
\int_{\M_{[t_0,t]}^{[\frac{\tau}2, \infty)}} |F_{\tN}| (|\partial\partial_t^{\tN} u| &+ r^{-1} |\partial_t^{\tN} u|) dV_g  \lesssim \int_{\M_{[t_0,t]}^{[\frac{\tau}2, \infty)}} \tau^{-1-\delta} |\partial u_{\leq \tN +3}|^2 +\\& \tau^{-2}\la t-\rs\ra^{-\delta} |u_{\leq \tN +3}|^2 dV_g  \lesssim \|\partial \partial_{\leq \tN +3} u(t_0)\|_{L^2}^2
\end{split} \end{equation}

 \eqref{FtN-est} now follows from \eqref{bdcpttest}, \eqref{bdinttest} and \eqref{bdconetest}.

 We now proceed by induction on the number of vector fields $K$ as in Theorem~\ref{higherLE}. We assume that Theorem~\ref{lowerLE} holds for all $\Lambda\prec\mathcal{P}$, and we will prove it for all $\Lambda\in\mathcal{P}$.

 As before, after commuting in the equation for $u_{\Lambda}$, $\Lambda\in\mathcal{P}^l$, we get that
\[
\Box_g u_{\Lambda} = F_{\Lambda} + L_{\Lambda} + (\Box_g u)_{\Lambda}
\]
where
\[
F_{\Lambda} = \sum_{\substack{\Lambda_1+\Lambda_2 = \Lambda,\,|\Lambda_1|>0}} h_{\Lambda_1}^{\alpha\beta} \partial_{\alpha}\partial_{\beta}u_{\Lambda_2} + \sum_{\Lambda_1+\Lambda_2 = \Lambda} W_{\Lambda_1}^{\alpha} \partial_{\alpha}u_{\Lambda_2} + [\Box_{g_S}, Z^{\Lambda}] u \]
and $L_{\Lambda}$ consists of lower order terms satisfying the bound
\[
|L_{\Lambda}| \lesssim \sum_{m+n < |\Lambda|} (|h_{\leq m}| |\partial^{2} u_{\leq n}| + |W_{\leq m}^{\alpha}| |\partial_{\alpha}u_{\leq n}|)
\]

 We now apply \eqref{LES} to $u_{\Lambda}$. It follows as in the proof of Theorem~\ref{higherLE} that
\[
 \int_{\M_{[t_0,t_1]}^{ps}} \frac{\epsilon}{t}|\partial u_{\Lambda}|^2 dV_g \lesssim \|\partial_{\leq 1} \partial u_{\leq \tN}(t_0)\|_{L^2}^2
\]

It will suffice to modify \eqref{FLambdainhom} to
\begin{equation}\label{bdFLambdainhom}
\begin{split}
\int_{\M_{[t_0,t]}} (|F_{\Lambda}| + |L_{\Lambda}|) (|\partial u_{\Lambda}| + r^{-1} |u_{\Lambda}|) dV_g & \leq  \delta_0 E_{\mathcal{P}^l}(t) + C(\delta_0) (E_{\mathcal{P}^{l-1}}(t) + \|\partial u_{\leq \tN+3}(t_0)\|_{L^2}^2)
\end{split}\end{equation}

 We will only show how to control the $F_{\Lambda}$ term in \eqref{bdFLambdainhom}, as the $L_{\Lambda}$ is similar.

 In the compact region $r_e\leq r\leq R_1$ we have as in the proof of Theorem~\ref{higherLE}
\begin{equation}\label{bdcpt1hgh} \begin{split}
 \int_{\M_{[t_0,t]}^{[r_e, R_1)}} & |F_{\Lambda}||\partial_{\leq 1} u_{\Lambda}| dV_g  \lesssim \sum_{T dyadic} \int_{\M_{[T,2T]}^{[r_e, R_1)}} \la \tau\ra^{-1/2} (|\partial u_{\mathcal{P}^l}(t)|^2 + |u_{\mathcal{P}^l}(t)|^2) + |u_{\prec\mathcal{P}}(t)|^2  dV_g \\ & \lesssim \sum_{T dyadic} T^{-1/2} T^{C\epsilon} \|\partial u_{\leq |\Lambda|+1}(t_0)\|_{L^2}^2 + E_{\prec\mathcal{P}}(t) \lesssim \|\partial u_{\leq |\Lambda|+1}(t_0)\|_{L^2}^2
\end{split}
\end{equation}
by the induction hypothesis.

In the intermediate region $R_1^* \leq \rs\leq\frac{t}2$ we have by \eqref{improvetwoderivs2}, \eqref{Wlotfar}, and \eqref{ptdecaydu2} that
\[
|\partial^2 u_{\leq \tN}|\lesssim \epsilon r^{-\frac52+\delta}, \qquad |W_{\leq \tN}| \lesssim \epsilon r^{-1-\delta}, \qquad |[\Box_{g_S}, Z^{\Lambda}] u| \lesssim r^{-2+} (|\partial u_{\tilde\Lambda}| + |\partial u_{\prec\mathcal{P}}|)
\]

 This yields
\[
|F_{\Lambda}| \leq C\epsilon r^{-\frac52+\delta}|u_{\leq \Lambda}| + \delta_0 r^{-1-\delta}|\partial u_{\leq \Lambda}| + Cr^{-2+} (|\partial u_{\tilde\Lambda}| + |\partial u_{\prec\mathcal{P}}|)
\]
and thus
\begin{equation}\label{bdinter1hgh} \begin{split}
\int_{\M_{[t_0,t]}^{[R_1, \frac{\tau}2)}} |F_{\Lambda}|(|\partial u_{\Lambda}| + r^{-1} |u_{\Lambda}|) dV_g & \leq \delta_0 E_{\mathcal{P}^l}(t) + C(\delta_0) (E_{\mathcal{P}^{l-1}}(t) + E_{\prec\mathcal{P}}(t))
\end{split} \end{equation}
which suffices by the induction hypothesis.

  Finally, near the cone we write the operator in null coordinates as in \eqref{null}. For the part without tangential derivatives we note that, by using \eqref{badpart2}, \eqref{ptdecayu2} and \eqref{ptdecaydu2}, we have
$
|h^{\lL\lL}_{\leq \tN}|\lesssim t^{-1-\delta/2}
$
and thus
\begin{equation}\label{bdcone1hgh}
\Bigl| \big[Z^{\Lambda}, h^{\uL\uL}\partial_{\uL}^2 + \frac{1}{\sqrt{|g|}}\partial_{\uL}(\sqrt{|g|}h^{\uL\uL})\partial_{\uL}\big] u \Bigr| \lesssim t^{-1-\delta} |\partial u_{\leq \tN +1}|
\end{equation}

For the part containing tangential derivatives, we use the fact that commuting with $\Omega$ and $S$ preserves the null structure. Indeed, a quick computation yields
 \[
 [\partial_t, \overline{\partial}] = 0, \quad [\Omega,  \overline{\partial}] \in \overline{\mathcal{T}}, \quad [S,  \overline{\partial}] \in \overline{\mathcal{T}},\qquad\text{for }
 \overline{\partial}\in \mathcal{T}.
\]
Here $\overline{\mathcal{T}}=\{\sum_{T\in \mathcal{T}}c^T \partial_T\}$, where
$c^T(\omega)$ are homogeneous functions of degree $0$. In fact,
 \[
  [\Omega,  T] \in \overline{\mathcal{T}}, \quad [S,  T] \in \overline{\mathcal{T}},\qquad\text{for }
 T\in \overline{\mathcal{T}}.
\]

It is not hard to see that
\[
\Bigl| \big[Z^{\Lambda}, T^i_g(h)\big]u\Bigr|\lesssim |h_{\leq\tN}||\overline{\partial}\partial_{\leq\tN} u| + |\partial h_{\leq\tN}||\overline{\partial} u_{\leq\tN}| + |\overline{\partial} h_{\leq\tN}||\partial u_{\leq\tN}|
\]

Using \eqref{ptdecayu2}, \eqref{tgimprov}, and \eqref{improvetwoderivs2}, we get
\[\begin{split}
& |h_{\leq\tN}||\overline{\partial}\partial_{\leq\tN} u|  \lesssim \frac{\epsilon \la t\ra^{C\epsilon}\la t-\rs\ra^{\frac12}}{\la t\ra} |\overline{\partial}\partial u_{\leq \tN}| \lesssim \frac{\epsilon \la t\ra^{C\epsilon}\la t-\rs\ra^{\frac12}}{\la t\ra} \Bigl(\frac{\la t-\rs\ra}{t}|\partial^2 u_{\leq \tN}| + \frac{1}{t}|\partial u_{\leq \tN+3}|\Bigr) \\ & \lesssim \frac{\epsilon \la t\ra^{C\epsilon}\la t-\rs\ra^{\frac12}}{\la t\ra} \Bigl(\frac{1}{t}|\partial u_{\leq \tN+3}| + \frac{1}{t\la t-\rs\ra} |u_{\leq\tN}|\Bigr) \lesssim t^{-1-\delta}|\partial u_{\leq \tN+3}| + t^{-2+\delta} \la t-\rs\ra^{-1/2}|u_{\leq\tN}|
\end{split}
\]
We can also bound, using \eqref{metricvsu}, \eqref{ptdecaydu2} and \eqref{tgimprov}
\[\begin{split}
& |\partial h_{\leq\tN}||\overline{\partial} u_{\leq\tN}| \lesssim \Bigl(|\partial u_{\leq\tN}| + \frac{1}{\la t-\rs\ra^{1/2+\delta}} |u_{\leq\tN}|\Bigr)\Bigl(\frac{\la t-\rs\ra}t |\partial u_{\leq\tN}| + \frac{1}{t}|u_{\leq\tN+3}|\Bigr) \\ & \lesssim \frac{\la t\ra^{C\epsilon}}{t\la t-\rs\ra^{\delta}}\Bigl(\frac{\la t-\rs\ra}t |\partial u_{\leq\tN}| + \frac{1}{t}|u_{\leq\tN+3}|\Bigr) \lesssim \frac{1}{t^{1+\delta/2}}|\partial u_{\leq\tN}| + \frac{1}{t^{2-\delta/2}\la t-\rs\ra^{\delta}}|u_{\leq \tN+3}|
\end{split}\]
 A similar computation yields the same bound for $|\overline{\partial} h_{\leq\tN}||\partial u_{\leq\tN}|$.
 We thus have near the cone
\begin{equation}\label{bdcone2hgh}\begin{split}
\int_{\M_{[t_0,t]}^{[\frac{\tau}2, \infty)}} |F_{\Lambda}| \big(|\partial u_{\Lambda}| + \frac{1}{r} |u_{\Lambda}|\big) dV_g \lesssim \int_{\M_{[t_0,t]}^{[\frac{\tau}2, \infty)}} \frac{1}{\tau^{1+\delta/2}} |\partial u_{\leq\tN+3}|^2 +\frac{1}{ \tau^{2+\delta}} |u_{\leq \tN+3}|^2 dV_g \lesssim \|\partial u_{\leq \tN+3}(t_0)\|_{L^2}^2
\end{split}\end{equation}
by Theorem~\ref{higherLE} and Hardy's inequality.

 Finally,  for the last term we use \eqref{S-farhgh} and estimate like in \eqref{lincomm} and \eqref{lincommlot}
\begin{equation}\label{bdcone3hgh}
\int_{\M_{[t_0,t]}} |[\Box_{g_S}, Z^{\Lambda}] u| (|\partial u_{\Lambda}| + r^{-1} |u_{\Lambda}|) dV_g \leq \delta_0 E_{\mathcal{P}^l}(t) + C(\delta_0)(E_{\mathcal{P}^{l-1}}(t)+E_{\prec\mathcal{P}}(t))
\end{equation}

 The estimate \eqref{bdFLambdainhom} now follows from \eqref{bdcpt1hgh}, \eqref{bdinter1hgh}, \eqref{bdcone1hgh}, \eqref{bdcone2hgh}, and \eqref{bdcone3hgh}
\end{proof}

\section*{Acknowledgments}
H.L. is supported in part by NSF grant DMS--1237212. M.T. is supported in part by the NSF grant DMS--1636435.

\bigskip

\end{document}